\documentclass[12pt]{amsart}
\usepackage{a4wide,enumerate,xcolor}
\usepackage{amsmath,amssymb}
\usepackage{comment}
\allowdisplaybreaks

\usepackage{graphicx}
\usepackage{float}
\usepackage{multirow}
\usepackage{multicol}

\usepackage{algorithm}
\usepackage{algorithmic}

\let\pa\partial
\let\na\nabla
\let\eps\varepsilon
\newcommand{\N}{{\mathbb N}}
\newcommand{\R}{{\mathbb R}}
\newcommand{\diver}{\operatorname{div}}

\newcommand{\T}{{\mathbb T}}
\def\dd{\mathrm{d}}
\def\m{\mathrm{m}}
\def\D{\mathrm{D}}

\newtheorem{theorem}{Theorem}
\newtheorem{lemma}[theorem]{Lemma}

\newtheorem{remark}[theorem]{Remark}

\newtheorem{example}{Example}

\graphicspath{{figures/}}

\begin{document}

\title[A generalized Scharfetter--Gummel scheme]{
A generalized Scharfetter--Gummel scheme \\ for nonlocal cross-diffusion systems}

\author[A. J\"ungel]{Ansgar J\"ungel}
\address{Institute of Analysis and Scientific Computing, TU Wien, Wiedner Hauptstra\ss e 8--10, 1040 Wien, Austria}
\email{juengel@tuwien.ac.at} 

\author[P. Li]{Panchi Li}
\address{Department of Mathematics, The University of Hong Kong, Hong Kong, China}
\email{lipch@hku.hk} 

\author[Z. Sun]{Zhiwei Sun}
\address{Institute of Analysis and Scientific Computing, TU Wien, Wiedner Hauptstra\ss e 8--10, 1040 Wien, Austria}
\email{zhiwei.sun@tuwien.ac.at}

\date{\today}

\thanks{The first author acknowledges partial support from   
the Austrian Science Fund (FWF), grant 10.55776/F65, and from the Austrian Federal Ministry for Women, Science and Research and implemented by \"OAD, project MultHeFlo. This work has received funding from the European Research Council (ERC) under the European Union's Horizon 2020 research and innovation programme, ERC Advanced Grant NEUROMORPH, no.~101018153. For open-access purposes, the authors have applied a CC BY public copyright license to any author-accepted manuscript version arising from this submission.}

\begin{abstract}
An implicit Euler finite-volume scheme for a nonlocal cross-diffusion system on the multidimensional torus is analyzed. The equations describe the dynamics of population species with repulsive or attractive interactions. The numerical scheme is based on a generalized Scharfetter--Gummel discretization of the nonlocal flux term. For merely integrable kernel functions, the scheme preserves the positivity, total mass, and entropy structure. The existence of a discrete solution and its convergence to a solution to the continuous problem, as the mesh size tends to zero, are shown. A key difficulty is the degeneracy of the generalized Bernoulli function in the Scharfetter--Gummel approximation. This issue is overcome by proving a uniform estimate for the discrete Fisher information, which requires both the Boltzmann and Rao entropy inequalities. Numerical simulations illustrate the features of the scheme in one and two space dimensions.
\end{abstract}

 

\maketitle


\section{Introduction}

Nonlocal cross-diffusion systems arise in the modeling of interacting population species, where the dynamics is driven not only by the local density gradients but also by nonlocal interaction terms. These systems capture a broad range of collective behaviors, including repulsive and attractive self- or cross-interactions. In this paper, we design and analyze a structure-preserving implicit Euler finite-volume discretization of the following nonlocal problem for the population densities $u_i$ of the $i$th species:
\begin{align}
  & \pa_t u_i - \kappa\Delta u_i = \diver(u_i\na p_i(u)), \label{1.eq} \\
  & p_i(u)(x) = \sum_{j=1}^n\int_{\T^d}W_{ij}(x-y)u_j(y)\dd y
  \quad\mbox{in }\T^d,\ t>0, \label{1.p} \\
  & u_i(0) = u_i^0\quad\mbox{in }\T^d, \label{1.ic}
\end{align}
where $\T^d$ is the $d$-dimensional torus, $\kappa>0$ is a diffusion coefficient, $W_{ij}:\T^d\to\R$ are the interaction kernels (extended periodically to $\R^d$), and $u=(u_1,\ldots,u_n)$ is the solution vector. The functions $p_i=p_i(u)$ can be interpreted as potentials depending on the densities in a nonlocal way. 

When the kernels $W_{ij}$ are positive semidefinite (in the sense of \eqref{1.posdef} below), equations \eqref{1.eq}--\eqref{1.p} capture the dynamics of populations exhibiting repulsive interactions. The equations have been derived from interacting particle systems in a mean-field-type limit in \cite{CDJ19}. A global existence analysis can be found in \cite{JPZ22}. When the kernels $W_{ij}$ converge to the Dirac delta distribution times a factor $a_{ij}\in\R$, the nonlocal equations converge to the local equations \eqref{1.eq} with $p_i(u)=\sum_{j=1}^n a_{ij}u_j$ (see \cite[Theorem 5]{JPZ22} and \cite[Theorem 8]{DHPP24}). The local system is a generalization of the population model first suggested by Busenberg and Travis \cite{BuTr83}. In the general case ($W_{ij}$ being not positive semidefinite), the interaction forces may be attractive, including multicell adhesion effects \cite{CMSW25}. The existence of global solutions to the $n$-species aggregation-diffusion system was proved in \cite{CSS24} assuming small total mass or kernels with bounded variation.

The motivation of our work is to develop a stable numerical scheme that remains effective for all values of $\kappa>0$, particularly in the drift-dominated regime where $\kappa$ is small. Additionally, we aim to establish a robust numerical analysis framework tailored for non-differentiable kernels. Standard discretizations often become unstable in low-diffusion scenarios. To overcome this challenge, we adopt a Scharfetter--Gummel-type discretization, which ensures stability even in these scenarios. Our contribution extends existing approaches in the literature to nonlocal cross-diffusion systems and a broader class of Bernoulli-type functions. Notably, we provide
the first numerical analysis for a Scharfetter--Gummel scheme with non-differentiable potential, even in the single-species case.

\subsection{Entropy structure}

We aim to devise a numerical scheme that preserves the structure of equations \eqref{1.eq}--\eqref{1.p}, namely positivity of the densities, total mass, and entropy production. To explain the last point, we introduce the Boltzmann and Rao entropies
\begin{align*}
  \mathcal{H}_B(u) &= \sum_{i=1}^n\int_{\T^d}u_i(\log u_i-1)\dd x, \\
  \mathcal{H}_R(u) &= \frac12\sum_{i,j=1}^n\int_{\T^d}
  W_{ij}(x-y)u_i(x)u_j(y)\dd x\dd y.
\end{align*}
A formal computation shows that the entropy equalities
\begin{align}
  \frac{\dd\mathcal{H}_B}{\dd t} + 4\kappa\sum_{i=1}^n\int_{\T^d}
  |\na\sqrt{u_i}|^2\dd x &= -\sum_{i=1}^n\int_{\T^d}\na u_i
  \cdot\na p_i\dd x, \label{1.HB} \\
  \frac{\dd\mathcal{H}_R}{\dd t} + \sum_{i=1}^n\int_{\T^d}
  u_i|\na p_i|^2\dd x &= -\kappa\sum_{i=1}^n\int_{\T^d}\na u_i
  \cdot\na p_i\dd x \label{1.HR}
\end{align}
hold true if the kernels are symmetric (see Hypothesis (H4) below). The last integral on the left-hand side of \eqref{1.HB} is called the Fisher information. If the kernels are positive semidefinite in the sense
\begin{align}\label{1.posdef}
  \sum_{i,j=1}^n\int_{\T^d}\int_{\T^d}W_{ij}(x-y)v_i(x)v_j(y)\dd x
  \dd y\ge 0
\end{align}
for functions $v_i\in L^2(\T^d)$, the right-hand sides of both \eqref{1.HB} and \eqref{1.HR} are nonpositive, which yields estimates for $\na u_i$. This argument cannot be used in the general case (including attractive interactions). The gradient-flow structure provides an alternative equality,
\begin{align}\label{1.HBR}
  \frac{\dd}{\dd t}(\kappa\mathcal{H}_B + \mathcal{H}_R)
  + \sum_{i=1}^n\int_{\T^d}u_i|\na(\kappa\log u_i+p_i)|^2\dd x = 0,
\end{align}
and hence a priori bounds for the entropies. As in \cite{ScSe22}, a bound for $\na\sqrt{u_i}$ can be derived when the kernel $W_{ij}$ is $C^1$-regular. However, this becomes infeasible for non-differentiable kernels $W_{ij}$. The work \cite{CSS24} overcomes this issue by assuming small total masses. Indeed, by H\"older's and Young's convolution inequalities \cite[(4.11)]{CSS24}, we can estimate
\begin{align*}
  -\sum_{i=1}^n\int_{\T^d}\na u_i\cdot\na p_i\dd x
  &\le \sum_{i,j=1}^n\|\na u_i\|_{L^1(\T^d)}
  \|W_{ij}*\na u_j\|_{L^\infty(\T^d)} \\
  &\le 4\max_{j=1,\ldots,n}\sum_{i=1}^n\|W_{ij}\|_{L^\infty(\T^d)}
  \|u_i\|_{L^1(\T^d)}\|\na\sqrt{u_i}\|_{L^1(\T^d)}^2.
\end{align*}
Thus, if the total mass $\max_j\|u_j\|_{L^1(\T^d)}$ is sufficiently small, we can absorb the right-hand side by the Fisher information in \eqref{1.HB}. The advantage of this argument is that the differentiability of the kernels is not required. Our goal is to ``translate'' these properties to the discrete level in the framework of the Scharfetter--Gummel discretization. 


\subsection{State of the art and main ideas}

Several works address the design and analysis of numerical schemes for nonlocal cross-diffusion systems. The paper \cite{CHS18} explores a positivity-preserving, one-dimensional finite-volume scheme for equations \eqref{1.eq}--\eqref{1.p} with two species and additional local cross-diffusion terms, focusing on segregated steady states. A convergence analysis is subsequently established in \cite{CFS20}.  For systems with an arbitrary number of species, structure-preserving finite-volume schemes were further investigated in \cite{JPZ24}. We also mention the work \cite{HeZu22}, which is concerned with numerical approximations of nonlocal Shigesada--Kawasaki--Teramoto population systems.

The Scharfetter--Gummel discretization was first suggested in \cite{ScGu69} for the semiconductor drift-diffusion equations. Early extensions to multi-dimensional settings were developed in \cite{FaGa91}, and adapted to finite-volume frameworks in \cite{Bes12,ChDr11}. Variants of the Scharfetter--Gummel scheme for problems with nonlocal self-repulsive interactions were numerically compared in \cite{CCFG19}. These schemes have been extended to nonlinear diffusion problems \cite{Bes12,Jue95}. 
More recently, \cite{ScSe22} introduced a Scharfetter--Gummel scheme tailored for general nonlocal aggregation-diffusion equations, focusing on $C^1$-regular interaction kernels. From a structural perspective, the classical Scharfetter--Gummel scheme was identified in \cite{HST24} as a generalized gradient flow associated to a discrete analog of the energy $\kappa\mathcal{H}_B+\mathcal{H}_R$. While most works focus on two-point approximations, a multi-point discrete duality finite-volume method was analyzed in \cite{Que22}. To the best of our knowledge, the only extension of the Scharfetter--Gummel scheme to (local) cross-diffusion systems is presented in \cite{CHM24}, where a convergent mixed square-root approximation Scharfetter--Gummel scheme for a (local) Poisson--Planck--Nernst model was proposed. 

This work presents the first numerical study of nonlocal cross-diffusion systems using a Scharfetter--Gummel scheme. A key novelty is the treatment of non-differentiable interaction kernels, which represents a significant advancement even for single-species nonlocal aggregation-diffusion models. Our approach addresses three main challenges:%
\begin{itemize}
\item The maximum principle is generally not applicable due to the nonlocal interactions, setting them apart from local semiconductor models.
\item The Scharfetter--Gummel scheme does not inherently preserve the Boltzmann entropy equality \eqref{1.HB} at the discrete level.
\item The Scharfetter--Gummel scheme exhibits an exponential dependence on the potential difference $\D_{K,\sigma}p^k$ (see below), which creates difficulties in cases involving non-differentiable kernels.
\end{itemize}
To overcome these challenges, we introduce a novel reformulation of the Scharfetter--Gummel scheme that depends only linearly on $\D_{K,\sigma}p^k$. Moreover, we prove that this reformulation admits a discrete Boltzmann entropy inequality similar to \eqref{1.HB}.

To explain our main ideas, we consider the scalar equation $\pa_t u + \diver\mathcal{F}=0$ with $\mathcal{F}=-\kappa\na u-u\na p$. We need some notation. Let $\mathcal{T}$ be a triangulation of the domain $\T^d$, consisting of control volumes $K$ and let $t_k=k\Delta t$ be a time step with step size $\Delta t>0$. Let $(u_K^k)_{K\in\mathcal{T}}$ and $(p_K^k)_{K\in\mathcal{T}}$ be piecewise constant approximations of $u(t_k)$ and $p(t_k)$ on the cell $K$, respectively, and let $\sigma=K|L$ be an edge or face between two control volumes $K$ and $L$. The difference is denoted by $\D_{K,\sigma}u^k:=u^k_L-u^k_K$. 

The classical Scharfetter--Gummel flux is defined by
\begin{align*}
  \mathcal{F}_{K,\sigma}^{(1)}[u^k,p^k] = \tau_\sigma\big(B_\kappa(\D_{K,\sigma}p^k)u^k_K
  - B_\kappa(-\D_{K,\sigma}p^k)u^k_L\big),
\end{align*}
where $\tau_\sigma>0$ is the so-called transmissibility coefficient (see \eqref{2.tau} below), $B_\kappa(s)=\kappa B(s/\kappa)$, and $B(s)$ is the Bernoulli function $B(s)=s/(\exp s-1)$ for $s\neq 0$ and $B(0)=1$. (We will use generalized Bernoulli functions in this paper.) The flux can be written equivalently as
\begin{align}\label{1.SG1}
  & \mathcal{F}_{K,\sigma}^{(1)}[u^k,p^k] = -\tau_\sigma
  \bigg(\widetilde{B}_\kappa(\D_{K,\sigma}p^k)\D_{K,\sigma}u^k 
  + \frac{u^k_K+u^k_L}{2}\D_{K,\sigma}p^k\bigg), \quad\mbox{where} \\
  & \widetilde{B}_\kappa(\D_{K,\sigma}p^k)
  = \frac12\big(B_\kappa(\D_{K,\sigma}p^k) 
  + B_\kappa(-\D_{K,\sigma}p^k)\big), \nonumber 
\end{align}
which reveals a mean-value drift and separates the diffusion and drift parts \cite[(2.8)]{ScSe22}. A discrete analog of the flux formulation $\mathcal{F}=-u(\kappa\na\log u+\na p)$ was suggested in \cite[(3.7)]{CCFG19}:
\begin{align}\label{1.SG2}
  & \mathcal{F}_{K,\sigma}^{(2)}[u^k,p^k] = -\tau_\sigma
  (a_1u^k_K+a_2u^k_L)\big(\kappa\D_{K,\sigma}\log u^k 
  + \D_{K,\sigma}p^k\big),
\end{align}
where
\begin{align*}
  & a_1 = \frac{B(z)-B(y)}{y-z}, \quad 
  a_2 = \frac{B(-y)-B(-z)}{y-z}, \quad y = \D_{K,\sigma}\log u^k, 
  \quad z = \kappa^{-1}\D_{K,\sigma}p^k. \nonumber 
\end{align*}
This representation is equivalent to $\mathcal{F}_{K,\sigma}^{(1)}[u^k,p^k]$ and implies a discrete analog of the entropy equality \eqref{1.HBR} but its implementation is delicate due to the logarithm.

We wish to find a discrete analog of the Boltzmann entropy equality \eqref{1.HB}. To this end, we split the numerical flux into two parts according to the diffusion part $-\kappa\na u$ and the drift part $-u\na p$, namely $\mathcal{F}_{K,\sigma}=\mathcal{F}_{K,\sigma}^{\rm diff}+\mathcal{F}_{K,\sigma}^{\rm drift}$. The aim is to prove the inequalities
\begin{align}
  -\sum_{\sigma=K|L}\mathcal{F}_{K,\sigma}^{\rm diff}
  \D_{K,\sigma}\log u^k_K
  &\ge \sum_{\sigma=K|L}\tau_\sigma|\D_{K,\sigma}(u^k)^{1/2}|^2, 
  \label{1.ineq1} \\
  -\sum_{\sigma=K|L}\mathcal{F}_{K,\sigma}^{\rm drift}
  \D_{K,\sigma}\log u^k_K
  &\ge \sum_{\sigma=K|L}\tau_\sigma\D_{K,\sigma}u^k\D_{K,\sigma}p^k,
  \label{1.ineq2}
\end{align}
which are associated to \eqref{1.HB}. Using the classical Scharfetter--Gummel scheme \eqref{1.SG1}, inequality \eqref{1.ineq2} does not hold generally, indicating that the (mean-value) drift term becomes too small in this formulation. This issue also arises with formulation \eqref{1.SG2}. Hence, it is necessary to find a new structure with an enlarged drift term. 

Our first idea is to use an upwind formula for the Scharfetter--Gummel flux. This can be motivated as follows. We deduce from the formula $B(-s)+B(s)=s$ the two equivalent formulations
\begin{align}
  \mathcal{F}_{K,\sigma}^{(1)}[u^k,p^k]
  &= -\tau_\sigma\big(B_\kappa(\D_{K,\sigma}p^k)\D_{K,\sigma}u^k
  + u_L^k\D_{K,\sigma}p^k\big), \label{1.F1} \\
  \mathcal{F}_{K,\sigma}^{(1)}[u^k,p^k]
  &= -\tau_\sigma\big(B_\kappa(-\D_{K,\sigma}p^k)\D_{K,\sigma}u^k
  + u^k_K\D_{K,\sigma}p^k\big). \label{1.F2}
\end{align}
Now, if $\D_{K,\sigma}p^k\ge 0$, we select formula \eqref{1.F1}, while we choose formula \eqref{1.F2} if $\D_{K,\sigma}p^k<0$. This leads to
\begin{align}\label{1.F3}
  \mathcal{F}_{K,\sigma}^{(3)}[u^k,p^k]
  = -\tau_\sigma\big(B_\kappa(|\D_{K,\sigma}p^k|)\D_{K,\sigma}u^k
  + \widehat{u}^k_\sigma\D_{K,\sigma}p^k\big),
\end{align}
where $\widehat{u}^k_\sigma=u_L^k$ if $\D_{K,\sigma}p^k\ge 0$ and $\widehat{u}^k_\sigma=u_K^k$ if $\D_{K,\sigma}p^k<0$. This formulation has two advantages. First, it allows us to derive inequality \eqref{1.ineq2} for the drift part. Second, the Bernoulli function in this formulation exhibits a linear dependence on $\D_{K,\sigma}p^k$, due to the bound $1 - |s|/2 \le B(|s|) \le 1$, in contrast to the exponential dependence of $\widetilde{B}_\kappa$ in \eqref{1.SG1}.

Unfortunately, formulation \eqref{1.F3} suffers from the degeneracy in the diffusion part (since $B(s)\to 0$ as $s\to\infty$), and the derivation of inequality \eqref{1.ineq1} fails. To overcome this issue, our second idea is to derive an estimate for the discrete analog of the Fisher information $\int_{\T^d}|\na\sqrt{u}|^2\dd x$. For this, we take advantage of the Rao entropy and consider generalized Bernoulli functions satisfying $B(s)\ge 1-\alpha s$ for some $0\le\alpha<1$. (The classical Bernoulli function $B(s)=s/(\exp s-1)$ satisfies this condition with $\alpha=1/2$.) Then the discrete Fisher information is controlled by the discrete Boltzmann and Rao entropy production terms and the discrete analog of the cross term $\int_{\T^d}\na u\cdot\na p\dd x$; see Lemma \ref{lem.fisher}. This estimate is used to compute the discrete entropy inequality (assuming an implicit Euler time discretization), resulting in 
\begin{equation}\label{discrete entropy intro}
\begin{aligned}
  \frac{1-\alpha}{\Delta t}&\big(H_B(u^k)-H_B(u^{k-1})\big)
  + \frac{\alpha}{\kappa\Delta t}\big(H_R(u^k)-H_R(u^{k-1})\big) \\
  &+ \kappa(1-\alpha)^2\sum_{\sigma=K|L}
  \tau_\sigma|\D_{K,\sigma}(u^k)^{1/2}|^2 \le 0,
\end{aligned}
\end{equation}
where $H_B$ and $H_R$ are the discrete Boltzmann and Rao entropies, respectively, defined in \eqref{2.HBHR} below. This inequality holds for positive semidefinite kernels (Lemma \ref{lem.fisher1}) and for attractive interactions if the initial data are sufficiently small (Lemma \ref{lem.fisher2}). The discrete gradient bound for $(u^k)^{1/2}$ is the key estimate for the convergence analysis. We make the previous arguments rigorous and extend them to the multi-species problem.

\subsection{Results}

Our main results can be sketched as follows (see Section \ref{sec.main} for details):
\begin{itemize}
\item We prove in Theorem \ref{thm.ex} the existence of solutions to an implicit Euler finite-volume scheme using the Scharfetter--Gummel flux \eqref{1.F3}. The solutions are positive componentwise, conserve the discrete mass, and satisfy \emph{degenerate} discrete versions of the entropy inequalities \eqref{1.HB} and \eqref{1.HR}. 
\item We derive an estimate for the discrete Fisher information, uniform in the mesh parameters, both for repulsive and attractive interactions; see Lemmas \ref{lem.fisher1} and \ref{lem.fisher2}. The key idea for this bound is the use of the condition $B(|s|)\ge 1-\alpha |s|$ for the generalized Bernoulli function and derive \emph{nondegenerate} discrete entropy inequalities \eqref{discrete entropy intro}.
\item We show that the discrete solutions converge to a weak solution to \eqref{1.eq}--\eqref{1.ic} as the mesh size converges to zero both for repulsive and attractive interactions; see Theorems \ref{thm.rep} and \ref{thm.att}.
\item We carry out some numerical experiments to validate the main features of the proposed method and to demonstrate the second-order convergence rate in space and first-order convergence rate in time. Using the proposed method, both repulsive and attractive interactions are studied. Strong repulsive self-interactions may exhibit small-scale oscillations in one and two space dimensions; see Section \ref{sec.num}.
\end{itemize}

The paper is organized as follows. The notation and the precise theorems are introduced in Section \ref{sec.notation}. The existence of discrete solutions is proved in Section \ref{sec.ex}, and the uniform estimate for the discrete Fisher information is shown in Section \ref{sec.fisher}. In Section \ref{sec.est}, further estimates are derived and the convergence of the scheme is proved, based on a compactness argument. The numerical experiments are presented in Section \ref{sec.num}.


\section{Notation and main results}\label{sec.notation}

We introduce the notation and definitions needed for the finite-volume scheme and detail our main results.

\subsection{Finite-volume notation}

We denote by the triplet $(\mathcal{T},\mathcal{E},\mathcal{P})$ a Cartesian finite-volume mesh of the (open) torus $\T^d$. The set $\mathcal{T}$ is the collection of control volumes (or cells), consisting of $M_1\times\cdots\times M_d$ identical hyper-rectangles. The length of the $\ell$th direction equals $\Delta x_\ell=1/M_\ell$. The set $\mathcal{E}$ consists of the edges (or hyper-surfaces) of the mesh, each lying in an affine hyperplane of codimension one. It is partitioned into the set $\mathcal{E}_{\rm int}=\mathcal{E}\cap\T^d$ of internal edges and periodic boundary edge pairs $\mathcal{E}_{\rm per}$. The set $\mathcal{E}_{\rm per}$ consists of pairs $(\sigma, \sigma')$, where $\sigma$ and $\sigma'$ are geometrically coincident edges located on opposite boundaries of the domain. We identify each such pair $(\sigma, \sigma')$ as a single element within $\mathcal{E}$. This identification reflects the toroidal topology of the domain. Finally, the set $\mathcal{P}$ contains the centers $x_K\in\T^d$ of all control volumes $K$, i.e.\ for a cell $K$ indexed by $(i_1,\ldots,i_d)$, the center is given by
\begin{align*}
  x_K = \big((i_1-1/2)\Delta x_1,\ldots,(i_d-1/2)\Delta x_d\big), \quad  \mbox{where }i_\ell\in\{1,\ldots,M_\ell\}.
\end{align*}
In view of the periodic boundary conditions, we identify the centers with indices $(i_1,\ldots,i_\ell,$ $\ldots,i_d)$ and $(i_1,\ldots,i_\ell+M_\ell,\ldots,i_d)$. The notation $\sigma = K|L$ is used if two cells $K$ and $L$ are adjacent either through an internal edge or through a periodic boundary.

For given $K\in\mathcal{T}$, $\mathcal{E}_K$ denotes the set of edges of $K$. For $\sigma\in\mathcal{E}$, we introduce the distance
\begin{align*}
  \dd_\sigma = \begin{cases}
  \dd(x_K,x_L) &\quad\mbox{if }\sigma=K|L\in\mathcal{E}_{\rm int}, \\
  \dd(x_K,\sigma)+\dd(\sigma',x_L) &\quad\mbox{if }(\sigma,\sigma')
  \in\mathcal{E}_{\rm per},\ \sigma\subset\pa K, \sigma'\subset\pa L,
  \end{cases}
\end{align*}
where d is the Euclidean distance in $\R^d$. Because of the uniform mesh, we have $\dd_\sigma=\Delta x_\ell$ for some $\ell\in\{1,\ldots,d\}$. We define the transmissibility coefficient
\begin{align}\label{2.tau}
  \tau_\sigma = \frac{\m(\sigma)}{\dd_\sigma}.
\end{align}

Let $T>0$ and $N\in\N$. We define the time step size $\Delta t=T/N$ and the time steps $t_k=k\Delta t$ for $k=0,\ldots,N$. The size of the space-time discretization $\mathcal{D}$, consisting of the mesh $\mathcal{T}$ and the values $(\Delta t,N)$, is defined by
\begin{align}\label{2.delta}
  \delta = \max\{h,\Delta t\}, \quad
  h = \max\{\Delta x_1,\ldots,\Delta x_d\}.
\end{align}

For the convergence result, we introduce the dual mesh $\mathcal{T}^*$ of $\mathcal{T}$. We associate to $K\in\mathcal{T}$ and $\sigma\in\mathcal{E}_K$ a dual cell $\Delta_\sigma\in\mathcal{T}^*$:
\begin{itemize}
\item ``Diamond'': If $\sigma=K|L\in\mathcal{E}_{\rm int}$, then $\Delta_\sigma$ is the interior of the convex hull of $\sigma\cup\{x_K,x_L\}$.
\item ``Triangles'': If $\{\sigma,\sigma'\}\in\mathcal{E}_{\rm int}$ with $\sigma\subset\pa K$, then $\Delta_\sigma$ is the interior of the convex hull of $\sigma\cup\{x_K\}$, along with the convex hull of $\sigma'\cup\{x_L\}$. 
\end{itemize}
The volume of the dual cell $\Delta_\sigma$ is computed by
\begin{align}\label{2.Deltasig}
  d\m(\Delta_\sigma) = \m(\sigma)\dd_\sigma \quad\mbox{for }
  \sigma\in\mathcal{E}. 
\end{align}
Notice that $\m(\sigma)\dd_\sigma=\m(K)$ for $\sigma\in\mathcal{E}_{K}$. We introduce for $\sigma=K|L$ the difference operators
\begin{align*}
  \D_{K,\sigma}v =  v_L-v_K,
  \quad \D_\sigma v= |\D_{K,\sigma}v|
\end{align*}
and the discrete gradient
\begin{align}\label{2.nahv}
  \na_\sigma^h v = d\frac{\D_{K,\sigma}v}{\dd_\sigma}\nu_{K,\sigma},
\end{align}
where $\nu_{K,\sigma}=(x_L-x_K)/\dd_\sigma$ is the unit vector that is normal to $\sigma$ and points outwards of $K$. For any $\sigma=K|L$ such that $x_K=x_L+\Delta x_\ell e_\ell$, where $\ell\in\{\pm 1,\ldots,\pm d\}$ and $e_\ell$ is the Euclidean unit vector of $\R^d$, we set, slightly abusing the notation,
\begin{align}\label{2.DKell}
  \D_{K,\ell}v := \D_{K,\sigma}v,
\end{align} 
where $e_{-\ell}:=-e_\ell$ for $\ell=1,\ldots,d$. 

Finally, we introduce the following reconstruction operators. Let $u=(u_K^k)_{K\in\mathcal{T},\,k=1,\ldots,N}$ and $v=(v_\sigma^k)_{\sigma\in\mathcal{E},\,k=1,\ldots,N}$ be given. Then we define
\begin{equation}\label{2.pi}
\begin{aligned}
  \pi_\delta u(t,x) = u_K^k 
  &\quad\mbox{for }(t,x)\in(t_{k-1},t_k]\times K, \\
  \pi_\delta^* v(t,x) = v_\sigma^k 
  &\quad\mbox{for }(t,x)\in(t_{k-1},t_k]\times\Delta_\sigma.
\end{aligned}
\end{equation}


\subsection{Discrete functional spaces}

Let $u=(u_K)_{K\in\mathcal{T}}$ be given and let $1\le p<\infty$. The discrete $L^p(\T^d)$ norm is defined by
\begin{align*}
  \|u\|_{0,p,\mathcal{T}} = \bigg(\sum_{K\in\mathcal{T}}\m(K)
  |u_K|^p\bigg)^{1/p}.
\end{align*}
For a function $v=(v_\sigma)_{\sigma\in\mathcal{E}}$, defined on the edges, the associated discrete $L^p(\T^d)$ norm is given by
\begin{align*}
  \|v\|_{0,p,\mathcal{T}^*} = \bigg(\sum_{\sigma\in\mathcal{E}}
  \m(\Delta_\sigma)|v_\sigma|^p\bigg)^{1/p}.
\end{align*}
Then the discrete $W^{1,p}(\T^d)$ norm reads as
\begin{align*}
  \|u\|_{1,p,\mathcal{T}} = \|u\|_{0,p,\mathcal{T}}
  + \|\na^h u\|_{0,p,\mathcal{T}^*},
\end{align*}
where $\na^h u=(\na^h_\sigma u)_{\sigma\in\mathcal{E}}$. Let $p>1$ and $1/p+1/q=1$. The dual norm to the $L^p(\T^d)$ norm with respect to the $L^2(\T^d)$ inner product is given by
\begin{align*}
  \|u\|_{-1,q,\mathcal{T}^*} = \sup_{\|\phi\|_{1,p,\mathcal{T}}=1}
  \bigg|\sum_{K\in\mathcal{T}}\m(K)u_K\phi_K\bigg|.
\end{align*}
For a space-time discrete function $(u^k_K)$, the discrete time derivative is denoted by
\begin{align}\label{2.ut}
  \pa_t^{\Delta t}u^k = \frac{u^k-u^{k-1}}{\Delta t}
  \quad\mbox{for }k=1,\ldots,N.
\end{align}
With this notation, the following identities hold for a function $u=(u_K)_{K\in\mathcal{T}}$:
\begin{align*}
  \|u\|_{0,p,\mathcal{T}} = \|\pi_\delta u\|_{L^p(\T^d)}, \quad
  \|\na^h u\|_{0,q,\mathcal{T}^*}
  = \|\pi_\delta^*(\na^h u)\|_{L^q(\T^d)}. 
\end{align*}


\subsection{Numerical scheme}

We introduce the two-point approximation finite-volume scheme for the cross-diffusion system \eqref{1.eq}--\eqref{1.ic}. The initial datum is approximated by
\begin{align*}
  u_{i,K}^0 = \frac{1}{\m(K)}\int_K u_i^0(x)\dd x
  \quad\mbox{for all }K\in\mathcal{T},\ i=1,\ldots,n.
\end{align*}
Let $u_K^{k-1}=(u_{1,K}^{k-1},\ldots,u_{n,K}^{k-1})$ be given for $K\in\mathcal{T}$. The implicit Euler finite-volume scheme for the values $u_{i,K}^k$, approximating $u_i(t_k)$ on the cell $K$, is defined by
\begin{align}\label{2.euler}
  \m(K)\frac{u_{i,K}^k-u_{i,K}^{k-1}}{\Delta t}
  + \sum_{\sigma\in\mathcal{E}_K}\mathcal{F}_{K,\sigma}[u_i^k,p_i^k]
  = 0.
\end{align}
In the following, we define the numerical flux $\mathcal{F}_{K,\sigma}[u^k_i,p^k_i]$ and the discrete potentials $p_i^k$.

Let $B\in C^0([0,\infty))$ be a weight function satisfying $0<B(s)\le 1$ for $s\ge 0$ and assume that there exists $0\le\alpha<1$ such that
\begin{align*}
  B(s) \ge 1-\alpha s\quad\mbox{for }0\le s\le 1/\alpha.
\end{align*}
These conditions imply that $B(0)=1$. Examples are the standard upwind scheme with $B(s)=1$ (with $\alpha=0$), the Scharfetter--Gummel scheme with the Bernoulli function $B(s)=s/(e^s-1)$ (with $\alpha=1/2$), and the geometric mean scheme $B(s)=e^{-s/2}$ (with $\alpha=1/2$). Further examples can be constructed by means of the Stolarsky mean; see \cite[(1.7)]{HKS21}. Furthermore, we set $B_\kappa(s)=\kappa B(s/\kappa)$ for $s\ge 0$. As motivated in the introduction, the numerical flux $\mathcal{F}_{K,\sigma}$ is given by the generalized Scharfetter--Gummel flux
\begin{align}\label{2.flux}
  \mathcal{F}_{K,\sigma}[u_i^k,p_i^k] 
  = -\tau_\sigma\big(B_\kappa(\D_{\sigma} p_i^k)\D_{K,\sigma}u_i^k
  + \widehat{u}_{i,\sigma}^k\D_{K,\sigma}p_i^k\big),
\end{align}
where $\widehat{u}_{i,\sigma}^k$ is an edge concentration with upwind structure:
\begin{align}\label{2.uhat}
  \widehat{u}_{i,\sigma}^k = \widehat{u}_{i,\sigma}^k(p_i^k)
  &= \left.\begin{cases}
  u_{i,L}^k &\mbox{if }\D_{K,\sigma}p_i^k \ge 0 \\
  u_{i,K}^k &\mbox{if }\D_{K,\sigma}p_i^k < 0
  \end{cases}\right\} = \frac{[\D_{K,\sigma}p_i^k]^+}{\D_{K,\sigma}p_i^k}u_{i,L}^k
  - \frac{[\D_{K,\sigma}p_i^k]^-}{\D_{K,\sigma}p_i^k}u_{i,K}^k
\end{align}
for $\sigma=K|L$ and $[s]^+=\max\{0,s\}$, $[s]^-=\max\{0,-s\}$. We recall the discrete integration-by-parts formula for piecewise constant functions $(v_K)$ \cite[(14)]{CCGJ18}:
\begin{align}\label{2.dibp}
  \sum_{K\in\mathcal{T}}\sum_{\sigma\in\mathcal{E}_K}
  \mathcal{F}_{K,\sigma}[u_i^k,p_i^k] v_K
  =
  -\sum_{\sigma=K|L\in\mathcal{E}}
  \mathcal{F}_{K,\sigma}[u_i^k,p_i^k] \D_{K,\sigma}v.
\end{align}

We introduce the discrete kernels $W_{KJ}^{ij}$ by 
\begin{align}\label{2.Wij}
  W_{KJ}^{ij} = \frac{1}{\m(K)\m(J)}
  \int_K\int_JW_{ij}(x-y)\dd y\dd x
  \quad\mbox{for }K,J\in\mathcal{T},\ i,j=1,\ldots,n.
\end{align}
There are various options to discretize the kernels. Our choice corresponds to \cite[Remark 2.3]{HST24}. Other choices can be found in, for instance, \cite[(7)]{CFS20}, \cite[Sec.~2.2]{HST24}, and \cite[(13)]{JPZ24}. 

To define the discrete potentials, we distinguish between repulsive and attractive self-interactions. For repulsive self-interactions, we define $p_{i,K}^k$ by a fully implicit in time scheme:
\begin{align}\label{3.repp}
  p_{i,K}^k = \sum_{j=1}^n\sum_{J\in\mathcal{T}}\m(J)
  W_{KJ}^{ij}u_{j,J}^k,
\end{align}
while for attractive self-interactions, we use a mid-point time averaging:
\begin{align}\label{3.attp}
  p_{i,K}^k = \sum_{j=1}^n\sum_{J\in\mathcal{T}}\m(J)
  W_{KJ}^{ij}\frac{u_{j,J}^k+u_{j,J}^{k-1}}{2}.
\end{align}


\subsection{Assumptions and precise statements of the main results}\label{sec.main}

We impose the following hypotheses:
\begin{itemize}
\item[(H1)] Domain and initial data: Let $T>0$, $\Omega_T=(0,T)\times\T^d$ and let $u_i^0\in L^2(\T^d)$ satisfy $u_i^0\ge 0$ in $\T^d$, $\|u_i^0\|_{L^1(\T^d)}\neq 0$ for $i=1,\ldots,n$. 
\item[(H2)] Weight function: $B\in C^0([0,\infty))$ satisfies $0<B(s)\le 1$ for $s\ge 0$.
\item[(H3)] Nonuniform coercivity: There exists $0\le\alpha<1$ such that
\begin{align*}
  B(s) \ge 1-\alpha s\quad\mbox{for }0\le s\le 1/\alpha.
\end{align*}
\item[(H4)] Symmetry: $W_{ij}\in L^1(\T^d)$ satisfies $W_{ij}(x)=W_{ji}(-x)$ for all $i,j=1,\ldots,n$ and a.e.\ $x\in\T^d$.
\item[(H5)] Positive semidefiniteness: $W_{ij}\in L^1(\T^d)$ satisfies for all $v_1,\ldots,v_n\in L^2(\Omega)$,
\begin{align*}
  \sum_{i,j=1}^n\int_{\T^d}\int_{\T^d}W_{ij}(x-y)v_i(x)v_j(y)\dd x\dd y
  \ge 0.
\end{align*}
\item[(H6)] Boundedness: $W_{ij}\in L^\infty(\T^d)$ for $i,j=1,\ldots,n$.
\end{itemize}

The condition $\|u_i^0\|_{L^1(\T^d)}\neq 0$ in Hypothesis (H1) is needed to ensure the positivity of $u_{i,K}^k>0$ for all $K\in\mathcal{T}$, $i=1,\ldots,n$, and $k\ge 1$. A significant technical difficulty comes from the fact that the weight function $B$ may be not uniformly positive since $B(s)\to 0$ as $s\to\infty$ is possible. We assume in Hypothesis (H3) a control of the positivity at least for small values $s>0$. Interestingly, we do not need any monotonicity condition on $B$. The symmetry of the kernels in Hypothesis (H4) is needed, for instance, to derive the discrete Rao entropy inequality. We request the positive semidefiniteness in Hypothesis (H5) for the self-repulsive case, while the boundedness of $W_{ij}$ in Hypothesis (H6) is required for the self-attractive case.

Under Hypothesis (H4), the discrete kernel $W_{KJ}^{ij}$ is symmetric in the sense $W_{KJ}^{ij}=W_{JK}^{ji}$, while under Hypothesis (H5), we have for any discrete functions $(v_{i,K})_{K\in\mathcal{T}}$ and $i=1,\ldots,n$,
\begin{align}\label{2.Wpos}
  \sum_{i,j=1}^n\sum_{K,J\in\mathcal{T}}\m(K)\m(J)W_{KJ}^{ij}
  v_{i,K}v_{j,J} 
  = \sum_{i,j=1}^n\int_{\T^d}\int_{\T^d}W_{ij}(x-y)\pi_\delta v_i(x)
  \pi_\delta v_j(y)\dd x\dd y\ge 0,
\end{align}
where we recall definition \eqref{2.pi} of the interpolation operator $\pi_\delta$. We introduce the discrete Boltzmann and Rao entropies
\begin{equation}\label{2.HBHR}
\begin{aligned}
  H_B(u^k) &= \sum_{i=1}^n\sum_{K\in\mathcal{T}}\m(K)
  u_{i,K}^k(\log u_{i,K}^k-1), \\
  H_R(u^k) &= \frac12\sum_{i,j=1}^n\sum_{K,J\in\mathcal{T}}
  \m(K)\m(J)W_{KJ}^{ij} u_{i,K}^k u_{j,J}^k.
\end{aligned}
\end{equation}

Our first main result is the existence of a discrete solution to scheme \eqref{2.euler}--\eqref{2.uhat} and its entropy producing properties.

\begin{theorem}[Existence of discrete solution]\label{thm.ex}
Let Hypotheses (H1)--(H2) hold. Then there exists a solution $(u_K^k)$ to scheme \eqref{2.euler}--\eqref{2.uhat} satisfying positivity and mass conservation:
\begin{align*}
  u_{i,K}^k>0\quad \mbox{for all }K\in\mathcal{T}, \quad 
  \sum_{K\in\mathcal{T}}\m(K)u_{i,K}^k 
  = \sum_{K\in\mathcal{T}}\m(K)u_{i,K}^0
\end{align*}
for all $i=1,\ldots,n$. If additionally Hypotheses (H3)--(H5) hold, the discrete Boltzmann and Rao entropy inequalities hold:
\begin{align}
  \frac{1}{\Delta t}\big(H_B(u^k) - H_B(u^{k-1})\big) + P_B(u^k,p^k)
  &\le -X(u^k,p^k), \label{2.HB} \\
  \frac{1}{\Delta t}\big(H_R(u^k) - H_R(^{k-1})\big)
  + (1-\alpha)P_R(u^k,p^k) &\le -\kappa X(u^k,p^k), \label{2.HR}
\end{align}
where the entropy production terms $P_B(u^k,p^k)$ and $P_R(u^k,p^k)$ and the cross term $X(u^k,p^k)$ are given by
\begin{align*}
  P_B(u^k,p^k) &= 4\kappa\sum_{i=1}^n\sum_{\sigma\in\mathcal{E}}
  \tau_\sigma B_\kappa(\D_\sigma p_i^k)|\D_\sigma(u_i^k)^{1/2}|^2, \\
  P_R(u^k,p^k) &= \sum_{i,j=1}^n\sum_{\sigma\in\mathcal{E}}
  \tau_\sigma\widehat{u}_{i,\sigma}^k|\D_\sigma p_i^k|^2, \\
  X(u^k,p^k) &= \sum_{i=1}^n\sum_{\sigma=K|L\in\mathcal{E}}
  \tau_\sigma (\D_{K,\sigma}p_i^k)(\D_{K,\sigma}u_i^k).
\end{align*}
\end{theorem}

Inequalities \eqref{2.HB} and \eqref{2.HR} are the discrete analogs of the entropy equalities \eqref{1.HB} and \eqref{1.HR}, taking into account the Bernoulli function $B_\kappa$. In fact, the formulas coincide if $\alpha=0$, which corresponds to the upwind choice $B(s)=1$ for $s\ge 0$. If $\alpha>0$, the Rao entropy production term is reduced, which is the price to pay for our scheme.

\begin{remark}[Zero-diffusion limit]\rm 
In the zero-diffusion limit $\kappa\to 0$, the scheme converges formally to the upwind finite-volume approximation for the aggregation equation. Indeed, since $B_\kappa(s)\to [s]^-$ as $\kappa\to 0$ and $\D_\sigma p_i^k\ge 0$ by construction, we obtain
\begin{align*}
  \mathcal{F}_{K,\sigma}[u^k,p^k] \to
  -\tau_\sigma\widehat{u}_{i,\sigma}^k\D_{K,\sigma}p_i^k
  \quad\mbox{as }\kappa\to 0.
\end{align*}
Using the classical Scharfetter--Gummel flux \eqref{1.SG1}, this limit was proved in \cite[Theorem B]{HST24} in the sense of the energy-dissipative principle for gradient flows.
\qed\end{remark}

The second main result concerns the convergence of the scheme. For this, we use either Hypothesis (H5) or (H6). Let $(\mathcal{D}_m)_{m\in\N}$ be a sequence of space-time discretizations of $[0,T]\times\T^d$ indexed by the size $\delta_m=\max\{h_m,\Delta t_m\}$ of the mesh, satisfying $\delta_m\to 0$ as $m\to\infty$. To simplifiy the notation, we define $(u_{m,i},p_{m,i}):=(\pi_{\delta_m}u_i,\pi_{\delta_m}p_i)$ and  $(\na^m u_{m,i}, \na^m p_{m,i}):=(\pi_{\delta_m}^* \na^{h_m} u_i,\pi_{\delta_m}^* \na^{h_m} p_i)$.

\begin{theorem}[Convergence of the scheme; positive semidefinite kernel matrix]\label{thm.rep}
Let Hypo\-theses (H1)--(H5) hold and $r=(d+2)/(d+1)$, $1\le s<(d+2)/d$. Let $(u_m)_{m\in\N}$ be a sequence of finite-volume solutions to scheme \eqref{2.euler}--\eqref{2.uhat} associated to the mesh $\mathcal{D}_m$ and let $p_m=p(u_m)$ be the associated discrete potential. Then there exists a function $u^*\in L^r(0,T;W^{1,r}(\T^d))$ satisfying $u^*_i\ge 0$ in $(0,T)\times\T^d$ and for $i=1,\ldots,n$,
\begin{equation*}
\begin{aligned}
  u_{m,i}\to u_i^*, \quad p_{m,i}\to p_i &\quad\mbox{strongly in }
  L^s(\Omega_T), \\
  \na^m u_{m,i}\rightharpoonup \na u_i^*, \quad  
  \na^m p_{m,i}\rightharpoonup \na p_i^* &\quad\mbox{weakly in }
  L^r(\Omega_T),
\end{aligned}
\end{equation*}
and $(u^*,p^*)$ solve \eqref{1.eq}--\eqref{1.ic} in the weak sense.
\end{theorem}

\begin{theorem}[Convergence of the scheme; bounded kernel functions]
\label{thm.att}\sloppy\ \ 
Let \ Hypotheses (H1)--(H4) and (H6) hold and let the initial data be sufficiently small in the sense
\begin{align*}
  \max_{j=1,\ldots,n}\sum_{i=1}^n\|W_{ij}\|_{L^\infty(\T^d)}
  \|u_j^0\|_{L^1(\T^d)} 
  \le \frac14\frac{\kappa(1-\alpha)^2}{\alpha(1-\alpha)+1}.
\end{align*}
Furthermore, we assume that either there exists $C>0$ such that
\begin{align}\label{assume: derivative of initial data}
  \sum_{i=1}^n\sum_{\sigma\in\mathcal{E}}
  \tau_\sigma|\D_{\sigma}(u_i^0)^{1/2}|^2 \le C
\end{align}
or the parabolic scaling $\Delta t_m\le Ch_m^2$ holds for some $C>0$ independent of $m$. Let $(u_m)_{m\in\N}$ be a sequence of finite-volume solutions to scheme \eqref{2.euler}--\eqref{2.uhat} associated to the mesh $\mathcal{D}_m$ and let $p_m$ be the associated discrete potential. Then there exists a function $u^*\in L^r(0,T;W^{1,r}(\T^d))$ satisfying the statements of Theorem \ref{thm.rep}. 
\end{theorem}

In the limit $\kappa\to 0$, the initial data is required to become smaller and smaller. Thus, the numerical convergence result does not hold for fully attractive interactions. This is not surprising, since, for instance, the aggregation equation for $n=1$ and $p=-u$ equals the backward porous-medium equation $\pa_t u=-\diver(u\na u)$, which is not globally solvable.


\section{Existence of discrete solutions}\label{sec.ex}

We prove the existence of a discrete solution to scheme \eqref{2.euler}--\eqref{2.uhat} by applying Schaefer's fixed-point theorem and derive the discrete Boltzmann and Rao entropy inequalities.

\subsection{Discrete Fokker--Planck equation}

To prove some properties of the fixed-point operator, we need to solve a linear discrete problem for functions $u^k=(u_K^k)_{K\in\mathcal{T}}$ solving the {\em scalar} problem
\begin{equation}\label{3.dfp}
\begin{aligned}
  & \m(K)\frac{u_K^k-u_K^{k-1}}{\Delta t} + \sum_{\sigma\in\mathcal{E}_K}
  \mathcal{F}_{K,\sigma}[u^k,p] = 0 \quad\mbox{for }K\in\mathcal{T}, \\
  & \mathcal{F}_{K,\sigma}[u^k,p] 
  = -\tau_\sigma\big(B_\kappa(\D_{\sigma} p)\D_{K,\sigma}u^k 
  + \widehat{u}_{i,\sigma}^k\D_{K,\sigma}p\big),
\end{aligned}
\end{equation}
where $u_K^{k-1}\ge 0$ for $K\in\mathcal{T}$ is such that $u_{L}^{k-1}>0$ for some $L\in\mathcal{T}$,  $p=(p_K)_{K\in\mathcal{T}}$ is a given potential, and $\widehat{u}_{i,\sigma}^k$ is the upwind term defined in \eqref{2.uhat}.

\begin{lemma}\label{lem.lin}
The discrete Fokker--Planck equation \eqref{3.dfp} has a unique solution $u^k$ which is strictly positive and preserves the mass in the sense
\begin{align*}
  \sum_{K\in\mathcal{T}}\m(K)u_K^k = \sum_{K\in\mathcal{T}}\m(K)
  u_K^{k-1}.
\end{align*}
\end{lemma}

\begin{proof}
The proof is similar to \cite[Prop.~1]{Bes12}; we present the details for the sake of completeness. We can formulate scheme \eqref{3.dfp} in the matrix form
\begin{align}\label{3.linsys}
  A(p)u^k = S(u^{k-1}),
\end{align}
where the matrix $A(p)$ is defined by
\begin{align*}
  A_{KK}(p) &= \frac{\m(K)}{\Delta t} + \sum_{\sigma\in\mathcal{E}_K}
  \tau_\sigma\big(B_\kappa(\D_{\sigma} p) + [\D_{K,\sigma}p]^-\big)
  \quad\mbox{for }K\in\mathcal{T}, \\
  A_{KL}(p) &= -\tau_\sigma\big(B_\kappa(\D_{\sigma} p) 
  + [\D_{K,\sigma}p]^+\big)\quad\mbox{for }K,L\in\mathcal{T}
  \mbox{ with }\sigma=K|L,
\end{align*}
and the vector $S(u^{k-1})$ is given by
\begin{align*}
  S_K(u^{k-1}) = \frac{\m(K)}{\Delta t}u_K^{k-1}
  \quad\mbox{for }K\in\mathcal{T}.
\end{align*}
The diagonal entries of the matrix $A(p)$ are positive and the off-diagonal entries are nonpositive. Moreover, since $|[\D_{K,\sigma} p]^+| = |[-\D_{L,\sigma}p]^+| = |[\D_{L,\sigma}p]^-|$ for $\sigma=K|L\in\mathcal{E}_K$, we have
\begin{align*}
  |A_{LL}(p)| - \sum_{K\in\mathcal{T},\, K\neq L}|A_{KL}|
  = \frac{\m(K)}{\Delta t} > 0.
\end{align*}
Hence, $A(p)$ is strictly diagonally dominant with respect to the columns. By \cite[Theorem 1.]{Var62}, $A(p)$ is invertible and, in fact, an M-matrix. This gives the existence of a unique solution to the linear system \eqref{3.linsys}. 

The M-matrix property implies that $A(p)$ is inverse-positive, i.e., all entries of $A(p)^{-1}$ are nonnegative. Thus, since the components of $u^{k-1}$ are nonnegative, the components of the solution $u^k$ are nonnegative as well. We claim that they are even positive. By contradiction, assume that $u_{K^*}^k=0$ for some $K^*\in\mathcal{T}$. Then, using scheme \eqref{3.dfp} and definition \eqref{2.uhat} of $\widehat{u}_{i,\sigma}^k$,
\begin{align*}
  0 = \frac{\m(K^*)}{\Delta t}u_{K^*}^{k-1}
  + \sum_{\sigma=K^*|L\in\mathcal{E}_{K^*}}\tau_\sigma
  \big(B_\kappa(\D_{\sigma} p) + [\D_{K^*,\sigma}p]^+\big)u_L^k.
\end{align*}
We know that $u_{K^*}^{k-1}$ is nonnegative and $B_\kappa(s)>0$ for all $s\ge 0$. This implies that $u_L^k=0$ for all neighboring cells $L$ 
of $K^*$. Repeating this argument for all cells in $\mathcal{T}$, we find that $u_K^k=0$ for all $K\in\mathcal{T}$. Then scheme \eqref{3.dfp} leads to $u_K^{k-1}=0$, which contradicts our hypothesis that $u_L^{k-1}>0$ for some $L\in\mathcal{T}$. Hence, $u_K^k>0$ for all $K\in\mathcal{T}$ and $k\ge 1$. 

Finally, the mass conservation is a consequence of the local conservation of the numerical fluxes:
\begin{align*}
  \sum_{K\in\mathcal{T}}\frac{\m(K)}{\Delta t}u_K^k
  = \sum_{K\in\mathcal{T}}\frac{\m(K)}{\Delta t}u_K^{k-1}
  - \sum_{K\in\mathcal{T}}\sum_{\sigma\in\mathcal{E}_K}
  \mathcal{F}_{K,\sigma}^k
  = \sum_{K\in\mathcal{T}}\frac{\m(K)}{\Delta t}u_K^{k-1},
\end{align*}
which finishes the proof.
\end{proof}


\subsection{Proof of Theorem \ref{thm.ex}}

We prove the existence result by induction on $k\ge 0$. For $k=0$, the statement follows from our assumptions. We suppose that $u^{k-1}$ is known for some $k\ge 1$, being nonnegative componentwise and conserving the total mass. We prove the existence of a solution $u^k$ to scheme \eqref{2.euler}--\eqref{2.uhat} by applying Schaefer's fixed-point theorem. To this end, let $u^*=(u_{i,K}^*)$ be given and consider the linear problem
\begin{align*}
  \m(K)\frac{u_{i,K}^k-u_{i,K}^{k-1}}{\Delta t}
  + \sum_{\sigma\in\mathcal{E}_K}\mathcal{F}_{K,\sigma}[u_i^k,p_i^*]
  = 0, \quad
  p_{i,K}^* = \sum_{j=1}^n\sum_{J\in\mathcal{T}}\m(J)
  W_{KJ}^{ij}u_{j,J}^*,
\end{align*}
and $\mathcal{F}_{K,\sigma}[u_i^k,p_i^*]$ is defined in \eqref{2.flux}. For given $p_i^*$, the existence of a solution $u_i^k$ to this linear problem follows from Lemma \ref{lem.lin}. This defines the mapping $S:\R^{n|\mathcal{T}|}\to\R^{n|\mathcal{T}|}$, $S(u^*) = u^k = (u_i^k)_{i=1,\ldots,n}$. Standard arguments show that $S$ is continuous. We infer from mass conservation that the set $\{u^k\in\R^{n|\mathcal{T}|}:$ $\exists\theta\in[0,1]$, $u^k = \theta S(u^*)\}$ is bounded. By Schaefer's fixed-point theorem \cite[Sec.~9.2.2]{Eva98}, there exists a fixed point $u^k$ of $S$, which is a solution to scheme \eqref{2.euler}--\eqref{2.uhat}. The strict positivity and mass conservation of $u^k$ is a result of Lemma \ref{lem.lin}. 

Next, we verify the discrete Boltzmann entropy inequality \eqref{2.HB}. We abbreviate $B_{i,\sigma}^k:=B(\kappa^{-1}\D_{\sigma} p_i^k)$. We multiply scheme \eqref{2.euler} by $\log u_{i,K}^k$ (which is well-defined since $u_{i,K}^k>0$) and sum over $i=1,\ldots,n$ and $K\in\mathcal{T}$. This gives $I_1+I_2=0$, where
\begin{align*}
  I_1 = \frac{1}{\Delta t}\sum_{i=1}^n\sum_{K\in\mathcal{T}}
  \m(K)(u_{i,K}^k-u_{i,K}^{k-1})\log u_{i,K}^k, \quad
  I_2 = \sum_{i=1}^n\sum_{K\in\mathcal{T}}\sum_{\sigma\in\mathcal{E}_K}
  \mathcal{F}_{K,\sigma}[u_i^k,p_i^k]\log u_{i,K}^k.
\end{align*}
It follows from the convexity of $s\mapsto s(\log s-1)$ that
\begin{align*}
  I_1 \ge \frac{1}{\Delta t}\big(H_B(u^k) - H_B(u^{k-1})\big).
\end{align*}
We apply discrete integration by parts (see \eqref{2.dibp}) to find that
\begin{align*}
  I_2 &= -\sum_{i=1}^n\sum_{\sigma=K|L\in\mathcal{E}}
  \mathcal{F}_{K,\sigma}[u_i^k,p_i^k]\D_{K,\sigma}\log u_i^k \\
  &= \sum_{i=1}^n\sum_{\sigma=K|L\in\mathcal{E} }
  \tau_\sigma\big(\kappa B_{i,\sigma}^k\D_{K,\sigma}u_i^k
  + \widehat{u}_{i,\sigma}^k\D_{K,\sigma}p_i^k\big)
  \D_{K,\sigma}\log u_i^k = I_{21} + I_{22}.
\end{align*}
The diffusion part $I_{21}$ is estimated by using the elementary inequality $(\log a-\log b)(a-b)\ge 4(\sqrt{a}-\sqrt{b})^2$ for $a,b>0$:
\begin{align*}
  I_{21} \ge 4\kappa\sum_{i=1}^n\sum_{\sigma\in\mathcal{E}}
  \tau_\sigma B_{i,\sigma}^k|\D_\sigma(u_i^k)^{1/2}|^2.
\end{align*} 
It follows from $\log s\le s-1$ with $s=u_{i,K}^k/u_{i,L}^k$ and $s=u_{i,L}^k/u_{i,K}^k$ that
\begin{align*}
  \frac{\D_{K,\sigma}u_i^k}{u_{i,L}^k} 
  = 1 - \frac{u_{i,K}^k}{u_{i,L}^k} 
  \le \log\frac{u_{i,L}^k}{u_{i,K}^k} 
  = \D_{K,\sigma}\log u_i^k
  \le \frac{u_{i,L}^k}{u_{i,K}^k} - 1
  = \frac{\D_{K,\sigma}u_i^k}{u_{i,K}^k},
\end{align*}
and consequently,
\begin{align*}
  I_{22} &= \sum_{i=1}^n\sum_{\sigma=K|L\in\mathcal{E}}\tau_\sigma
  \big([\D_{K,\sigma}p_i^k]^+ u_{i,L}^k\D_{K,\sigma}\log u_i^k
  - [\D_{K,\sigma}p_i^k]^- u_{i,K}^k\D_{K,\sigma}\log u_i^k\big) \\
  &\ge \sum_{i=1}^n\sum_{\sigma=K|L\in\mathcal{E}}\tau_\sigma
  \big([\D_{K,\sigma}p_i^k]^+\D_{K,\sigma}u_i^k
  - [\D_{K,\sigma}p_i^k]^-\D_{K,\sigma}u_i^k\big) \\
  &= \sum_{i=1}^n\sum_{\sigma=K|L\in\mathcal{E}}\tau_\sigma
  \D_{K,\sigma}p_i^k\D_{K,\sigma}u_i^k.
\end{align*}
Putting the estimates together, we end up with
\begin{align*}
  0 = I_1 + I_2 &\ge \frac{1}{\Delta t}\big(H_B(u^k) - H_B(u^{k-1})\big)
  + 4\kappa\sum_{i=1}^n\sum_{\sigma\in\mathcal{E}}
  \tau_\sigma B_{i,\sigma}^k|\D_\sigma(u_i^k)^{1/2}|^2 \\
  &\phantom{xx}+ \sum_{i=1}^n\sum_{\sigma=K|L\in\mathcal{E}}\tau_\sigma
  \D_{K,\sigma}p_i^k\D_{K,\sigma}u_i^k,
\end{align*}
which proves inequality \eqref{2.HB}. 

It remains to verify the discrete Rao entropy inequality \eqref{2.HR}. For this, we multiply scheme \eqref{2.euler} by $p_{i,K}^k$ and sum over $i=1,\ldots,n$ and $K\in\mathcal{T}$, giving $I_3+I_4=0$, where
\begin{align*}
  I_3 = \frac{1}{\Delta t}\sum_{i=1}^n\sum_{K\in\mathcal{T}}
  \m(K)(u_{i,K}^k-u_{i,K}^{k-1})p_{i,K}^k, \quad
  I_4 = \sum_{i=1}^n\sum_{K\in\mathcal{T}}\sum_{\sigma\in\mathcal{E}_K}
  \mathcal{F}_{K,\sigma}[u_i^k,p_i^k]p_{i,K}^k.
\end{align*}
To estimate $I_3$, we distinguish between the fully implicit scheme and the mid-point method. In the former case, we use definition \eqref{3.repp} of $p_{i,K}^k$ and the symmetry $W_{KJ}^{ij}=W_{JK}^{ji}$, which follows from Hypothesis (H4), to find that
\begin{align*}
  I_3 &= \frac{1}{\Delta t}\sum_{i,j=1}^n\sum_{K,J\in\mathcal{T}}
  \m(K)\m(J)W_{KJ}^{ij}u_{j,J}^k(u_{i,K}^k-u_{i,K}^{k-1}) \\
  &= \frac{1}{2\Delta t}\sum_{i,j=1}^n\sum_{K,J\in\mathcal{T}}\m(K)\m(J)
  W_{KJ}^{ij}(u_{i,K}^ku_{j,J}^k - u_{i,K}^{k-1}u_{j,J}^{k-1}) \\
  &\phantom{xx}+ \frac{1}{2\Delta t}\sum_{i,j=1}^n
  \sum_{K,J\in\mathcal{T}}\m(K)\m(J)
  W_{KJ}^{ij}(u_{i,K}^k-u_{i,K}^{k-1})(u_{j,J}^{k}-u_{j,J}^{k-1}).
\end{align*}
By \eqref{2.Wpos}, which follows from Hypothesis (H5), the second term on the right-hand side is nonnegative, which gives
\begin{align*}
  I_3 \ge \frac{1}{\Delta t}\big(H_R(u^k)-H_R(u^{k-1})\big).
\end{align*}
For the mid-point scheme, we use definition \eqref{3.attp} of $p_{i,K}^k$:
\begin{align*}
  I_3 &= \frac{1}{2\Delta t}\sum_{i,j=1}^n
  \sum_{K,J\in\mathcal{T}}\m(K)\m(J)
  W_{KJ}^{ij}(u_{i,K}^k-u_{i,K}^{k-1})(u_{j,J}^k+u_{j,J}^{k-1}) \\
  &= \frac{1}{2\Delta t}\sum_{i,j=1}^n
  \sum_{K,J\in\mathcal{T}}\m(K)\m(J)W_{KJ}^{ij}
  \big(u_{i,K}^ku_{j,J}^k - u_{i,K}^{k-1}u_{j,J}^{k-1}\big) \\
  &\phantom{xx}+ \frac{1}{2\Delta t}\sum_{i,j=1}^n
  \sum_{K,J\in\mathcal{T}}\m(K)\m(J)
  W_{KJ}^{ij}u_{i,K}^ku_{j,J}^{k-1} \\
  &\phantom{xx}- \frac{1}{2\Delta t}\sum_{i,j=1}^n
  \sum_{K,J\in\mathcal{T}}\m(K)\m(J)
  W_{KJ}^{ij}u_{i,K}^{k-1}u_{j,J}^k.
\end{align*}
The symmetry property in Hypothesis (H4) shows that the last two terms cancel. Therefore,
\begin{align*}
  I_3 = \frac{1}{\Delta t}\big(H_R(u^k)-H_R(u^{k-1})\big).
\end{align*} 

We turn to the estimate of $I_4$. By discrete integration by parts,
\begin{align*}
  I_4 = \sum_{i=1}^n\sum_{\sigma=K|L\in\mathcal{E} }\tau_\sigma
  \big(\kappa B_{i,\sigma}^k\D_{K,\sigma}u_i^k
  + \widehat{u}_{i,\sigma}^k\D_{K,\sigma}p_i^k\big)\D_{K,\sigma}p_i^k
  = I_{41} + I_{42}. 
\end{align*}
The drift part $I_{42}$ will be used to absorb part of $I_{41}$. We rewrite the diffusion part by splitting $\kappa B_{i,\sigma}^k = \kappa + \kappa(B_{i,\sigma}^k-1)$ and using the definition of the cross term $X(u^k,p^k)$ in Theorem \ref{thm.ex}:
\begin{align*}
  I_{41} = \kappa X(u^k,p^k) 
  + \kappa\sum_{i=1}^n\sum_{\sigma=K|L\in\mathcal{E} }
  \tau_\sigma(B_{i,\sigma}^k-1)\D_{K,\sigma}u_i^k\D_{K,\sigma}p_i^k.
\end{align*}
To estimate the second term on the right-hand side, we insert $\D_{K,\sigma}p_i^k = [\D_{K,\sigma}p_i^k]^+ - [\D_{K,\sigma}p_i^k]^-$ and use
\begin{align*}
  -\D_{K,\sigma}u_i^k\D_{K,\sigma}p_i^k
  &= -(u_{i,L}^k-u_{i,K}^k)\big([\D_{K,\sigma}p_i^k]^+ 
  - [\D_{K,\sigma}p_i^k]^-\big) \\
  &\ge -u_{i,L}^k[\D_{K,\sigma}p_i^k]^+ 
  - u_{i,K}^k[\D_{K,\sigma}p_i^k]^-.
\end{align*}
Then, together with inequality 
\begin{align}\label{3.Bineq}
  0\le 1 - B_{i,\sigma}^k \le \frac{\alpha}{\kappa}\D_\sigma p_i^k,
\end{align}
which follows from Hypothesis (H4), and definition \eqref{2.uhat} of $\widehat{u}_{i,\sigma}^k$, we find that
\begin{align*}
  I_{41} &\ge \kappa X(u^k,p^k) 
  - \alpha\sum_{i=1}^n\sum_{\sigma=K|L\in\mathcal{E} }
  \tau_\sigma\D_\sigma p_i^k\big(u_{i,L}^k[\D_{K,\sigma}p_i^k]^+ 
  + u_{i,K}^k[\D_{K,\sigma}p_i^k]^-\big) \\
  &\ge \kappa X(u^k,p^k) 
  - \alpha\sum_{i=1}^n\sum_{\sigma\in\mathcal{E} }
  \tau_\sigma\widehat{u}_{i,\sigma}^k|\D_\sigma p_i^k|^2.
\end{align*}
We conclude that
\begin{align*}
  0 \ge \frac{1}{\Delta t}\big(H_R(u^k)-H_R(u^{k-1})\big)
  + (1-\alpha)\sum_{\sigma\in\mathcal{E} }\tau_\sigma
  \widehat{u}_{i,\sigma}^k|\D_\sigma p_i^k|^2
  + \kappa X(u^k,p^k),
\end{align*}
which equals the discrete Rao entropy inequality \eqref{2.HR}. 
The proof of Theorem \ref{thm.ex} is finished.


\section{Uniform estimate for the Fisher information}\label{sec.fisher}

We cannot conclude a uniform bound directly from \eqref{2.HB} for the discrete gradient $\D_\sigma(u_i^k)^{1/2}$ because of the factor $B_\kappa(\D_\sigma p_i^k)$ that is generally not bounded from below by a positive constant. Therefore, we derive first a bound for the Fisher information depending on the entropy productions $P_B$, $P_R$ and the cross term $K$. A combination of the entropy inequalities then yields the desired gradient bound.

\subsection{Fisher information}

We show an inequality for the Fisher information.

\begin{lemma}[Fisher information]\label{lem.fisher}
Let Hypotheses (H1)--(H3) hold and $0\le\alpha<1$. Then
\begin{align*}
  \kappa(1-\alpha)\sum_{i=1}^n\sum_{\sigma\in\mathcal{E}}
  \tau_\sigma|\D_\sigma(u_i^k)^{1/2}|^2
  \le \frac14 P_B(u^k,p^k) + \frac{\alpha}{\kappa}P_R(u^k,p^k) 
  - \alpha X(u^k,p^k),
\end{align*}
where $P_B$, $P_R$, and $K$ are defined in Theorem \ref{thm.ex}. 
\end{lemma}

\begin{proof}
The discrete analog of the chain rule $\na v = 2\sqrt{v}\na\sqrt{v}$ reads as
\begin{align}\label{3.chain}
  \D_{K,\sigma}u_i^k = 2(\bar{u}_{i,\sigma}^k)^{1/2}
  \D_{K,\sigma}(u_i^k)^{1/2},
\end{align}
where the power mean $\bar{u}_{i,\sigma}^k$ is defined by
\begin{align}\label{3.baru}
  \bar{u}_{i,\sigma}^k = \bigg(\frac12(u_{i,K}^k)^{1/2}
  + \frac12(u_{i,L}^k)^{1/2}\bigg)^2 \quad\mbox{for }\sigma=K|L.
\end{align}
The discrete chain rule \eqref{3.chain} follows directly from
\begin{align*}
  \D_{K,\sigma}u_i^k = \big((u_{i,L}^k)^{1/2}+(u_{i,K}^k)^{1/2}\big)
  \big((u_{i,L}^k)^{1/2}-(u_{i,K}^k)^{1/2}\big).
\end{align*}
Thus, the cross term $X(u^k,p^k)$ can be written as
\begin{align*}
  X(u^k,p^k) = 2\sum_{i=1}^n\sum_{\sigma=K|L\in\mathcal{E} }
  \tau_\sigma(\bar{u}_{i,\sigma}^k)^{1/2}\D_{K,\sigma}p_i^k
  \D_{K,\sigma}(u_i^k)^{1/2}.
\end{align*}
We define the upwind and downwind cross terms by
\begin{align*}
  X_{\rm up}(u^k,p^k) &= \sum_{i=1}^n
  \sum_{\sigma=K|L\in\mathcal{E} }\tau_\sigma
  (\widehat{u}_{i,\sigma}^{k}(p_i^k))^{1/2}\D_{K,\sigma}p_i^k
  \D_{K,\sigma}(u_i^k)^{1/2}, \\
  X_{\rm down}(u^k,p^k) &= \sum_{i=1}^n
  \sum_{\sigma=K|L\in\mathcal{E} }\tau_\sigma
  (\widehat{u}_{i,\sigma}^{k}(-p_i^k))^{1/2}\D_{K,\sigma}(-p_i^k)
  \D_{K,\sigma}(u_i^k)^{1/2}.
\end{align*}
Referring to definition \eqref{2.uhat} of the upwind concentration $\widehat{u}_{i,\sigma}^k=\widehat{u}_{i,\sigma}^k(p_i^k)$, it follows that $\widehat{u}_{i,\sigma}^{k}(-p_i^k)$ corresponds to the downwind concentration.

Splitting $\kappa B_{i,\sigma}^k = \kappa + \kappa(B_{i,\sigma}^k-1)$, we reformulate the discrete Fisher information as
\begin{align}\label{3.aux0}
  4\kappa\sum_{i=1}^n\sum_{\sigma\in\mathcal{E} }
  \tau_\sigma|\D_\sigma(u_i^k)^{1/2}|^2
  &= P_B(u^k,p^k) 
  + 4\kappa\sum_{i=1}^n\sum_{\sigma\in\mathcal{E} }
  \tau_\sigma(1-B_{i,\sigma}^k)|\D_\sigma(u_i^k)^{1/2}|^2 \\
  &\le P_B(u^k,p^k) 
  + 4\alpha\sum_{i=1}^n\sum_{\sigma\in\mathcal{E} }
  \tau_\sigma|\D_\sigma p_i^k||\D_\sigma(u_i^k)^{1/2}|^2, \nonumber 
\end{align}
where the last step follows from inequality \eqref{3.Bineq} for $B_{i,\sigma}^k$. We estimate the last term
\begin{align}\label{3.R}
  R(u^k,p^k) := \sum_{i=1}^n\sum_{\sigma\in\mathcal{E} }
  \tau_\sigma|\D_\sigma p_i^k||\D_\sigma(u_i^k)^{1/2}|^2.
\end{align}
Splitting $|\D_\sigma p_i^k|=[\D_{K,\sigma}p_i^k]^+ + [\D_{K,\sigma}p_i^k]^-$, we divide $R(u^k,p^k)$ into an upwind and a downwind part and use the property $[\D_{K,\sigma}p_i^k]^{\pm} = -[-\D_{K,\sigma}p_i^k]^{\mp}$:
\begin{align}
  R(u^k,p^k) &= \sum_{i=1}^n\sum_{\sigma=K|L\in\mathcal{E} }
  \tau_\sigma\big([\D_{K,\sigma}p_i^k]^+ 
  + [\D_{K,\sigma}p_i^k]^-\big)
  \big((u_{i,L}^k)^{1/2}-(u_{i,K}^k)^{1/2}\big)
  \D_{K,\sigma}(u_i^k)^{1/2} \nonumber \\
  &= \sum_{i=1}^n\sum_{\sigma=K|L\in\mathcal{E} }
  \tau_\sigma\big([\D_{K,\sigma}p_i^k]^+(u_{i,L}^k)^{1/2}
  - [\D_{K,\sigma}p_i^k]^-(u_{i,K}^k)^{1/2}\big)\D_{K,\sigma}
  (u_{i}^k)^{1/2} \label{3.aux} \\
  &\phantom{xx}+ \sum_{i=1}^n\sum_{\sigma=K|L\in\mathcal{E} }
  \tau_\sigma\big([\D_{K,\sigma}p_i^k]^-(u_{i,L}^k)^{1/2}
  - [\D_{K,\sigma}p_i^k]^+(u_{i,K}^k)^{1/2}\big)
  \D_{K,\sigma}(u_{i}^k)^{1/2} \nonumber \\
  &= \sum_{i=1}^n\sum_{\sigma=K|L\in\mathcal{E} }
  \tau_\sigma\big([\D_{K,\sigma}p_i^k]^+(u_{i,L}^k)^{1/2}
  - [\D_{K,\sigma}p_i^k]^-(u_{i,K}^k)^{1/2}\big)\D_{K,\sigma}
  (u_{i}^k)^{1/2} \nonumber \\
  &\phantom{xx}- \sum_{i=1}^n\sum_{\sigma=K|L\in\mathcal{E} }
  \tau_\sigma\big([-\D_{K,\sigma}p_i^k]^+(u_{i,L}^k)^{1/2}
  - [-\D_{K,\sigma}p_i^k]^-(u_{i,K}^k)^{1/2}\big)
  \D_{K,\sigma}(u_{i}^k)^{1/2} \nonumber \\
  &= \sum_{i=1}^n\sum_{\sigma=K|L\in\mathcal{E} }\tau_\sigma
  \big((\widehat{u}_{i,\sigma}^{k}(p_i^k))^{1/2}\D_{K,\sigma}p_i^k
  - (\widehat{u}_{i,\sigma}^{k}(-p_i^k))^{1/2}\D_{K,\sigma}(-p_i^k)\big)
  \D_{K,\sigma}(u_{i}^k)^{1/2} \nonumber \\
  &= X_{\rm up}(u^k,p^k) - X_{\rm down}(u^k,p^k). \nonumber 
\end{align} 
We need to estimate the upwind and downwind cross terms. First, by the Cauchy--Schwarz inequality,
\begin{align}\label{3.Cup}
  X_{\rm up}(u^k,p^k) \le \frac{\kappa}{2}\sum_{i=1}^n
  \sum_{\sigma\in\mathcal{E} }\tau_\sigma
  |\D_\sigma(u_i^k)^{1/2}|^2 + \frac{1}{2\kappa}P_R(u^k,p^k).
\end{align}
To estimate the downwind cross term, we apply the identity $a_1b_1-a_2b_2 = \frac12(a_1-a_2)(b_1+b_2) + \frac12(a_1+a_2)(b_1-b_2)$ to $a_1 = (u_{i,L}^k)^{1/2}$, $a_2 = (u_{i,K}^k)^{1/2}$, $b_1 = [\D_{K,\sigma}p_i^k]^-$, $b_2 = [\D_{K,\sigma}p_i^k]^+$, which yields
\begin{align*}
  [\D_{K,\sigma}p_i^k]^-&(u_{i,L}^k)^{1/2}
  - [\D_{K,\sigma}p_i^k]^+(u_{i,K}^k)^{1/2} \\
  &= \frac12|\D_{K,\sigma}p_i^k|\D_{K,\sigma}(u_i^k)^{1/2}
  - \frac12\D_{K,\sigma}p_i^k
  \big((u_{i,K}^k)^{1/2}+(u_{i,L}^k)^{1/2}\big).
\end{align*}
We insert this identity into the downwind cross term:
\begin{align*}
  -X_{\rm down}(u^k,p^k) &= \sum_{i=1}^n
  \sum_{\sigma=K|L\in\mathcal{E} }
  \tau_\sigma\big([\D_{K,\sigma}p_i^k]^-(u_{i,L}^k)^{1/2}
  - [\D_{K,\sigma}p_i^k]^+(u_{i,K}^k)^{1/2}\big)
  \D_{K,\sigma}(u_{i}^k)^{1/2} \\
  &= \frac12\sum_{i=1}^n\sum_{\sigma=K|L\in\mathcal{E} }
  \tau_\sigma|\D_{K,\sigma}p_i^k||\D_{K,\sigma}(u_i^k)^{1/2}|^2 \\
  &\phantom{xx}- \frac12\sum_{i=1}^n
  \sum_{\sigma=K|L\in\mathcal{E} }\tau_\sigma
  \D_{K,\sigma}p_i^k\big((u_{i,K}^k)^{1/2}+(u_{i,L}^k)^{1/2}\big)
  \D_{K,\sigma}(u_{i}^k)^{1/2} \\
  &= \frac12\sum_{i=1}^n\sum_{\sigma=K|L\in\mathcal{E} }
  \tau_\sigma|\D_{\sigma}p_i^k||\D_{K,\sigma}(u_i^k)^{1/2}|^2 
  - \frac12\sum_{i=1}^n\sum_{\sigma=K|L\in\mathcal{E} }\tau_\sigma
  \D_{K,\sigma}p_i^k\D_{K,\sigma}u_i^k,
\end{align*}
where we used {\color{blue} use} in the last step the identity
\begin{align*}
  \big((u_{i,K}^k)^{1/2}+(u_{i,L}^k)^{1/2}\big)
  \D_{K,\sigma}(u_{i}^k)^{1/2}
  = u_{i,L}^k - u_{i,K}^k = \D_{K,\sigma}u_i^k.
\end{align*}
We infer from definition \eqref{3.R} of $R(u^k,p^k)$ that
\begin{align*}
  -X_{\rm down}(u^k,p^k) = \frac12 R(u^k,p^k) - \frac12 X(u^k,p^k).
\end{align*} 
Substituting \eqref{3.Cup} and the previous expression into \eqref{3.aux} yields
\begin{align*}
  R(u^k,p^k) \le \kappa\sum_{i=1}^n
  \sum_{\sigma\in\mathcal{E} }\tau_\sigma
  |\D_\sigma(u_i^k)^{1/2}|^2 + \frac{1}{\kappa}P_R(u^k,p^k)
  - X(u^k,p^k).
\end{align*}
Finally, we insert this inequality into \eqref{3.aux0}, written as
\begin{align*}
  \kappa\sum_{i=1}^n\sum_{\sigma\in\mathcal{E} }
  \tau_\sigma|\D_\sigma(u_i^k)^{1/2}|^2
  \le \frac14 P_B(u^k,p^k) + \alpha R(u^k,p^k),
\end{align*}
to conclude the proof.
\end{proof}


\subsection{Estimate for the fully implicit scheme}

First, we prove a discrete analog of the differentiation rule $\na(B*u)=B*\na u$.

\begin{lemma}\label{lem.diff}
Let $u_i=(u_{i,K})_{K\in\mathcal{T}}$ be given and let $p_i=(p_{i,K})_{K\in\mathcal{T}}$ be defined by 
\begin{align*}
  p_{i,K} = \sum_{j=1}^n\sum_{J\in\mathcal{T}}\m(J)W_{KJ}^{ij}
  u_{j,J}.
\end{align*}
Then, for any $\ell\in\{\pm 1,\ldots,\pm d\}$,
\begin{align*}
  \D_{K,\ell} p_i = \sum_{j=1}^n\sum_{J\in\mathcal{T}}\m(J)
  W_{KJ}^{ij}\D_{J,\ell}u_j.
\end{align*}
\end{lemma}

\begin{proof}
Recall the definition $\D_{K,\ell}u_i=\D_{K,\sigma}u_i$ for $\sigma=K|L\in\mathcal{E} $ such that $x_L=x_K+\Delta x_\ell e_\ell$, where $e_\ell$ is the $\ell$th canonical unit vector of $\R^d$ (see \eqref{2.DKell}). We compute, using definition \eqref{2.Wij} of $W_{KJ}^{ij}$ and the periodic boundary conditions, 
\begin{align*}
  \D_{K,\ell} p_i &= \sum_{j=1}^n\sum_{J\in\mathcal{T}}\m(J)
  (W_{LJ}^{ij}-W_{KJ}^{ij})u_{j,J} \\
  &= \sum_{j=1}^n\sum_{J\in\mathcal{T}}
  \bigg(\frac{1}{\m(L)}\int_L\int_J W_{ij}(x-y)\dd y\dd x
  - \frac{1}{\m(K)}\int_K\int_J W_{ij}(x-y)\dd y\dd x\bigg)u_{j,J} \\
  &= \sum_{j=1}^n\frac{1}{\m(K)}
  \int_K\int_{\T^d}\big(W_{ij}(x+\Delta x_\ell e_\ell-y)
  - W_{ij}(x-y)\big)\pi_\delta u_j(y)\dd y\dd x \\
  &= \sum_{j=1}^n\int_{\T^d}\frac{1}{\m(K)}\int_K W_{ij}(x-y)\big(
  \pi_\delta u_j(y+\Delta x_\ell e_\ell) - \pi_\delta u_j(y)
  \big)\dd y\dd x \\
  &= \sum_{j=1}^n\sum_{J\in\mathcal{T}}\m(J)W_{KJ}^{ij}
  \D_{J,\ell}u_{j}.
\end{align*}
This finishes the proof.
\end{proof}

We claim that the discrete Fisher information is uniformly bounded.

\begin{lemma}[Discrete gradient bound for the fully implicit scheme]\label{lem.fisher1}
Let $p_i^k$ be given by \eqref{3.repp}. Then there exists $C(u^0)>0$ depending on the initial data such that
\begin{align*}
  \kappa\sum_{k=1}^{N}\Delta t\sum_{i=1}^n\sum_{\sigma\in\mathcal{E}}
  \tau_\sigma|\D_\sigma(u_i^k)^{1/2}|^2 \le C(u^0).
\end{align*}
\end{lemma}

\begin{proof}
We multiply the discrete Boltzmann entropy inequality \eqref{2.HB} by $1-\alpha$ and the discrete Rao entropy inequality \eqref{2.HR} by $\alpha/\kappa$ and add both inequalities:
\begin{align}\label{3.aux2}
  \frac{1-\alpha}{\Delta t}&\big(H_B(u^k)-H_B(u^{k-1})\big)
  + \frac{\alpha}{\kappa\Delta t}\big(H_R(u^k)-H_R(u^{k-1})\big) \\
  &\le -(1-\alpha)\bigg(P_B(u^k,p^k) 
  + \frac{\alpha}{\kappa}P_R(u^k,p^k)\bigg) - X(u^k,p^k). \nonumber 
\end{align}
Lemma \ref{lem.fisher} implies that
\begin{align*}
  P_B(u^k,p^k) + \frac{\alpha}{\kappa}P_R(u^k,p^k)
  &\ge \frac14 P_B(u^k,p^k) 
  + \frac{\alpha}{\kappa}P_R(u^k,p^k) \\
  &\ge \kappa(1-\alpha)\sum_{i=1}^n
  \sum_{\sigma\in\mathcal{E} }\tau_\sigma
  |\D_\sigma(u_i^k)^{1/2}|^2 + \alpha X(u^k,p^k).
\end{align*}
Inserting this estimate into \eqref{3.aux2} yields
\begin{align}\label{3.aux22}
  \frac{1-\alpha}{\Delta t}&\big(H_B(u^k)-H_B(u^{k-1})\big)
  + \frac{\alpha}{\kappa\Delta t}\big(H_R(u^k)-H_R(u^{k-1})\big) \\
  &\le -\kappa(1-\alpha)^2\sum_{i=1}^n
  \sum_{\sigma\in\mathcal{E} }\tau_\sigma
  |\D_\sigma(u_i^k)^{1/2}|^2 - \big(\alpha(1-\alpha)+1\big)X(u^k,p^k).
  \nonumber 
\end{align}
Observe that $\alpha(1-\alpha)+1 > 0$. We claim that also $X(u^k,p^k)$ is nonnegative. Indeed, we use the fact that the mesh is uniform and apply Lemma \ref{lem.diff}:
\begin{align*}
  X(u^k,p^k) &= \frac12\sum_{i,j=1}^n\sum_{K\in\mathcal{T}}
  \sum_{|\ell|=1}^d\frac{\m(K)}{(\Delta x_\ell)^2}
  \D_{K,\ell}p_i^k\D_{K,\ell}u_i^k \\
  &= \frac12\sum_{i,j=1}^n\sum_{|\ell|=1}^d\sum_{K,J\in\mathcal{T}}
  \frac{\m(K)}{\Delta x_\ell}\frac{\m(J)}{\Delta x_\ell}
  W_{KJ}^{ij}\D_{K,\ell}u_i^k\D_{J,\ell}u_j^k \ge 0,
\end{align*}
where the inequality follows from \eqref{2.Wpos}. Thus, summing \eqref{3.aux22} over $k=1,\ldots,N$,
\begin{align*}
  (1-\alpha)&(H_B(u^{N})-H_B(u^0))
  + \frac{\alpha}{\kappa}(H_R(u^{N})-H_R(u^0)) \\
  &+ \kappa(1-\alpha)^2\sum_{k=1}^N\Delta t\sum_{i=1}^n
  \sum_{\sigma\in\mathcal{E}}\tau_\sigma|\D_\sigma(u_i^k)^{1/2}|^2 
  \le 0.
\end{align*}
Since $H_B(u^{N})\ge 0$, $H_R(u^{N})\ge 0$, and $\alpha<1$, this finishes the proof.
\end{proof}


\subsection{Estimate for the mid-point scheme}

We show first an auxiliary estimate of the cross term.

\begin{lemma}\label{lem.estC} 
Let $u_i^k=(u_{i,K}^k)_{K\in\mathcal{T}}$ and let $p_i^k=(p_{i,K}^k)_{K\in\mathcal{T}}$ be defined by \eqref{3.attp}. Then
\begin{align*}
  X(u^k,p^k) \le c^*
  \sum_{i=1}^n\sum_{\sigma\in\mathcal{E}}
  \tau_\sigma\big(3|\D_\sigma(u_i^k)^{1/2}|^2
  + |\D_\sigma(u_i^{k-1})^{1/2}|^2\big),
\end{align*}
where $X(u^k,p^k)$ is defined in Theorem \ref{thm.ex} and $c^*>0$ is given by 
\begin{align}\label{3.cstar}
  c^*=\max_{j=1,\ldots,n}\sum_{i=1}^n \|W_{ij}\|_{L^\infty(\T^d)} \|u_i^0\|_{L^1(\T^d)}.
\end{align} 
\end{lemma}

\begin{proof}
Let $\sigma=K|J\in\mathcal{E} $ with $x_J=x_K+\Delta x_\ell e_\ell$. Since the mesh consists of hyper-rectangles, we have $\tau_\sigma\m(J) = \m(\sigma)^2$. Then, applying Young's inequality,
\begin{align*}
  X(u^k,p^k) 
  &= \frac12\sum_{|\ell|=1}^d\sum_{i,j=1}^n\sum_{K,J\in\mathcal{T}}
  \m(\sigma)^2W_{KJ}^{ij}\D_{K,\ell}u_i^k
  \frac{\D_{J,\ell}u_j^k+\D_{J,\ell}u_j^{k-1}}{2} \\
  &\le \frac18\sum_{|\ell|=1}^d\sum_{i,j=1}^n\sum_{K,J\in\mathcal{T}}
  \m(\sigma)^2W_{KJ}^{ij}\big(2|\D_{K,\ell}u_i^k|^2
  + |\D_{J,\ell}u_j^k|^2 + |\D_{J,\ell}u_j^{k-1}|^2\big).
\end{align*}
It follows from the symmetry $W_{KJ}^{ij}=W_{JK}^{ji}$ that
\begin{align}\label{3.aux3}
  X(u^k,p^k) &\le \frac38\sum_{|\ell|=1}^d\bigg(\sum_{i=1}^n
  \sum_{K\in\mathcal{T}}\m(\sigma)\max_{j,J}(W_{KJ}^{ij})^{1/2}
  |\D_{K,\ell}u_i^k|\bigg)^2 \\
  &\phantom{xx}+ \frac18\sum_{|\ell|=1}^d\bigg(\sum_{i=1}^n
  \sum_{K\in\mathcal{T}}\m(\sigma)\max_{j,J}(W_{KJ}^{ij})^{1/2}
  |\D_{K,\ell}u_i^{k-1}|\bigg)^2. \nonumber 
\end{align}
Definition \eqref{3.baru} of $\bar{u}_{i,\sigma}^k$ gives 
$\D_{K,\ell}u_i^k =  2(\bar{u}_{i,\sigma}^k)^{1/2} \D_{K,\ell}(u_i^k)^{1/2}$. Then the Cauchy--Schwarz inequality and the identity $\m(\sigma)\dd_\sigma=\m(K)$ show that
\begin{align*}
  \bigg(&\sum_{i=1}^n\sum_{K\in\mathcal{T}}
  \m(\sigma)\max_{j,J}(W_{KJ}^{ij})^{1/2}
  |\D_{K,\ell}u_i^k|\bigg)^2 \\
  &= 4\bigg(\sum_{i=1}^n\sum_{K\in\mathcal{T}}\bigg\{
  \m(\sigma)^{1/2}\dd_\sigma^{1/2}\max_{j,J}(W_{KJ}^{ij})^{1/2}
  (\bar{u}_{i,\sigma}^k)^{1/2}\bigg\}
  \bigg\{\frac{\m(\sigma)^{1/2}}{\dd_\sigma^{1/2}}
  |\D_{K,\ell}(u_i^k)^{1/2}|\bigg\}\bigg)^2 \\
  &\le 4\max_{j,J}\bigg(\sum_{i=1}^n\sum_{K\in\mathcal{T}}\m(K)
  W_{KJ}^{ij}\bar{u}_{i,\sigma}^k\bigg)
  \bigg(\sum_{i=1}^n\sum_{K\in\mathcal{T}}\tau_\sigma
  |\D_{K,\ell}(u_i^k)^{1/2}|^2\bigg) \\
  &\le 4\max_{j=1,\ldots,n}\sum_{i=1}^n\|W_{ij}\|_{L^\infty(\T^d)}
  \|\pi_\delta\bar{u}_i^k\|_{L^1(\T^d)}
  \sum_{i=1}^n\sum_{K\in\mathcal{T}}\tau_\sigma
  |\D_{K,\ell}(u_i^k)^{1/2}|^2 \\
  &\le 4c^*\sum_{i=1}^n\sum_{K\in\mathcal{T}}\tau_\sigma
  |\D_{K,\ell}(u_i^k)^{1/2}|^2,
\end{align*}
where we used mass conservation and definition \eqref{3.cstar} of $c^*$ in the last step. Similarly,
\begin{align*}
  \bigg(\sum_{i=1}^n
  \sum_{K\in\mathcal{T}}\m(\sigma)\max_{j,J}(W_{KJ}^{ij})^{1/2}
  |\D_{K,\ell}u_i^{k-1}|\bigg)^2
  \le 4c^*\sum_{i=1}^n\sum_{K\in\mathcal{K}}\tau_\sigma
  |\D_{K,\ell}(u_i^{k-1})^{1/2}|^2.
\end{align*}
Inserting the previous two estimates into \eqref{3.aux3} and applying a symmetrization argument prove the lemma.
\end{proof}

Now, we show the desired discrete gradient estimate.

\begin{lemma}[Discrete gradient bound for the mid-point scheme]
\label{lem.fisher2}
Under the assumptions of Theorem \ref{thm.att}, there exists $C(u^0)>0$ depending on the initial data (and $\kappa$) such that
\begin{align*}
  \sum_{k=1}^{N}\Delta t\sum_{i=1}^n\sum_{\sigma\in\mathcal{E}}
  \tau_\sigma|\D_\sigma(u_i^k)^{1/2}|^2 \le C(u^0).
\end{align*}
\end{lemma}

\begin{proof}
Set $\beta=\alpha(1-\alpha)+1$. We sum estimate \eqref{3.aux22} over $k=1,\ldots,N$ and replace $X(u^k,p^k)$ by the estimate in Lemma \ref{lem.estC}:
\begin{align}\label{3.aux4}
  (1-\alpha)&\big(H_B(u^{N})-H_B(u^0)\big)
  + \frac{\alpha}{\kappa}\big(H_R(u^{N})-H_R(u^0)\big) \\
  &\le -\kappa(1-\alpha)^2\sum_{k=1}^{N}\Delta t
  \sum_{i=1}^n\sum_{\sigma\in\mathcal{E}}
  \tau_\sigma|\D_\sigma(u_i^k)^{1/2}|^2 
  - \beta\sum_{k=1}^{N}\Delta t\sum_{k=1}^{N}X(u^k,p^k) \nonumber \\
  &\le -\bigg(\kappa(1-\alpha)^2 - 4\beta c^*\bigg)
  \sum_{k=1}^{N}\Delta t\sum_{i=1}^n\sum_{\sigma\in\mathcal{E}}
  \tau_\sigma|\D_\sigma(u_i^k)^{1/2}|^2 \nonumber \\
  &\phantom{xx}+ 
  \beta c^*\Delta t\sum_{i=1}^n\sum_{\sigma\in\mathcal{E}}
  \tau_\sigma|\D_\sigma(u_i^0)^{1/2}|^2.
  \nonumber 
\end{align}
We bound the Rao entropy $H_R(u^{N})$, defined in \eqref{2.HBHR}, as follows:
\begin{align*}
  |H_R(u^{N})| &\le \max_{i,j=1,\ldots,n}\|W_{ij}\|_{L^\infty(\T^d)}
  \bigg(\sum_{i=1}^n\sum_{K\in\mathcal{T}}\m(K)u_{i,K}^{N}\bigg)^2 \\
  &= \max_{i,j=1,\ldots,n}\|W_{ij}\|_{L^\infty(\T^d)}
  \bigg(\sum_{i=1}^n\|u_i^0\|_{L^1(\T^d)}\bigg)^2 \le C(u^0).
\end{align*}
If the initial data satisfy \eqref{assume: derivative of initial data}, we conclude from \eqref{3.aux4} and the nonnegativity of $H_B(u^{N})$ that
\begin{align*}
  \big(\kappa(1-\alpha)^2-4\beta c^*\big)\sum_{k=1}^{N}\Delta t
  \sum_{i=1}^n\sum_{\sigma\in\mathcal{E}}\tau_\sigma
  |\D_\sigma(u_i^k)^{1/2}|^2 \le C(u^0).
\end{align*}
It holds that $\kappa(1-\alpha)^2-4\beta c^*>0$ if $c^*<\kappa(1-\alpha)^2/(4\beta)$. This proves the claim under condition \eqref{assume: derivative of initial data}.

If assumption \eqref{assume: derivative of initial data} is not satisfied, we can bound the last term in \eqref{3.aux4} as follows. Because of $\m(\sigma)=\m(K)/\dd_\sigma$ and $\dd_\sigma\ge Ch$, we have
\begin{align*}
  \sum_{\sigma\in\mathcal{E}}\tau_\sigma|\D_\sigma(u_i^0)^{1/2}|^2
  &= \sum_{K\in\mathcal{T}}\sum_{\sigma=K|L\in\mathcal{E}_K}
  \frac{\m(K)}{\dd_\sigma^2}
  \big|(u_{i,L}^0)^{1/2}-(u_{i,K}^0)^{1/2}\big|^2 \\
  &\le \frac{C}{h^2}\sum_{K\in\mathcal{T}}\m(K)u_{i,K}^0
  = \frac{C}{h^2}\|u_i^0\|_{L^1(\T^d)}.
\end{align*}
If $\Delta t\le Ch^2$, we obtain
\begin{align*}
  \Delta t\sum_{\sigma\in\mathcal{E}}
  \tau_\sigma|\D_\sigma(u_i^0)^{1/2}|^2
  \le C(u^0).
\end{align*}
The proof is finished.
\end{proof}


\section{Uniform estimates, compactness, and convergence}\label{sec.est}

We prove further uniform estimates by leveraging the previously derived uniform bound on the Fisher information. We then apply a discrete compactness argument to deduce the convergence.

\subsection{Uniform estimates}

Let $u^k=(u_K^k)_{K\in\mathcal{T}}$ for $k=0,\ldots,N$ be a solution to scheme \eqref{2.euler}--\eqref{2.uhat} with the potential $p^k=(p_K^k)_{K\in\mathcal{T}}$ defined in \eqref{3.repp} or \eqref{3.attp}. Recall definition \eqref{2.delta} of the mesh size $\delta$. The mass conservation and the discrete gradient bound in Lemmas \ref{lem.fisher1} and \ref{lem.fisher2} give the following result.

\begin{lemma}\label{lem.r1}
Let $r_1=(d+2)/d$. Then there exists a constant $C>0$ independent of the mesh size $\delta$ such that for $i=1,\ldots,n$,
\begin{align}
  \max_{k=1,\ldots,N}\|u_i^k\|_{0,1,\mathcal{T}}
  + \sum_{k=1}^N\Delta t\|(u_i^k)^{1/2}\|_{1,2,\mathcal{T}}^2 
  &\le C, \label{5.est1} \\ 
  \sum_{k=1}^N\Delta t\big(\|u_i^k\|_{0,r_1,\mathcal{T}}^{r_1}
  + \|\widehat{u}_i^k\|_{0,r_1,\mathcal{T}}^{r_1}
  + \|\bar{u}_i^k\|_{0,r_1,\mathcal{T}}^{r_1}\big) &\le C. 
  \label{5.est2}
\end{align}
\end{lemma}

\begin{proof}
Estimate \eqref{5.est1} immediately follows from mass conservation and 
Lemmas \ref{lem.fisher1} and \ref{lem.fisher2}. We claim that estimate \eqref{5.est2} is a consequence of the discrete Gagliardo--Nirenberg inequality \cite[Theorem 3.4]{BCF15}. Indeed, starting from the inequality
\begin{align*}
  \|(u_i^k)^{1/2}\|_{0,2r_1,\mathcal{T}}
  \le C\|(u_i^k)^{1/2}\|_{1,2,\mathcal{T}}^{d/(d+2)}
  \|(u_i^k)^{1/2}\|_{0,2,\mathcal{T}}^{2/(d+2)}
\end{align*}
we sum over $k=1,\ldots,N$ yielding
\begin{align*}
  \sum_{k=1}^N\Delta t\|(u_i^k)^{1/2}\|_{0,2r_1,\mathcal{T}}^{2r_1}
  \le C\max_{k=1,\ldots,N}\|u_i^k\|_{0,1,\mathcal{T}}^{2/d}
  \sum_{k=1}^N\Delta t\|(u_i^k)^{1/2}\|_{1,2,\mathcal{T}}^2 \le C.
\end{align*}

To prove the bound for the upwind concentration $\widehat{u}_i^k$, we note first that for any $x\in\Delta_\sigma$ with $\sigma=K|L$ and $t\in(t_{k-1},t_k]$, we have $\pi_\delta^*(\widehat{u}_i)^{1/2}(t,x) \le (u_{i,K}^k)^{1/2} + (u_{i,L}^k)^{1/2}$. Then an integration leads to
\begin{align*}
  \int_0^T\int_{\T^d}|\pi_\delta^*(\widehat{u}_i)^{1/2}|^{2r_1}
  \dd x\dd t
  &\le \sum_{k=1}^N\Delta t\sum_{\sigma=K|L\in\mathcal{E}}
  \m(\Delta_\sigma)\big((u_{i,K}^k)^{1/2} 
  + (u_{i,L}^k)^{1/2}\big)^{2r_1} \\
  &\le C\sum_{k=1}^N\Delta t\sum_{K\in\mathcal{T}}\m(K)|u_{i,K}^k|^{r_1}
  \le C.
\end{align*}
A similar computation holds for $\bar{u}_i^k$. This proves the lemma.
\end{proof}

We need a further gradient estimate.

\begin{lemma}[Gradient bounds]\label{lem.r2}
Let $r_2=(d+2)/(d+1)$. Then there exists a constant $C>0$ independent of the mesh size $\delta$ such that for $i=1,\ldots,n$,
\begin{align*}
  \sum_{k=1}^N\Delta t\big(\|\na^h u_i^k\|_{0,r_2,\mathcal{T}}^{r_2}
  + \|\widehat{u}_i^k\na^h p_i^k\|_{0,r_2,\mathcal{T}^*}^{r_2}\big)\le C.
\end{align*}
\end{lemma}

\begin{proof}
We use the chain rule \eqref{3.chain} and H\"older's inequality with $1/(2r_1)+1/2=1/r_2$ to estimate the diffusion term:
\begin{align*}
  \sum_{k=1}^N&\Delta t\|\na^h u_i^k\|_{0,r_2,\mathcal{T}^*}^{r_2}
  \le 2^{r_2}\sum_{k=1}^N\Delta t\|(\bar{u}_i^k)^{1/2}
  \na^h(u_i^k)^{1/2}\|_{0,r_2,\mathcal{T}^*}^{r_2} \\
  &\le C\sum_{k=1}^N\Delta t 
  \|(\bar{u}_i^k)^{1/2}\|_{0,2r_1,\mathcal{T}^*}^{r_2}
  \|\na^h(u_i^k)^{1/2}\|_{0,2,\mathcal{T}^*}^{r_2} \\
  &\le C\bigg(\sum_{k=1}^N\Delta t 
  \|(\bar{u}_i^k)^{1/2}\|_{0,2r_1,\mathcal{T}^*}^{2r_1}
  \bigg)^{r_2/(2r_1)}\bigg(\sum_{k=1}^N\Delta t\|\na^h(u_i^k)^{1/2}
  \|_{0,2,\mathcal{T}^*}^2\bigg)^{r_2/2} \le C,
\end{align*}
where the last step follows from Lemma \ref{lem.r1}. We estimate the drift term in a similar way:
\begin{align*}
  \sum_{k=1}^N\Delta t\|\widehat{u}_i^k\na^h p_i^k
  \|_{0,r_2,\mathcal{T}}^{r_2}
  &\le \bigg(\sum_{k=1}^N\Delta t\|(\widehat{u}_i^k)^{1/2}
  \|_{0,2r_1,\mathcal{T}^*}^{2r_1}\bigg)^{r_2/(2r_1)} \\
  &\times\bigg(\sum_{k=1}^N\Delta t\|(\widehat{u}_i^k)^{1/2}
  \na^h p_i^k\|_{0,r_2,\mathcal{T}^*}^2\bigg)^{r_2/2} \le C,
\end{align*}
again applying Lemma \ref{lem.r1} in the last step.
\end{proof}

\begin{lemma}[Bounds for the potential]\label{lem.p}
There exists a constant $C>0$ independent of the mesh size $\delta$ such that for $i=1,\ldots,n$,
\begin{align*}
  \sum_{k=1}^N\Delta t\big(\|p_i^k\|_{0,r_1,\mathcal{T}}^{r_1}
  + \|\na^h p_i^k\|_{0,r_2,\mathcal{T}^*}^{r_2}\big) \le C.
\end{align*}
\end{lemma}

\begin{proof}
We prove the result only for the fully implicit scheme \eqref{3.repp}, as the mid-point scheme \eqref{3.attp} is shown in an analogous way. It follows from definition \eqref{2.Wij} of $W_{KJ}^{ij}$ that
\begin{align*}
  p_{i,K}^k = \sum_{j=1}^n\sum_{J\in\mathcal{T}}\m(J)W_{KJ}^{ij}
  u_{j,J}^k 
  = \sum_{j=1}^n\int_{\T^d}
  W_{K,y}^{ij}  \pi_\delta u_j^k(y)\dd y,
\end{align*}
where $W_{K,y}^{ij}:=\m(K)^{-1}\int_KW_{ij}(x-y)\dd x$. Consequently,
\begin{align}\label{5.aux}
  \pi_\delta p_i^k(x) = \sum_{j=1}^n\int_{\T^d}(\pi_\delta  
  W_{K,y}^{ij})(x)\pi_\delta u_j^k(y)\dd y\quad\mbox{for }x\in K.
\end{align}
The function $(x,y)\mapsto(\pi_\delta W_{K,y}^{ij})(x)$ is again a kernel, but the integral is strictly speaking not a convolution. We use the following version of the Young convolution inequality: Let $v\in L^q(\T^d)$ for $1\le q\le \infty$ and let $W=W(x,y)$ satisfy
\begin{align*}
  \sup_{y\in\T^d}\|W(\cdot,y)\|_{L^1(\T^d)} < \infty, \quad
  \sup_{x\in\T^d}\|W(x,\cdot)\|_{L^1(\T^d)} < \infty.
\end{align*}
Then
\begin{align*}
  \bigg\|\int_{\T^d}W(\cdot,y)v(y)\dd y\bigg\|_{L^q\T^d)}
  \le \bigg(\sup_{y\in\T^d}\|W(\cdot,y)\|_{L^1(\T^d)}
  \sup_{x\in\T^d}\|W(x,\cdot)\|_{L^1(\T^d)}\bigg)\|v\|_{L^q(\T^d)}.
\end{align*}
The proof follows directly from H\"older's inequality and is thus omitted. We apply this result to \eqref{5.aux} to find that
\begin{align*}
  \|\pi_\delta p_i^k\|_{L^{r_1}(\T^d)} 
  \le C\sum_{j=1}^n\|\pi_\delta u_j^k\|_{L^{r_1}(\T^d)},
\end{align*}
where the constant $C>0$ depends on $\|W_{ij}\|_{L^1(\T^d)}$ but is independent of $\delta$. A summation over $k=1,\ldots,N$ leads to
\begin{align*}
  \sum_{k=1}^N\Delta t\|p_i^k\|_{0,r_1,\mathcal{T}}^{r_1}
  = \sum_{k=1}^N\Delta t\|\pi_\delta p_i^k\|_{L^{r_1}(\T^d)}^{r_1}
  \le C\sum_{j=1}^n\sum_{k=1}^N\Delta t\|u_j^k\|_{0,r_1,\mathcal{T}}
  \le C,
\end{align*}
where we use the bound from Lemma \ref{lem.r1}. The remaining estimate follows from
\begin{align*}
  \sum_{k=1}^N\Delta t\|\na^h p_i^k\|_{0,r_2,\mathcal{T}^*}^{r_2}
  = \sum_{k=1}^N\Delta t\|\pi_\delta^*\na^h p_i^k
  \|_{L^{r_2}(\T^d)}^{r_2}
  \le C\sum_{j=1}^n\sum_{k=1}^N\Delta t\|\pi_\delta \na^h u_j^k
  \|_{L^{r_2}(\T^d)}^{r_2}
\end{align*}
and Lemma \ref{lem.r2}. 
\end{proof}

It remains to derive a uniform estimate for the discrete time derivative.

\begin{lemma}[Bound for the discrete time derivative]\label{lem.time}
There exists a constant $C>0$ independent of the mesh size $\delta$ such that for $i=1,\ldots,n$,
\begin{align*}
  \sum_{k=1}^N\Delta t\|\pa_t^{\Delta t}u_i^k\|_{-1,r_2,\mathcal{T}}
  \le C,
\end{align*}
where $\pa_t^{\Delta t}u_i^k$ is defined in \eqref{2.ut} and $r_2=(d+2)/(d+1)$. 
\end{lemma}

\begin{proof}
Let $\phi\in C_0^\infty(\T^d)$ and set $\phi_K=\phi(x_K)$ for $K\in\mathcal{T}$. We multiply \eqref{2.euler} by $\phi_K$ and integrate by parts:
\begin{align*}
  0 &= \sum_{K\in\mathcal{T}}\frac{\m(K)}{\Delta t}
  (u_{i,K}^k-u_{i,K}^{k-1})\phi_K
  + \sum_{\sigma\in\mathcal{E}}\tau_\sigma
  \big(\kappa B_{i,\sigma}^k\D_{K,\sigma}u_i^k
  + \widehat{u}_{i,\sigma}^k\D_{K,\sigma}p_i^k\big)\D_{K,\sigma}\phi,
\end{align*}
where we identify the functions $\phi=\phi(x)$ and $\phi=(\phi_K)_{K\in\mathcal{T}}$ and recall that $B_{i,\sigma}^k=B(\kappa^{-1}\D_\sigma p_i^k)$. It follows from \eqref{2.Deltasig} and \eqref{2.nahv} that
\begin{align*}
  \tau_\sigma\D_{K,\sigma} u_i^k\D_{K,\sigma}\phi
  = \frac{1}{d}\m(\Delta_\sigma)\na_\sigma^h u_i^k\cdot\na_\sigma^h\phi,
\end{align*}
from which we infer that
\begin{align}\label{5.euler}
  \sum_{K\in\mathcal{T}}\frac{\m(K)}{\Delta t}
  (u_{i,K}^k-u_{i,K}^{k-1})\phi_K
  = -\frac{1}{d}\sum_{\sigma\in\mathcal{E}}\m(\Delta_\sigma)
  \big(\kappa B_{i,\sigma}^k\na_\sigma^h u_i^k
  + \widehat{u}_{i,\sigma}^k\na_\sigma^h p_i^k\big)
  \cdot\na_\sigma^h\phi.
\end{align}
We conclude that
\begin{align*}
  \bigg|\sum_{K\in\mathcal{T}}\m(K)\pa_t^{\Delta t}u_{i,K}^k\phi_K\bigg|
  \le C\big(\|\na^h u_i^k\|_{0,r_2,\mathcal{T}^*}
  + \|\widehat{u}_i^k\na^h p_i^k\|_{0,r_2,\mathcal{T}^*}\big)
  \|\na^h\phi\|_{0,r_2',\mathcal{T}^*},
\end{align*}
where $1/r_2+1/r_2'=1$. After summing over $k=1,\ldots,N$, using Lemma \ref{lem.r2}, and taking the supremum over all $\phi$, we obtain the desired bound.
\end{proof}


\subsection{Compactness}

Let $u=(u_1,\ldots,u_n)$ be a finite-volume solution to scheme \eqref{2.euler}--\eqref{2.uhat} associated to the mesh $\mathcal{D}_m=(\mathcal{T}_m,\mathcal{E}_m,\mathcal{P}_m;\Delta t_m,N_m)$ with mesh size $\delta_m\to 0$ as $m\to\infty$ and constructed in Theorem \ref{thm.ex}. The uniform estimates from Lemmas \ref{lem.r2} and \ref{lem.time} allow us to conclude the relative compactness of $(u_m)$. Recall that $r_1=(d+2)/d$ and $r_2=(d+2)/(d+1)$. To simplify the notation, we set 
\begin{align*}
  \pa_t^m:=\pa_t^{\Delta t_m},\quad \na^m:=\na^{h_m}, \quad \pi_m:=\pi_{\delta_m}, \quad \pi^*_m:=\pi^*_{\delta_m}.
\end{align*}

\begin{lemma}\label{lem.convu}
There exists a limit function $u_i^*\in L^{r_1}(\Omega_T)$ satisfying $\na u_i^*\in L^{r_2}(\Omega_T)$ such that, up to a subsequence and for all $1\le r<r_1$ and $i=1,\ldots,n$, as $m\to\infty$, 
\begin{align*}
  \pi_m u_{i}\to u_i^* \quad\mbox{strongly in }L^r(\Omega_T), \quad
  \pi_{m}^*\na^m u_{i}\rightharpoonup \na u_i^*
  \quad\mbox{weakly in }L^{r_2}(\Omega_T).
\end{align*}
\end{lemma}

\begin{proof}
In view of the uniform estimates
\begin{align*}
  \sum_{k=1}^{N_m}\Delta t_m\|u_{i}^k\|_{1,r_2,\mathcal{T}_m}^{r_2}
  + \sum_{k=1}^{N_m}\Delta t_m\|\pa_t^{m}u_{i}
  \|_{-1,r_2,\mathcal{T}_m} \le C
\end{align*}
from Lemmas \ref{lem.r2} and \ref{lem.time}, we can apply \cite[Theorem 3.4]{GaLa12} and argue as in \cite[Sec.~4.2]{JPZ24} to conclude the existence of a subsequence of $(u_m)$ such that, as $m\to\infty$,
\begin{align*}
  \pi_m u_{i}\to u_i^* \quad\mbox{strongly in }L^1(0,T;L^{r_2}(\T^d)).
\end{align*}
In particular, up to a subsequence, $(\pi_m u_{i})$ converges a.e. Then it follows from the uniform $L^{r_1}(\Omega_T)$ bound for $\pi_m u_{i}$ that $\pi_m u_{i}\to u_i^*$ strongly in $L^r(\Omega_T)$ for $r<r_1$. The weak convergence $\pi_m^*\na^m u_{i}\rightharpoonup \na u_i^*$ in $L^{r_2}(\Omega_T)$ is a consequence of the uniform estimate of Lemma \ref{lem.r2} and the arguments in the proof of \cite[Lemma 4.4]{CLP03} or \cite[Prop.~3.8]{ScSe22}. 
\end{proof}

\begin{lemma}\label{lem.convp}
There exists a limit function $p_i^*\in L^{r_1}(\Omega_T)$ satisfying $\na p_i^*\in L^{r_2}(\Omega_T)$ such that, up to a subsequence and for all $1\le r<r_1$ and $i=1,\ldots,n$, as $m\to\infty$, 
\begin{align*}
  \pi_m p_{i}\to p_i^* \quad\mbox{strongly in }L^r(\Omega_T), \quad
  \pi_{m}^*\na^m p_{i}\rightharpoonup \na p_i^*
  \quad\mbox{weakly in }L^{r_2}(\Omega_T),
\end{align*}
and it holds that $p_i^*=p_i(u^*)$ (see \eqref{1.p}). 
\end{lemma}

\begin{proof}
We prove the lemma for the fully implicit scheme, where $p_i^k$ is defined in \eqref{3.repp}. The mid-point scheme \eqref{3.attp} is treated in a similar way. We compute the error between $\pi_{m}p_i^k$ and $p_i^*$, using formulation \eqref{5.aux} and choosing $x\in K\in\mathcal{T}_m$ and $t\in(t_{k-1},t_k]$:
\begin{align*}
  \pi_{m}p_{i}^k(x) - p_i^*(x,t)
  &= \sum_{j=1}^n\int_{\T^d}\big((\pi_{m}W_{K,y}^{ij})(x)
  - W_{ij}(x-y)\big)\pi_{m}u_{j}^k(y)\dd y \\
  &\phantom{xx}+ \sum_{j=1}^n\int_{\T^d}W_{ij}(x-y)
  \big(\pi_{m}u_{j}^k(y)-u_{j}^*(y)\big)\dd y,
\end{align*}
recalling the definition $W_{K,y}^{ij}=\m(K)^{-1}\int_KW_{ij}(x-y)\dd x$. By the (generalized) Young convolution inequality,
\begin{align*}
  \int_{\T^d}|\pi_{m}p_i^k(x) - p_i^*(x,t)|\dd x
  &\le \sum_{j=1}^n\|u_{j}^k\|_{0,1,\mathcal{T}_m}
  \int_{\T^d}\big|(\pi_{m}W_{K,y}^{ij})(x) - W_{ij}(x-y)\big|\dd x \\
  &\phantom{xx}+ \sum_{j=1}^n\|W_{ij}\|_{L^1(\T^d)}
  \int_{\T^d}\big|\pi_{m}u_j^k(y) - u_j^*(t,y)\big|\dd y.
\end{align*}
We deduce from the boundedness of $W_{ij}$ in $L^1(\T^d)$ and the a.e.\ convergence $(\pi_{m}W_{K,y}^{ij})(x)\to W_{ij}(x-y)$ that the first term on the right-hand side converges to zero as $m\to \infty$. The second term on the right-hand side converges to zero since $\pi_m u_{i}$ converges strongly in $L^1(\Omega_T)$. Thus, $\pi_m p_{i}\to p_i^*=p_i(u^*)$ strongly in $L^1(\Omega_T)$ and, up to a subsequence, a.e. The $L^{r_1}(\Omega_T)$ bound in Lemma \ref{lem.p} shows that this convergence holds for any $1\le r<r_1$. 

Lemma \ref{lem.p} provides a uniform bound for $(\na^m p_{i})$.  Arguing as in the proof of Lemma \ref{lem.convu}, we conclude the weak convergence of $\na^m p_{i}$ in $L^{r_2}(\Omega_T)$. 
\end{proof}

\begin{lemma}\label{lem.convB}
The following convergences hold for all $1\le r<r_1=(d+2)/d$ and $1\le s<\infty$:
\begin{align*}
  \pi_{m}^*\widehat{u}_i\to u_i^*
  \quad\mbox{strongly in }L^r(\Omega_T), \quad
  \pi_{m}^*(B_i^k)\to 1 \quad\mbox{strongly in }L^s(\Omega_T).
\end{align*}
\end{lemma}

\begin{proof}
Let $\sigma=K|L\in\mathcal{E}_m$, $x\in\Delta_\sigma$ and $t\in(t_{k-1},t_k]$. Then we infer from
\begin{align*}
  |\pi_{m}^*&\widehat{u}_i(t,x)-\pi_{m}u_i(t,x)| \\
  &= \bigg|\frac{[\D_{K,\sigma}p_i^k]^+}{\D_{K,\sigma}p_i^k}u_{i,L}^k
  - \frac{[\D_{K,\sigma}p_i^k]^-}{\D_{K,\sigma}p_i^k}u_{i,K}^k
  - \frac{[\D_{K,\sigma}p_i^k]^+ 
  - [\D_{K,\sigma}p_i^k]^-}{\D_{K,\sigma}p_i^k}u_{i,K}^k\bigg| \\
  &=  \bigg|\frac{[\D_{K,\sigma}p_i^k]^+}{\D_{K,\sigma}p_i^k}
  (u_{i,L}^k-u_{i,K}^k)\bigg| \le |u_{i,L}^k-u_{i,K}^k|
\end{align*}
after integration that
\begin{align*}
  \int_0^T\int_{\T^d}|\pi_{m}^*\widehat{u}_i
  -\pi_{m}u_i|\dd x\dd t
  &\le C\sum_{k=1}^N\Delta t_m\sum_{K\in\mathcal{T}_m}
  \sum_{\sigma\in\mathcal{E}_K}|u_{i,L}^k-u_{i,K}^k| \\
  &\le Ch_m\|\na^m u_i\|_{0,1,\mathcal{T}_m^*} \to 0
  \quad\mbox{as }m\to\infty.
\end{align*}
We know from Lemma \ref{lem.convu} that $\pi_{m}u_i\to u_i^*$ a.e.\ in $\Omega_T$, from which we deduce that $\pi_{m}^*\widehat{u}_i\to u_i^*$ a.e. Then the $L^{r_1}(\T^d)$ bound for $\pi_m u_{i}$ from Lemma \ref{lem.r1} implies that $\pi_{m}^* \widehat{u}_i\to u_i^*$ strongly in $L^r(\Omega_T)$ for any $1\le r<r_1$. 

Let $1\le s<\infty$. The second convergence follows from $0<B(s)\le 1$, inequality \eqref{3.Bineq}, and definition \eqref{2.nahv} of $\na^m_\sigma p_i^k$:
\begin{align*}
  \sum_{k=1}^{N_m}\Delta t_m\int_{\T^d}
  |1-\pi_{m}^*(B_{i}^k)|^s\dd x
  &\le \sum_{k=1}^{N_m}\Delta t_m\int_{\T^d}
  |1-\pi_{m}^* B(\kappa^{-1}\D_\sigma p_i^k)|\dd x \\
  &\le Ch_m\sum_{k=1}^{N_m}
  \Delta t_m\int_{\T^d}|\pi_{m}^*\na^m p_i^k|\dd x.
\end{align*}
The limit $m\to\infty$ finishes the proof.
\end{proof}


\subsection{Convergence of the scheme}

In this subsection, we show that the solution to \eqref{2.euler}--\eqref{2.uhat} converges to a solution $(u^*,p^*)$ to \eqref{1.eq}--\eqref{1.ic} with $p^*=p(u^*)$, as the mesh size tends to zero. Let $\phi\in C_0^\infty(\Omega_T)$ and set $\phi_K^k=\phi(t_k,x_K)$ for all $K\in\mathcal{T}$ and $k=1,\ldots,N_m$. We multiply scheme \eqref{2.euler} by $\Delta t_m\phi_K^k$, sum over $k=1,\ldots,N_m$, and argue as in \eqref{5.euler} to obtain $J^m_1+J^m_2=0$, where
\begin{align*}
  J^m_1 &= \sum_{k=1}^{N_m}\sum_{K\in\mathcal{T}_m}\m(K)
  (u_{i,K}^k-u_{i,K}^{k-1}), \\
  J^m_2 &= \frac{1}{d}\sum_{k=1}^{N_m}\Delta t_m
  \sum_{\sigma\in\mathcal{E}_m}\m(\Delta_\sigma)
  \big(\kappa B_{i,\sigma}^k\na_\sigma^m u_i^k
  + \widehat{u}_{i,\sigma}^k\na_\sigma^m p_i^k\big)
  \cdot\na_\sigma^m\phi^k.
\end{align*}
Furthermore, we introduce the terms
\begin{align*}
  J_{10}^m &= -\int_0^T\int_{\T^d}\pa_t\phi \pi_{m}u_i\dd x\dd t
  - \int_{\T^d}\phi(0,x)\pi_{m}u_i(0,x)\dd x, \\
  J_{20}^m &= \int_0^T\int_{\T^d}\pi_{m}^*
  \big(\kappa B_{i}^k\na^m u_i^k
  + \widehat{u}_{i}^k\na^m p_i^k\big)
  \cdot\na\phi\dd x\dd t.
\end{align*}

We show that $J_1^m-J_{10}^m\to 0$ and $J_2^m-J_{20}^m\to 0$ as $m\to\infty$. The former convergence is proved in \cite[Theorem 5.2]{CLP03}. For the latter convergence, let $\sigma=K|L\in\mathcal{E}_m$, $x\in\Delta_\sigma$, and $t\in(t_{k-1},t_k]$. Then we deduce from $\phi_L^k-\phi_K^k = \na\phi(t,x)\cdot(x_L-x_K) + O(\dd_\sigma h_m)$ that
\begin{align*}
  \Delta t_m\m(\Delta_\sigma)(\phi_L^k-\phi_K^k)
  = \int_{t_{k-1}}^{t_k}\int_{\Delta_\sigma}\na\phi\dd x\dd t
  \cdot(x_L-x_K) + O(h_m\Delta t_m),
\end{align*}
and definition \eqref{2.nahv} of $\na_\sigma^m$ yields after multiplication of $d\nu_{K,\sigma}/\dd_\sigma$ that
\begin{align*}
  \Delta t_m \m(\Delta_\sigma)\na_\sigma^m\phi^k 
  = d\int_{t_{k-1}}^{t_k}\int_{\Delta_\sigma}\na\phi\dd x\dd t
  + O(\delta_m).
\end{align*}
This result gives
\begin{align*}
  J_2^m &= \sum_{k=1}^{N_m}\sum_{\sigma\in\mathcal{E}_m}
  \int_{t_{k-1}}^{t_k}\int_{\Delta_\sigma}\na\phi\dd x\dd t
  \cdot\big(\kappa B_{i,\sigma}^k\na_\sigma^m u_i^k
  + \widehat{u}_{i,\sigma}^k\na_\sigma^m p_i^k\big)
  + O(\delta_m) \\
  &= \int_0^T\int_{\T^d}\na\phi
  \cdot\pi_m^*\big(\kappa B_{i,\sigma}^k\na_\sigma^m u_i^k
  + \widehat{u}_{i,\sigma}^k\na_\sigma^m p_i^k\big)\dd x\dd t
  + O(\delta_m) = J_{20}^m + O(\delta_m),
\end{align*}
proving that $J_2^m-J_{20}^m\to 0$. 

The strong convergence of $\pi_{m}u_i$ and the fact $\pi_{m}u_i(0,x) = \m(K)^{-1}\int_K u_i^0\dd x$ for $x\in K$ shows that
\begin{align*}
  J_{10}^m \to -\int_0^T\int_{\T^d}\pa_t\phi u_i^* \dd x\dd t
  - \int_{\T^d}\phi(0,x)u_i^0(x)\dd x.
\end{align*}
Next, we perform the limit $m\to\infty$ in $J_{20}^m$. For this, we observe that the strong convergence $\pi_{m}^*(B_i^k)\to 1$ in $L^s(\Omega_T)$ for any $s<\infty$ (Lemma \ref{lem.convB}) and the weak convergence $\pi_{m}^*\na^h u_{i}\rightharpoonup \na u_i^*$ in $L^{r_2}(\Omega_T)$ (Lemma \ref{lem.convu}) imply that the product converges weakly in $L^1(\Omega_T)$:
\begin{align*}
  \kappa\int_0^T\int_\Omega\pi_{m}^*(B_i^k\na^m u_i^k)
  \cdot\na\phi\dd x\dd t \to \kappa\int_0^T\int_{\T^d}
  \na u_i^*\cdot\na\phi\dd x\dd t.
\end{align*}
We know from Lemmas \ref{lem.convp} and \ref{lem.convB} that 
\begin{align*}
  \pi_{m}^*\widehat{u}_i^k\to u_i^* &\quad\mbox{in }
  L^r(\Omega_T)\quad\mbox{for } r<r_1, \\
  \pi_{m}^*\na^h p_i^k\rightharpoonup \na p_i^*
  &\quad\mbox{in }L^{r_2}(\Omega_T).
\end{align*}
Moreover, the uniform bound in Lemma \ref{lem.r2} yields
$\pi_{m}^*(\widehat{u}_i^k\na^h p_i^k)\rightharpoonup g$ weakly in $L^{r_2}(\Omega_T)$ for some function $g\in L^{r_2}(\Omega_T)$. It follows from \cite[Lemma 12]{JPZ22} that we can identify the limit, $g=u_i^*\na p_i^*$. Therefore,
\begin{align*}
  \int_0^T\int_{\T^d}\pi_{m}^*(\widehat{u}_i^k\na^h p_i^k)
  \cdot\na\phi\dd x\dd t \to \int_0^T\int_{\T^d}
  u_i^*\na p_i^*\cdot\na\phi\dd x\dd t.
\end{align*}
We infer that
\begin{align*}
  J_{20}^m \to \int_0^T\int_{\T^d}(\kappa\na u_i^* + u_i^*\na p_i^*)
  \cdot\na\phi\dd x\dd t.
\end{align*}
Summarizing the previous convergences, we end up with
\begin{align*}
  0 &= J_1^m + J_2^m = (J_1^m-J_{10}^m) + (J_2^m-J_{20}^m)
  + J_{10}^m + J_{20}^m \\
  &\to -\int_0^T\int_{\T^d}\pa_t\phi u_i^* \dd x\dd t
  - \int_{\T^d}\phi(0,x)u_i^0(x)\dd x
  + \int_0^T\int_{\T^d}(\kappa\na u_i^* + u_i^*\na p_i^*)
  \cdot\na\phi\dd x\dd t,
\end{align*}
which concludes the proof.


\section{Numerical experiments}\label{sec.num}

We present in this section several numerical tests. We consider both repulsive and attractive interactions and use two different kernel functions. More precisely, we need to differentiate between interactions involving two distinct species and those within the same species. If $W_{ii}>0$ ($W_{ii}<0$), we say that the interactions within the same $i$th species are self-repulsive (self-attractive), and if $W_{ij}>0$ ($W_{ij}<0$) for $i\neq j$, the interactions are cross-repulsive (cross-attractive). The first kernel function we consider is the Gaussian
\begin{align}\label{5.gauss}
  W_{ij}(z) = \frac{\alpha_{ij}}{\sqrt{2\pi\eps^2}} 
  \exp \bigg(\frac{-|z|^2}{2\eps^2}\bigg), \quad z\in\R^d,
\end{align}
where $\eps>0$ and $\alpha_{ij}\in\R$, which has been used in \cite{JPZ24}. The second one is the top-hat kernel, which was studied in \cite{CSS24}:
\begin{equation}\label{5.tophat}
    W_{ij}(z) = \begin{cases}
    \alpha_{ij}/(2R) &\mbox{if } z \in [-R, R]^d, \\
    0 &\text{otherwise},
    \end{cases}
\end{equation}
where $R>0$ is the detection radius and $\alpha_{ij}$ measures the strength of attraction ($\alpha_{ij}<0$) or repulsion ($\alpha_{ij}>0$). In the numerical simulations, the top-hat kernel is extended periodically, while we use the whole-space Gaussian, which introduces jumps at the boundary. We use a fixed-point method to solve the system numerically; see Algorithm \ref{alg:fixed-iteration}. We consider the two-species cases only; the scheme can be easily extended to the $n$-species system. 

\begin{algorithm}[htbp]
  \caption{Iteration method for the two-species system.}
  \label{alg:fixed-iteration}
  \begin{algorithmic}[1]
    \REQUIRE $\mbox{tol}$, $u_i^{0}, i = 1,2$.
    \ENSURE $u_i^{N_T}$.
    \FOR {$k = 1, 2, \cdots, N_T$}
    \STATE $u_i^{k, 0} = u_i^{k-1}$, $\ell = 1$, $e_k = 1$;
        \WHILE{ $e_{k} > \mbox{tol}$}
    \STATE solve the equations 
    \begin{equation*} 
    \mathrm{m}(K) \frac{u_{i,K}^{k, \ell}
      - u_{i,K}^{k-1}}{\Delta t} 
    + \sum_{\sigma \in \mathcal{E}_K} \mathcal{F}_{K, \sigma} 
    [u_{i}^{k, \ell}, p_i^{k, \ell-1} ] = 0, \qquad i=1,2;
\end{equation*}
\STATE calculate $e_{k} = \max_i\{\|u_i^{k, \ell - 1} - u_i^{k, \ell}\|_{\infty}\}$ and let $u_i^{k,\ell} = u_i^{k,\ell-1}$, $\ell = \ell + 1$;
    \ENDWHILE
    \STATE $u_i^{k} = u_i^{k,\ell}$;
    \ENDFOR
  \end{algorithmic}
\end{algorithm}

The fluxes $\mathcal{F}_{K,\sigma}[u_i^{k,\ell},p_i^{k,\ell}]$ and $\mathcal{F}_{K,\sigma}[u_i^{k,\ell},p_i^{k,\ell-1}]$ are defined by \eqref{2.flux}. We use the discrete potentials $p_{i,K}^{k,\ell-1}$ for the fully implicit time scheme \eqref{3.repp} if $W_{11}>0$ and $W_{22}>0$, and the mid-point time averaging scheme \eqref{3.attp} if $W_{11}<0$ and $W_{22}<0$.


\subsection{Convergence rates}

We compute the convergence rates in space and time in one and two space dimensions to verify our numerical scheme.

\begin{example}[Convergence rates -- one space dimension]\rm 
We consider first the one-di\-men\-sional equations. Since the exact solution generally cannot be computed explicitly, we calculate a reference solution on a fine mesh with $\Delta t=T/2096$ (with $T=0.1$) and $\Delta x:=h=1/2048$. The initial data is
\begin{align*}
  u_1^0(x) = \sin(2\pi x) + 0.5, \quad
  u_2^0(x) = 0.1(\cos(2\pi x)+0.5) \quad\mbox{for }x\in[0, 1). 
\end{align*}
We consider the Gaussian kernel \eqref{5.gauss} with $\eps=1$, $\alpha_{ij}=10^{-3}$ for $i,j=1,2$, and $\kappa=0.01$. The $L^\infty$ and $L^1$ spatial errors for various mesh sizes $h$ are presented in Table \ref{tab.space1d}, confirming the second-order convergence in the discrete $L^1$ norm. The implicit Euler approximation is of first order, as confirmed by our numerical experiments; see Table \ref{tab.time1d}.
\end{example}

\begin{table}[htbp]
        \centering
        \caption{Spatial convergence rates in one space dimension.}
        \begin{tabular}{|c|c|c|c|c|}
            \hline
             \multirow{2}*{$\Delta x$} &  \multicolumn{2}{c|}{$u_1$} & \multicolumn{2}{c|}{$u_2$} \\
             \cline{2-5}
             ~ & $L^{\infty}$-error & $L^1$ error & $L^{\infty}$ error & $L^1$-error \\
             \hline
             $2^{-5}$ & 1.42e-03 & 9.08e-04 & 1.52e-04 & 9.04e-05 \\
             $2^{-6}$ & 3.55e-04 & 2.28e-04 & 4.01e-05 & 2.26e-05 \\
             $2^{-7}$ & 8.85e-05 & 5.68e-05 & 1.11e-05 & 5.64e-06 \\
             $2^{-8}$ & 2.21e-05 & 1.40e-05 & 3.22e-06 & 1.39e-06 \\
             $2^{-9}$ & 9.66e-06 & 3.34e-06 & 9.61e-07 & 3.31e-07 \\
             $2^{-10}$ & 3.25e-06 & 6.68e-07 & 2.51e-07 & 6.62e-08 \\
             \hline
             Order &  1.75 & 2.07 & 1.83 & 2.07\\
             \hline
        \end{tabular}
        \label{tab.space1d}
\end{table}

\begin{table}[htbp]
        \centering
        \caption{Temporal convergence rates in one space dimension.}
        \begin{tabular}{|c|c|c|c|c|}
            \hline
             \multirow{2}*{$\Delta t$} &  \multicolumn{2}{c|}{$u_1$} & \multicolumn{2}{c|}{$u_2$} \\
             \cline{2-5}
             ~ & $L^{\infty}$ error & $L^1$-error & $L^{\infty}$ error & $L^1$-error \\
             \hline
             $T/2^{5}$ & 2.33e-05 & 1.52e-05 & 4.00e-06 & 1.48e-06 \\
             $T/2^{6}$ & 1.15e-05 & 7.52e-06 & 1.99e-06 & 7.34e-07 \\
             $T/2^{7}$ & 5.68e-06 & 3.70e-06 & 9.78e-07 & 3.61e-07 \\
             $T/2^{8}$ & 2.75e-06 & 1.79e-06 & 4.73e-07 & 1.75e-07 \\
             $T/2^{9}$ & 1.28e-06 & 8.36e-07 & 2.21e-07 & 8.16e-08 \\
             $T/2^{10}$ & 5.50e-07 & 3.58e-07 & 9.46e-08 & 3.50e-08 \\
             \hline
             Order &  1.07 & 1.07 & 1.07 & 1.07\\
             \hline
        \end{tabular}
        \label{tab.time1d}
\end{table}

\begin{example}[Convergence rates -- two space dimensions]\rm 
We consider the two-dimensional domain  $\Omega=[0,1)^2$ and use the top-hat kernel \eqref{5.tophat} with $R=\frac18$, $\alpha_{ij}=-1$ for $i,j=1,2$, and $\kappa=0.01$. The initial data equal
\begin{align*}
  u_1^0(x,y) = \frac{0.1\kappa(\sin(2\pi(x-y))+1)
  }{\|\sin(2\pi(x-y))+1\|_{L^1(\Omega)}}, \quad
  u_2^0(x,y) = \frac{0.1\kappa(\cos(2\pi(x+y))+1)
  }{\|\cos(2\pi(x+y))+1\|_{L^1(\Omega)}}
\end{align*} 
for $(x,y)\in\Omega$. The end time is $T=0.01$ and the reference mesh sizes are $\Delta t=T/2^8$ and $\Delta x=2^{-8}$ (i.e., the mesh size in both directions equals $\Delta x_1=\Delta x_2=2^{-8}$). As in the previous example, Tables \ref{tab:conv2d_u1} confirms the second-order convergence in space and first-order convergence in time. Since the errors for $u_2$ are practically identical, only the results for $u_1$ are reported.
\end{example}

\begin{table}[htbp]
\centering
\caption{Two-dimensional convergence for $u_1$: spatial (left) and temporal (right) refinements, measured in $L^1$ and $L^{\infty}$ norms.}
\label{tab:conv2d_u1}
\setlength{\tabcolsep}{6pt}
\begin{tabular}{|c|cc|c|cc|}
\hline
\multicolumn{3}{|c|}{\textbf{Space refinement} ($\Delta t = T/2^8$)} &
\multicolumn{3}{c|}{\textbf{Time refinement} ($\Delta x = 2^{-8}$)} \\
\hline
$\Delta x$ & $L^{\infty}$ error & $L^1$ error &
$\Delta t$ & $L^{\infty}$ error & $L^1$ error \\
\hline
$2^{-4}$ & 1.25e-05 & 7.85e-06 & $T/2^{3}$ & 4.19e-09 & 2.34e-09 \\
$2^{-5}$ & 3.11e-06 & 1.96e-06 & $T/2^{4}$ & 2.03e-09 & 1.13e-09 \\
$2^{-6}$ & 7.46e-07 & 4.64e-07 & $T/2^{5}$ & 9.47e-10 & 5.30e-10 \\
$2^{-7}$ & 1.54e-07 & 9.20e-08 & $T/2^{6}$ & 4.06e-10 & 2.27e-10 \\
\hline
Order & 2.11 & 2.13 & Order & 1.12 & 1.12 \\
\hline
\end{tabular}
\end{table}

\subsection{Some model features}

We present some simulations for the attractive case in the one-dimensional interval $\Omega=[-L, L)$ with $L=10$, using the top-hat kernel \eqref{5.tophat}. The time step size is chosen as $\Delta t=\Delta x$. Our aim is to discuss the choice of the weight function $B$ and the parameters $\kappa$ in \eqref{2.flux} and $\eps$ in \eqref{5.gauss}.

\begin{example}[Choice of weight function]\rm 
We compare the classical Bernoulli weight $B(s)=s/(\mathrm{e}^s-1)$ with a (scaled) sigmoid weight
\[
  B(s)=\frac{2}{\mathrm{e}^s+1},\qquad s\in\mathbb{R}.
\]
Both choices satisfy Assumptions~(H2)–(H3) used to define the generalized Scharfetter--Gummel flux (in particular, it holds that $B(|s|)\ge 1-\alpha|s|$ for some $0\le\alpha<1$). 
We use the top-hat kernel \eqref{5.tophat} with the values $R=1$,  $\alpha_{11}=-20$, $\alpha_{22}=-2$, $\alpha_{12}=\alpha_{21}=10$ (self-attractive and cross-repulsive interactions). The parameters are $\Delta t = \Delta x = 1/25$, $\kappa=0.25$, and the initial data is $u_1^0=u_2^0=\frac18\mathrm{1}_{[-4,4]}$. Figure~\ref{fig.sigmoid} shows $(u_1,u_2)$ at three times. Consistently with the theory, the two weights yield the same qualitative dynamics and nearly identical profiles, confirming that the analysis applies to a broad class of weight functions beyond the classical Bernoulli case. From an implementation viewpoint, the continuous sigmoid avoids the special-case handling at $s=0$ required by the Bernoulli weight.
\end{example}

\begin{figure}[ht]
\includegraphics[width=0.32\linewidth]{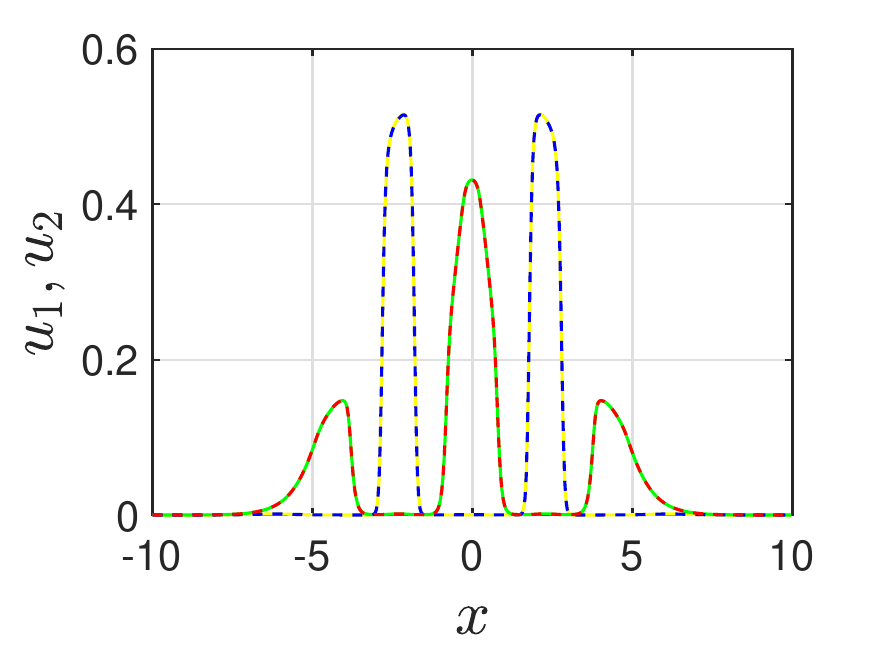}
\includegraphics[width=0.32\linewidth]{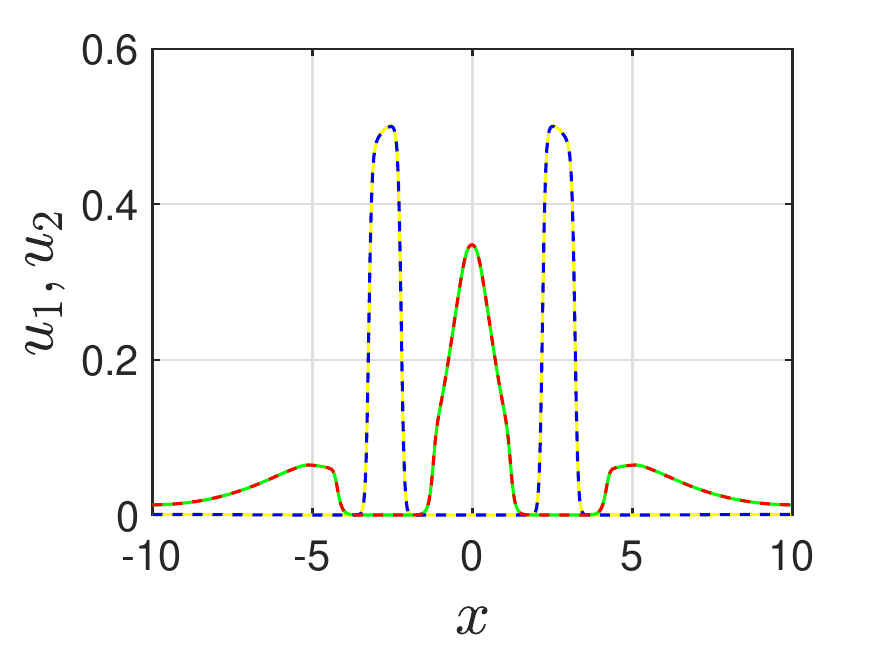}
\includegraphics[width=0.32\linewidth]{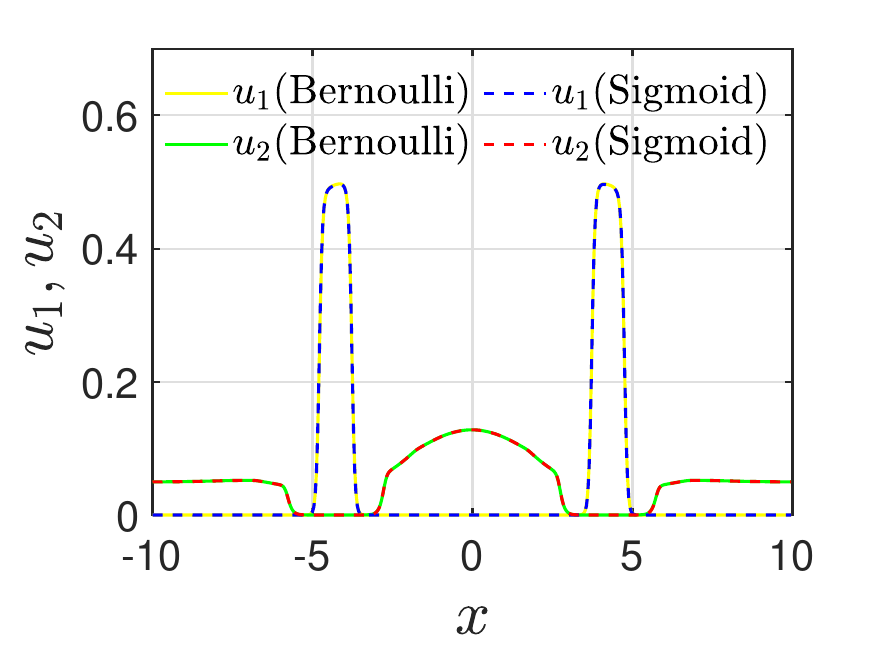}
\caption{Solution profiles using the Bernoulli and Sigmoid weight functions at times $t=3.4$ (left), $t=22.4$ (middle), $t=160$ (right), with self-attractive/cross-repulsive parameters $(-20,-2,10,10)$.}
\label{fig.sigmoid}
\end{figure}

\begin{example}[Robustness with respect to $\kappa$]\rm
A key advantage of the Scharfetter--Gummel discretization is its robustness for small diffusion coefficients~$\kappa$. To illustrate this property, we consider a system with the top-hat kernel \eqref{5.tophat} using $R=1$, $\alpha_{11}=-5$, $\alpha_{22}=-2$, $\alpha_{12}=\alpha_{21}=15$ (self-attractive and cross-repulsive interactions) and the initial data $u_1^0 = u_2^0 = \frac12\mathrm{1}_{[-1,1]}$. The mesh parameters are $\Delta t = \Delta x = 0.05$. Figure~\ref{fig.kappa} displays the stationary profiles at $t=200$ for different values of $\kappa$. As expected, diffusion counteracts segregation: For larger $\kappa$, the supports overlap more, whereas for smaller $\kappa$ the interfaces sharpen and the overlap decreases. Importantly, even for very small diffusion ($\kappa=5\times10^{-3}$), the scheme remains stable and positivity-preserving; no spurious oscillations or overshoots are observed.
\end{example}

\begin{figure}[ht]
\includegraphics[width=0.32\linewidth]{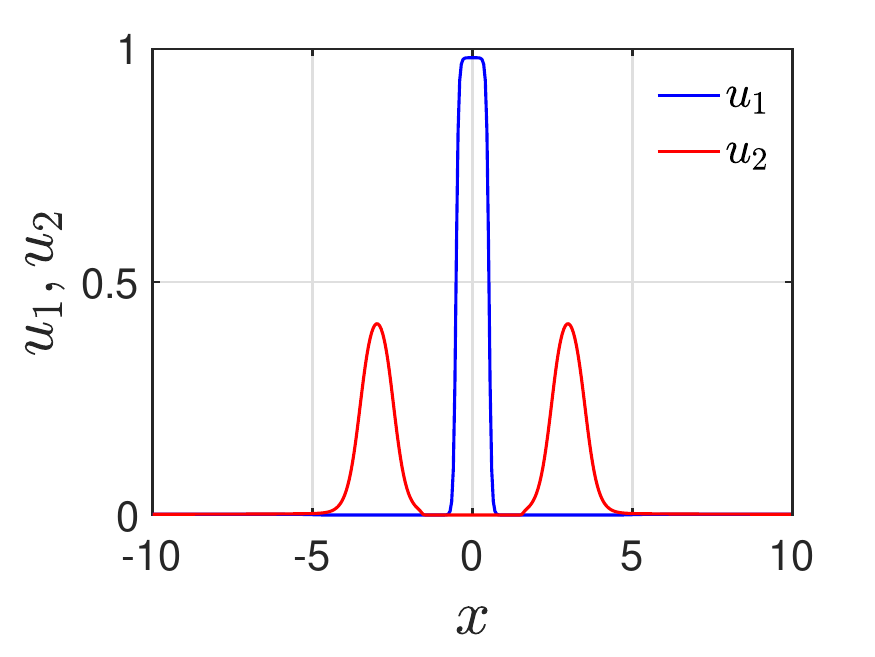}
\includegraphics[width=0.32\linewidth]{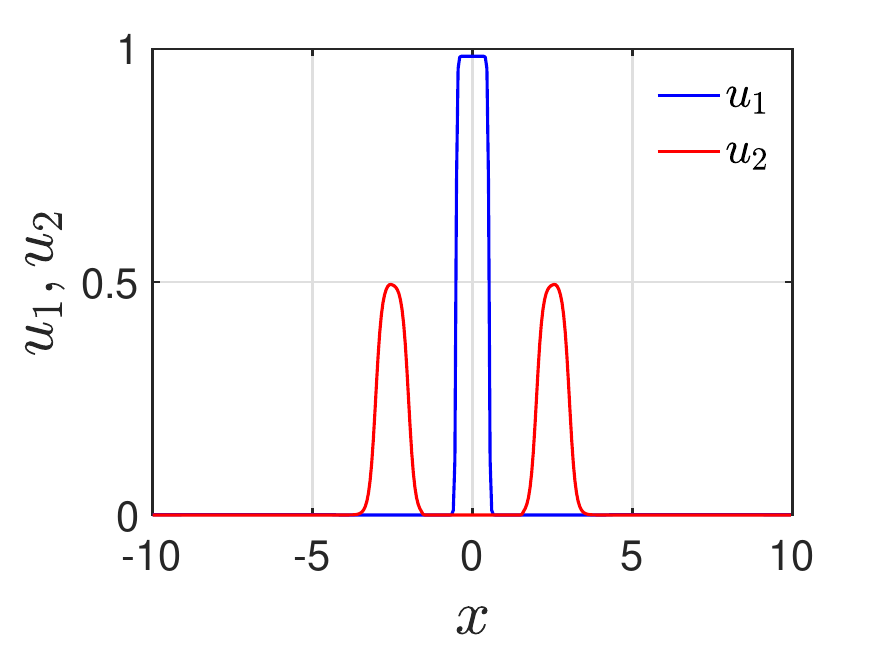}
\includegraphics[width=0.32\linewidth]{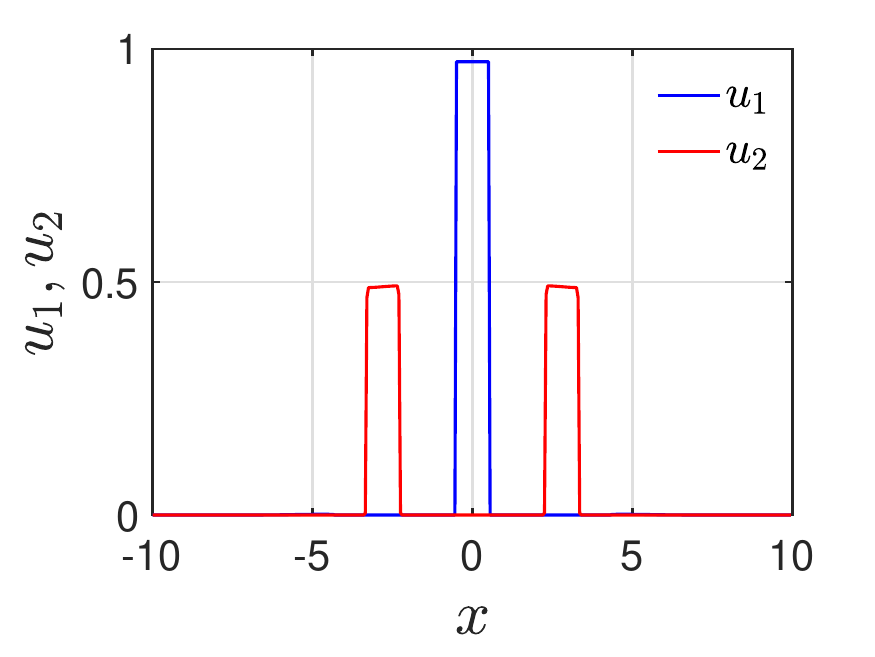}
\caption{Robustness of the scheme for small $\kappa$: Solution profiles at time $t=200$ for $\kappa = 0.1$ (left), $\kappa=0.05$ (middle), $\kappa=0.005$ (right).}
\label{fig.kappa}
\end{figure}

\begin{figure}
\centering
\includegraphics[width=0.32\linewidth]{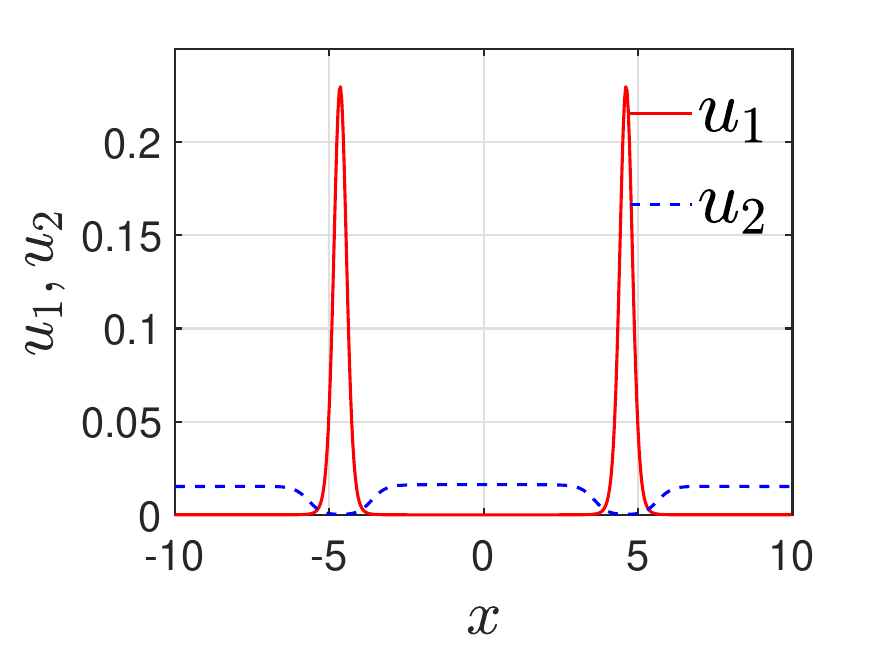}
\includegraphics[width=0.32\linewidth]{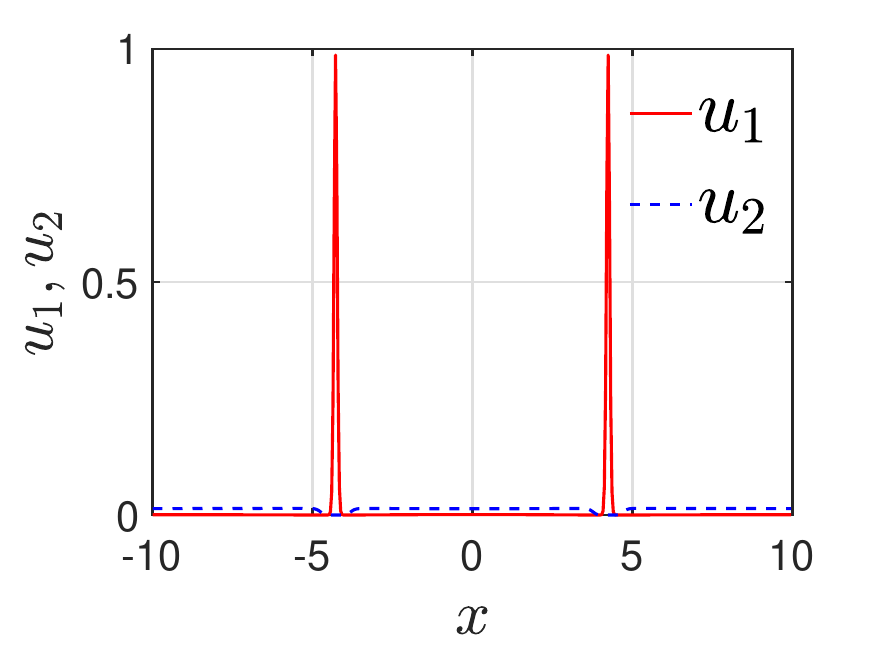}
\includegraphics[width=0.32\linewidth]{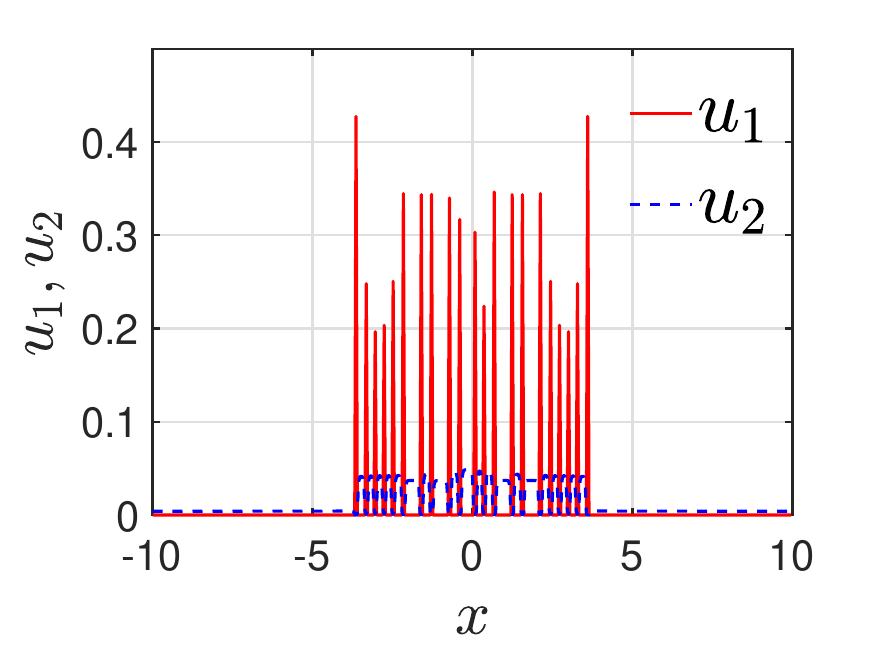}
\caption{Variation of $\eps$: Solution profiles at $t=300$ using the periodically extended Gaussian kernel for $\eps=0.5$ (left), $\eps=0.2$ (middle), and $\eps=0.02$ (right). The simulation breaks down when $\eps=0.02$.}
\label{fig.blowup}
\end{figure}

\begin{example}[Choice of $\eps$]\rm
We discuss the limit $\eps\to 0$ in the one-dimensional setting with the Gaussian kernel \eqref{5.gauss}, periodically extended over the whole line, and with the coefficients $\alpha_{11} = 20$, $\alpha_{22} = 2$, $\alpha_{12} = \alpha_{21} = -10$ (self-repulsion, cross-attraction) and $\kappa=0.01$. The initial data is \( u_1^0 = u_2^0 = \frac{1}{32}\mathrm{1}_{[-4,4]} \). We expect that the nonlocal equations converge to the local ones \cite[Theorem 5]{JPZ22},
\begin{align*}
  \pa_t u_i + \diver(u_i\na p_i(u)) = 0, \quad 
  p_i(u) = \alpha_{i1}u_1+\alpha_{i2}u_2, \quad i=1,2.
\end{align*}
The local system is solvable only if the matrix $(\alpha_{ij})$ is positive definite, which is not the case in the present example. 
As a consequence, the numerical simulations are expected to break down if $\eps$ becomes too small. This expectation is confirmed in Figure \ref{fig.blowup}.
\end{example} 

\subsection{Evolution of the entropies}

We study the evolution of the Boltzmann and Rao entropies $H_B$ and $H_R$, respectively (see \eqref{2.HBHR} for the definitions). In one space dimension, we choose $\Omega=[-L, L)$ and the initial data
\begin{align*}
  u_1^0(x) = \mathrm{1}_{[-L/8,L/8]}(x), \quad
  u_2^0(x) = \mathrm{1}_{[L/8,3L/4]}(x) \quad\mbox{for }x\in\Omega,
\end{align*}
while in two space dimensions, we use $\Omega=[-L, L)^2$ and
\begin{align*}
  u_1^0 = 0.1\cdot\mathrm{1}_{[3L/8,5L/8]\times[L/2,3L/4]}, \quad
  u_2^0 = 0.1\cdot\mathrm{1}_{[3L/8,5L/8]\times[L/4,L/2]} \quad
  \mbox{for }(x,y)\in\Omega.
\end{align*}

\begin{figure}[ht]
\includegraphics[width=0.32\linewidth]{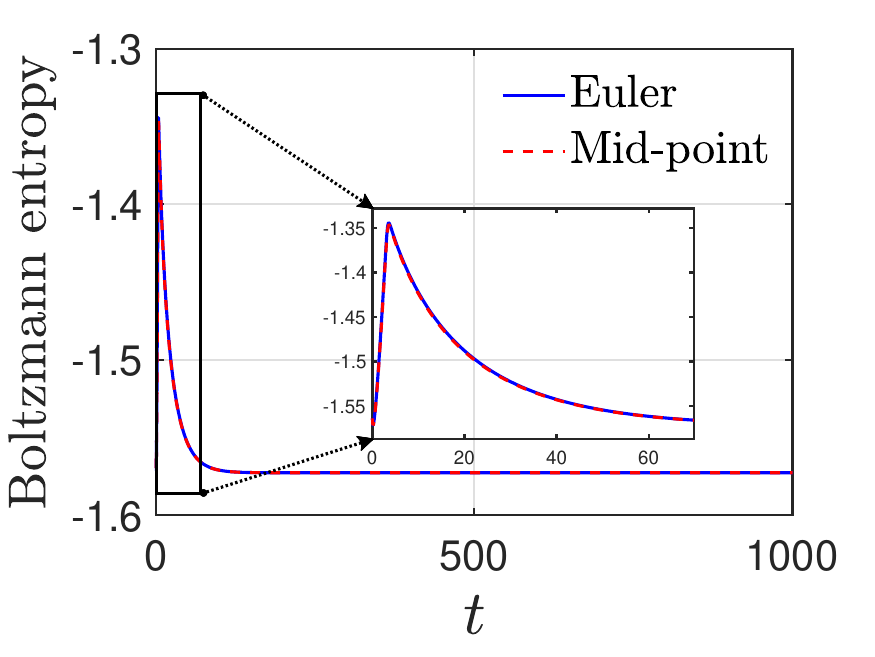}
\includegraphics[width=0.32\linewidth]{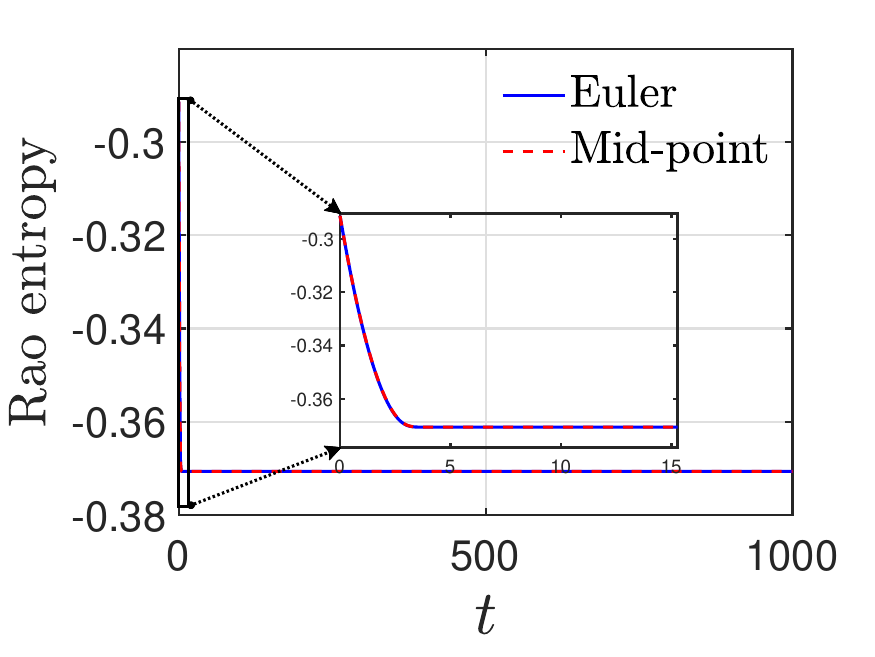}
\includegraphics[width=0.32\linewidth]{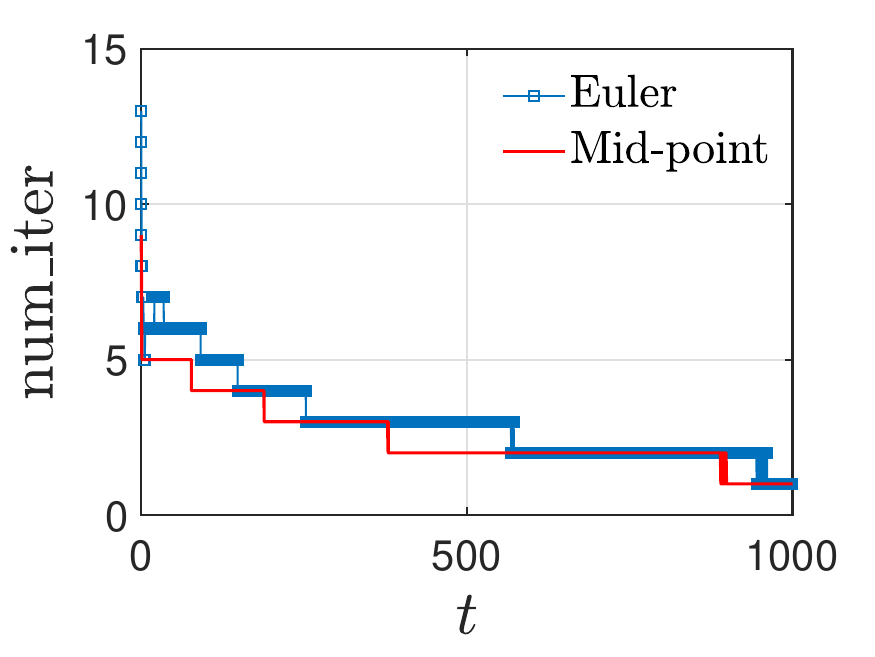}
\caption{Attractive interactions with the top-hat kernel: evolution of the Boltzmann entropy (left) and Rao entropy (middle) as well as the number of iterations (right).}
\label{fig.1Datt}
\end{figure}

\begin{figure}[ht]
\includegraphics[width=0.49\linewidth]{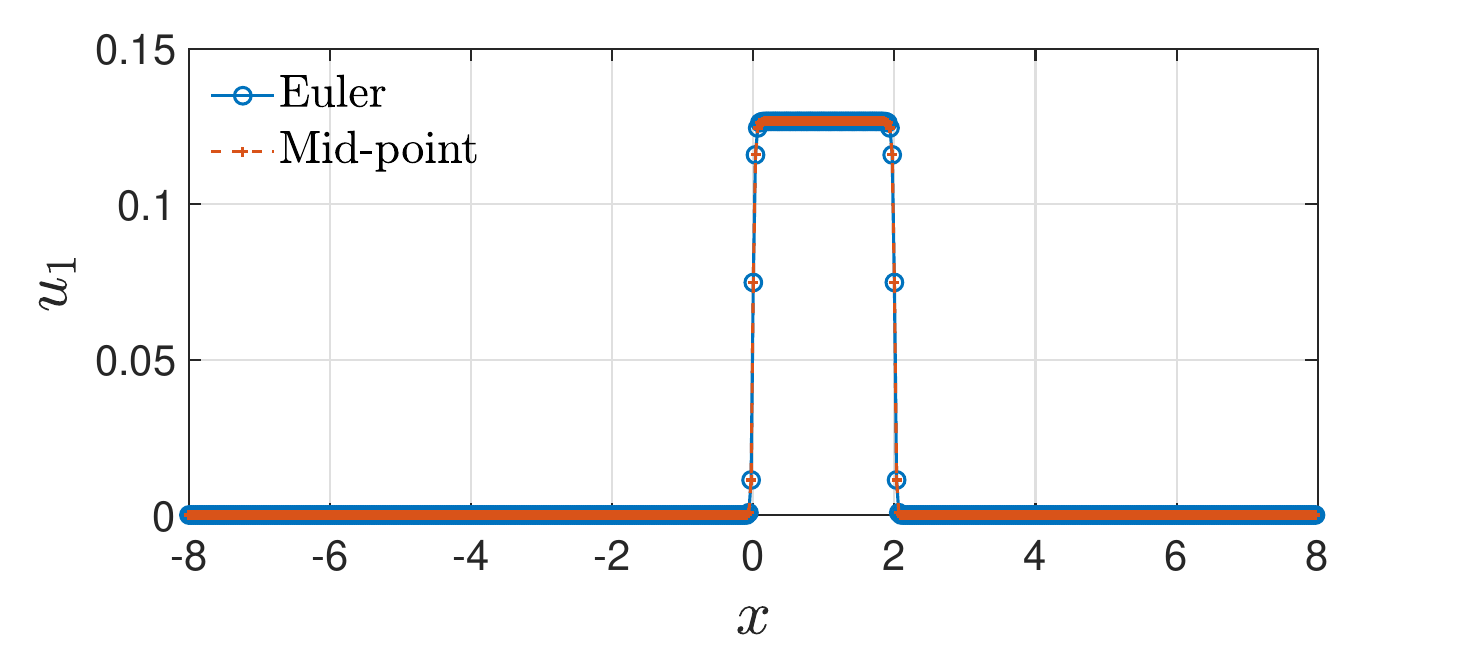}
\includegraphics[width=0.49\linewidth]{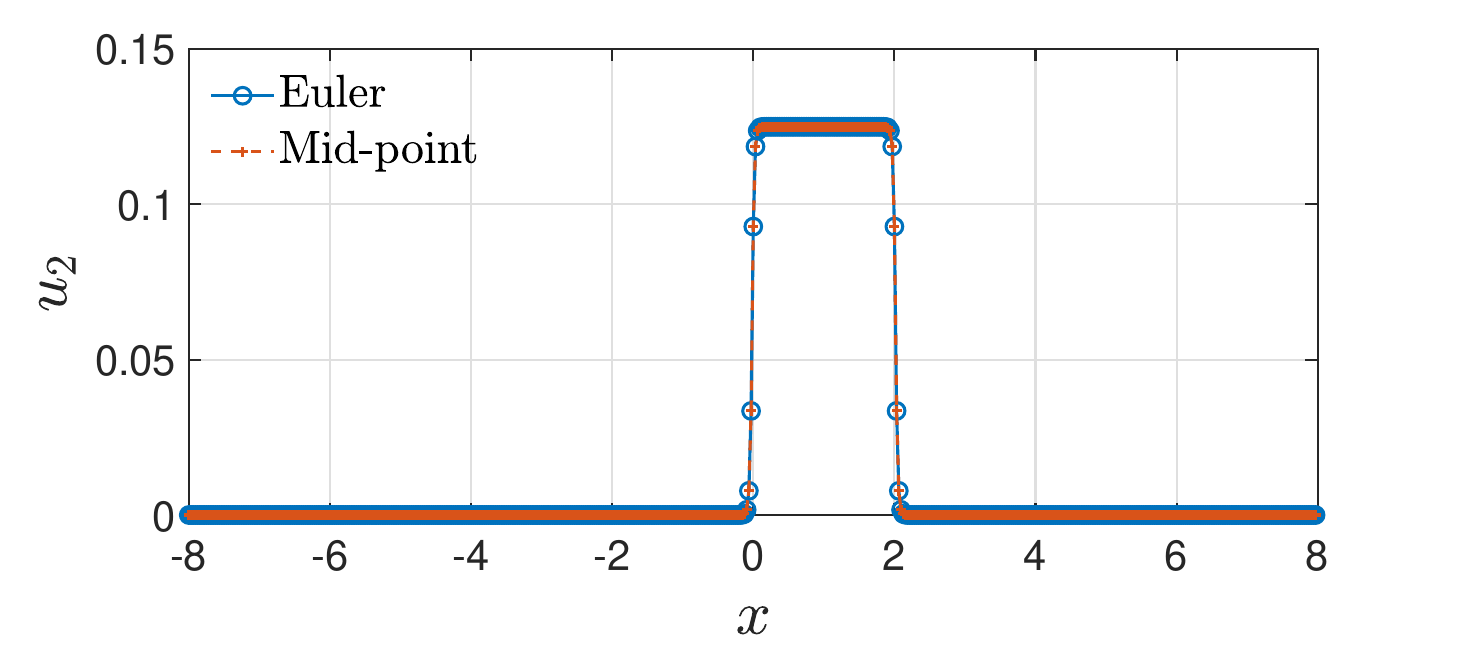}
\caption{Attractive interactions with the top-hat kernel: Solution profiles $u_1$ (left) and $u_2$ (right).}
\label{fig.1Ddens}
\end{figure}

\begin{example}[One-dimensional attractive interactions]\rm 
We choose the top-hat kernel \eqref{5.tophat} with $R=2$ and attractive interactions, $\alpha_{11}=-20$, $\alpha_{22}=-6$, $\alpha_{12}=\alpha_{21}=-10$, and $\kappa=0.01$. The numerical parameters are $T=1000$, $L=8$, and $\Delta t=\Delta x=L/2^8$. We use both the implicit Euler and mid-point schemes for the discrete potentials and compare the evolution of the Boltzmann and Rao entropies; see Figure \ref{fig.1Datt}. We see that the Boltzmann entropy increases initially but decreases for all larger times, while the Rao entropy decays for all times. The mid-point scheme needs fewer iterations than the implicit Euler method, but both schemes produce almost the same solution. However, when we decrease the detection radius to $R=0.6$, we observe that with the mid-point rule, the Rao entropy decays, but the implicit Euler method fails to converge unless the time step size is decreased. This test illustrates that the mid-point rule is preferable in the case of attractive interactions. The corresponding solution profiles are reported in Figure \ref{fig.1Ddens}. 
\end{example}

\begin{example}[One-dimensional repulsive interactions]\rm
We consider repulsive interactions with the Gaussian kernel \eqref{5.gauss} with $\eps=1$, $\alpha_{11}=10$, $\alpha_{22}=3$, $\alpha_{12}=\alpha_{21}=5$, and $\kappa=0.01$. In this example, both the Boltzmann and Rao entropies decay, using the implicit Euler or mid-point scheme (see Figure \ref{fig.1Drep}). When the whole-space Gaussian is used with periodic boundary condition, a boundary layer appears near the domain boundary; this layer disappears when a periodic Gaussian extension is employed.
\end{example}

\begin{figure}[ht]
\includegraphics[width=0.32\linewidth]{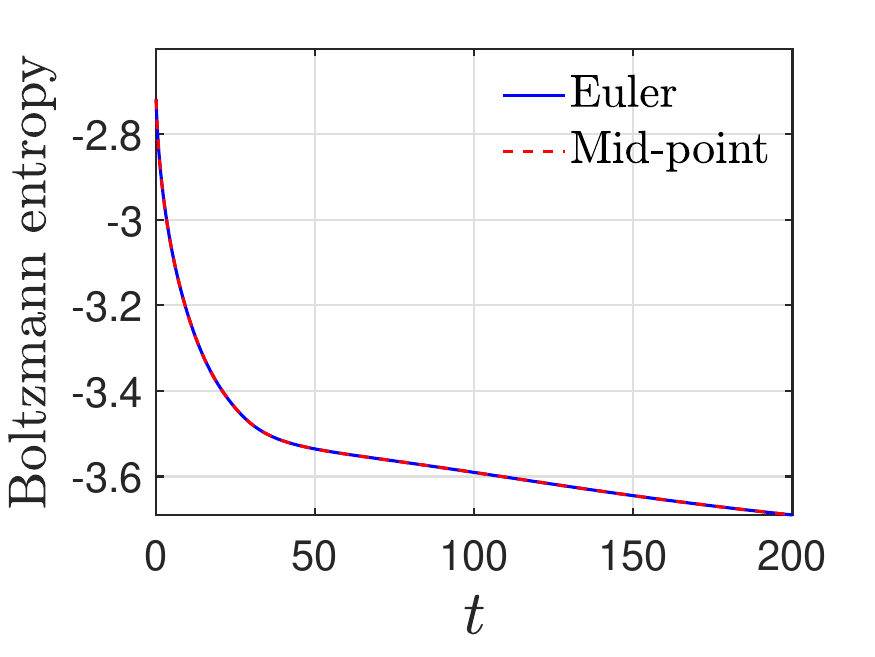}
\includegraphics[width=0.32\linewidth]{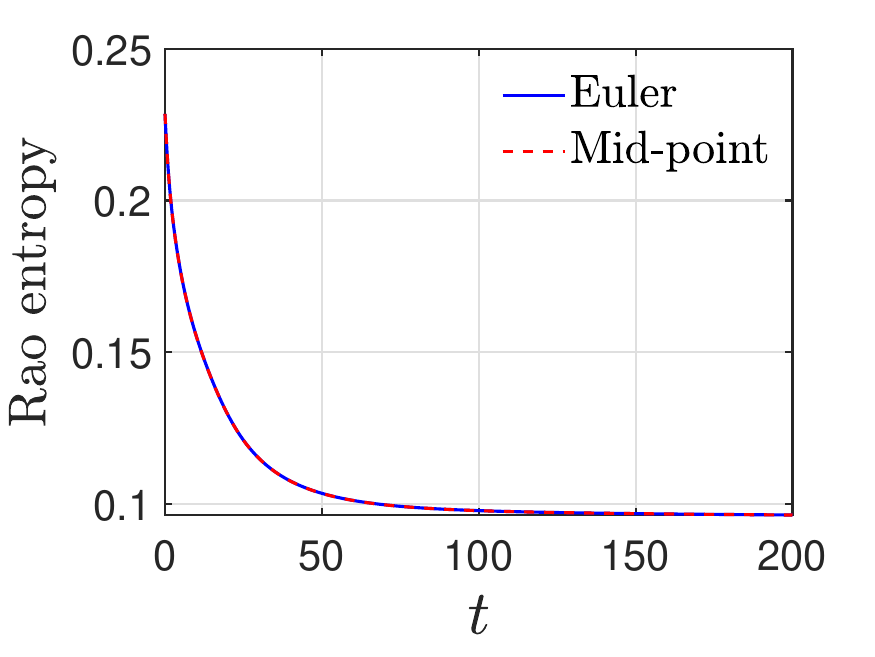}
\includegraphics[width=0.32\linewidth]{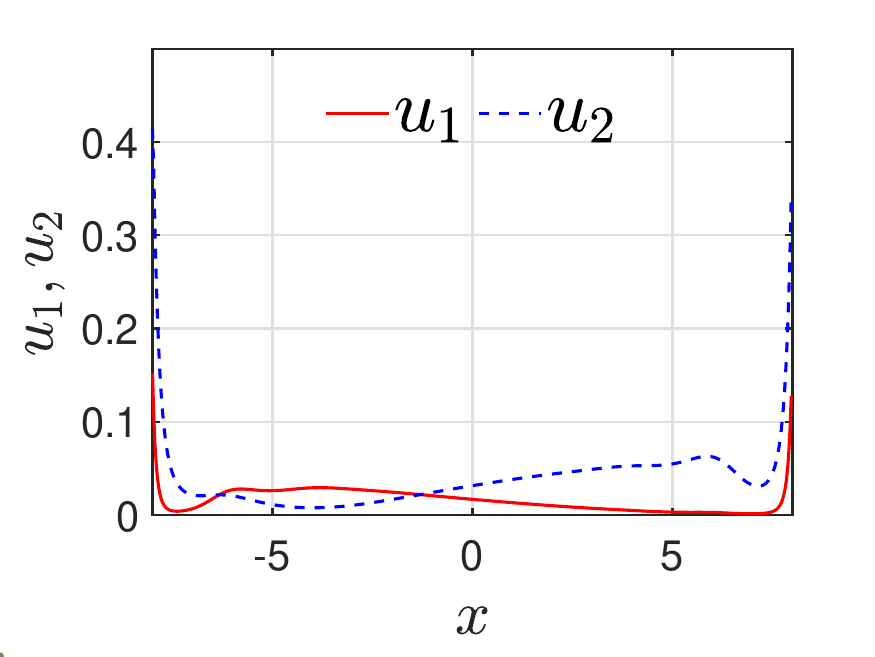}
\caption{Repulsive interactions with the Gaussian kernel: Evolution of the Boltzmann entropy (left), Rao entropy (middle), and solution profiles (right).}
\label{fig.1Drep}
\end{figure}

\begin{example}[Two-dimensional repulsive and attractive interactions]\rm 
We consider the top-hat kernel \eqref{5.tophat} for both repulsive and attractive interactions. The domain is $\Omega=[0, 1)^2$, and the mesh size is $\Delta t=\Delta x=0.01$. The interaction parameters are $\kappa=0.01$ and 
\begin{align*}
  &\mbox{repulsive interactions:} && \alpha_{11}=10,\ \alpha_{22}=6,
  \ \alpha_{12}=\alpha_{21}=5, \\
  &\mbox{self-attractive interactions:} && \alpha_{11}=-10,\ \alpha_{22}=-6,\ \alpha_{12}=\alpha_{21}=5.
\end{align*}
Similarly as in the previous example, Figure \ref{fig.ent2d} shows that the Boltzmann and Rao entropies are decreasing in time for repulsive interactions, while for the self-attractive case, the Boltzmann entropy increases initially and the Rao entropy is decreasing. In the former case, the system almost reaches its steady state at $T=30$, but the solutions in the latter case are still not stationary at $T=50$, i.e., the relative change of the entropies is 2.40e-05 (Boltzmann entropy) and 8.25e-05 (Rao entropy).  
\end{example}

\begin{figure}[ht]
\includegraphics[width=0.35\linewidth]{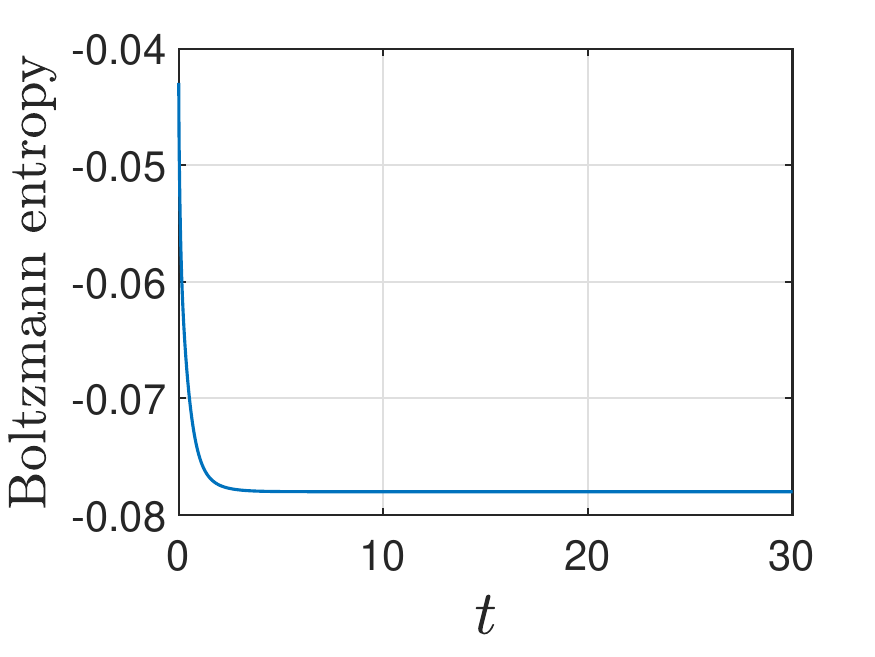}
\includegraphics[width=0.35\linewidth]{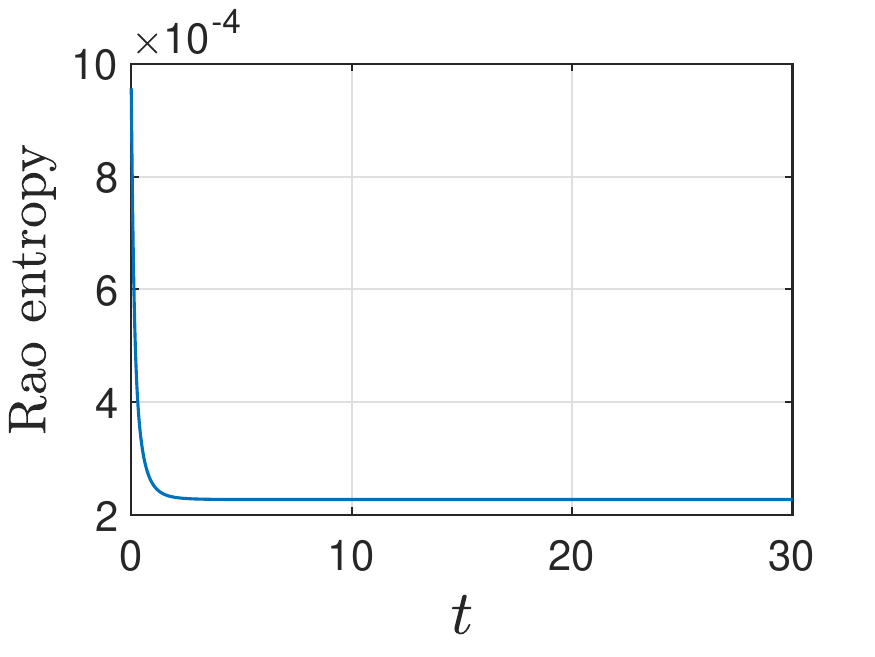}
\includegraphics[width=0.35\linewidth]{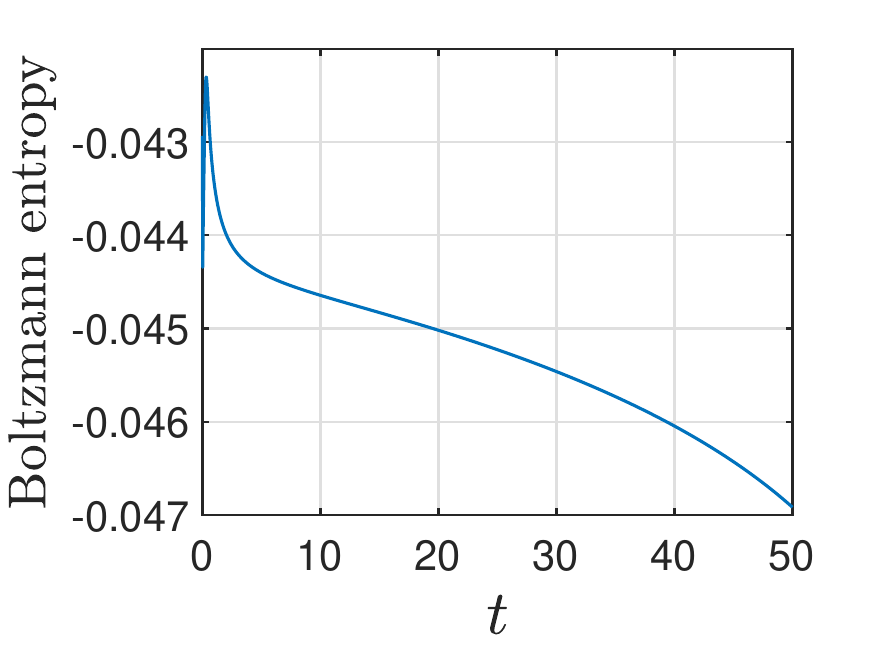}
\includegraphics[width=0.35\linewidth]{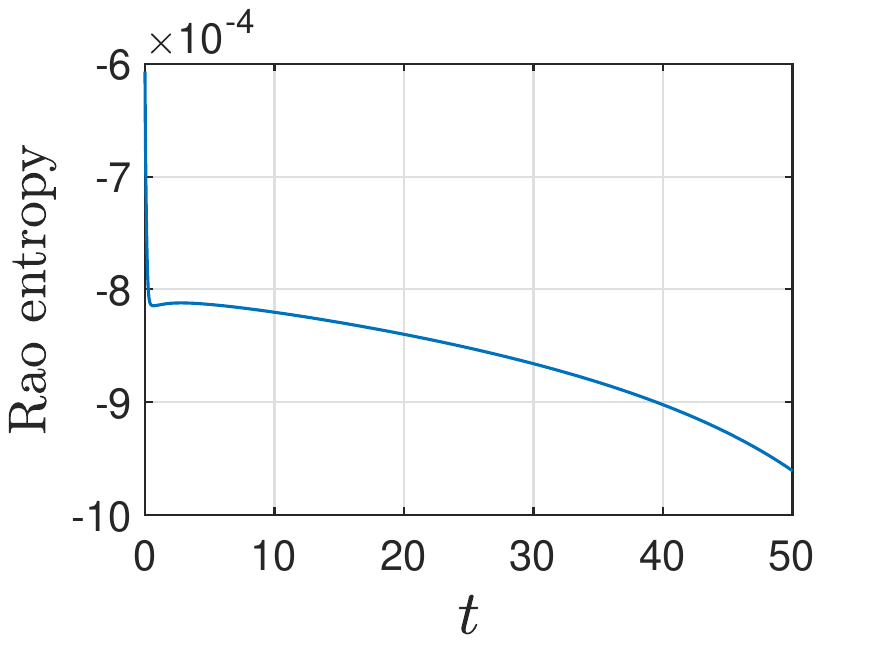}
\caption{Evolutions of the Boltzmann entropy (left) and Rao entropy (right) for repulsive systems (top) and self-attractive systems (bottom).}
\label{fig.ent2d}
\end{figure}

\subsection{Strong repulsive interactions}

We represent strong interactions by kernel functions with large coefficients:
\begin{align}\label{5.alpha}
  \alpha_{11} = 20, \quad \alpha_{22} = 2, \quad
  \alpha_{12} = \alpha_{21}=10.
\end{align}
The parameters are $\kappa=0.25$, $L=10$, and $\Delta t=\Delta x = 1/25$. We choose the Gaussian kernel with $\eps=1$ and the initial data $u_1^0=u_2^0 = \frac18\mathrm{1}_{[-4,4]}$. As expected, the Boltzmann and Rao entropies are decreasing in time (not shown). Figure \ref{fig.sdens} (top row) illustrates the solution profiles at various times in one space dimension. Since the kernel is not periodic, a jump at the boundary is introduced, which produces a boundary layer. The boundary layer does not appear when a truncated Gaussian kernel with periodic extension is employed (not shown). We notice that the discrete mass is conserved over time, also in presence of boundary layers.

When we use the top-hat kernel \eqref{5.tophat} with $R=1$ and the coefficients \eqref{5.alpha}, small-scale oscillations appear; see Figure \ref{fig.sdens} (bottom row). Such oscillations are also observed in \cite[Figure 4]{CSS24} when self-repulsion is large. The wavelength of the oscillations is related to the detection radius $R$. Indeed, when the radius is doubled to $R=2$, the wavelength doubles too; see Figure \ref{fig.sosc2}. For larger diffusion coefficient $\kappa=0.5$, the small-scale oscillations disappear as $t\to\infty$, and the system converges to the constant steady state, which is consistent with the observations of \cite[Sec.~5.1]{CSS24}. These results indicate that the presence of small-scale oscillations is jointly determined by both the diffusion constant and the detection radius. 

\begin{figure}[ht]
\includegraphics[width=0.32\linewidth]{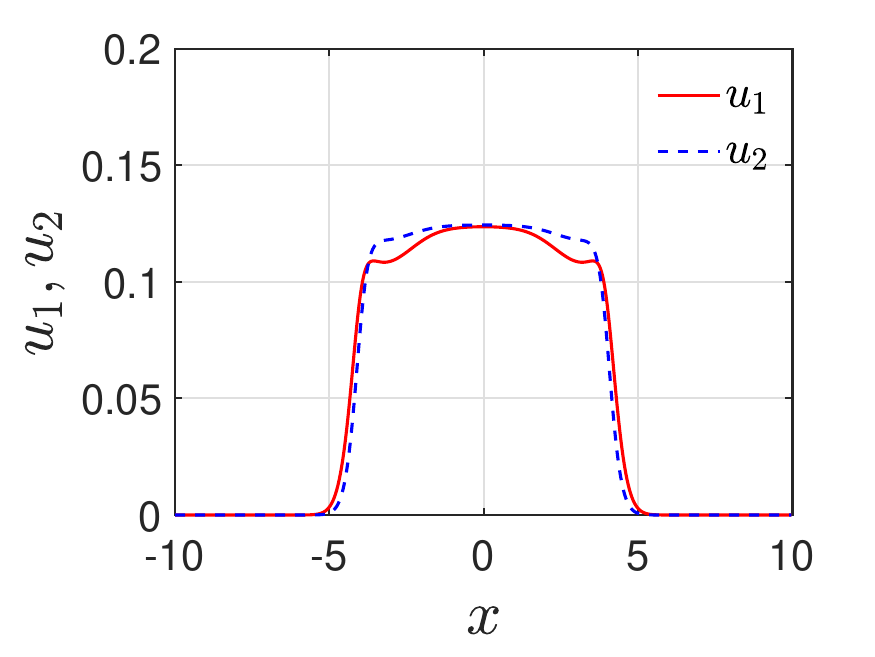}
\includegraphics[width=0.32\linewidth]{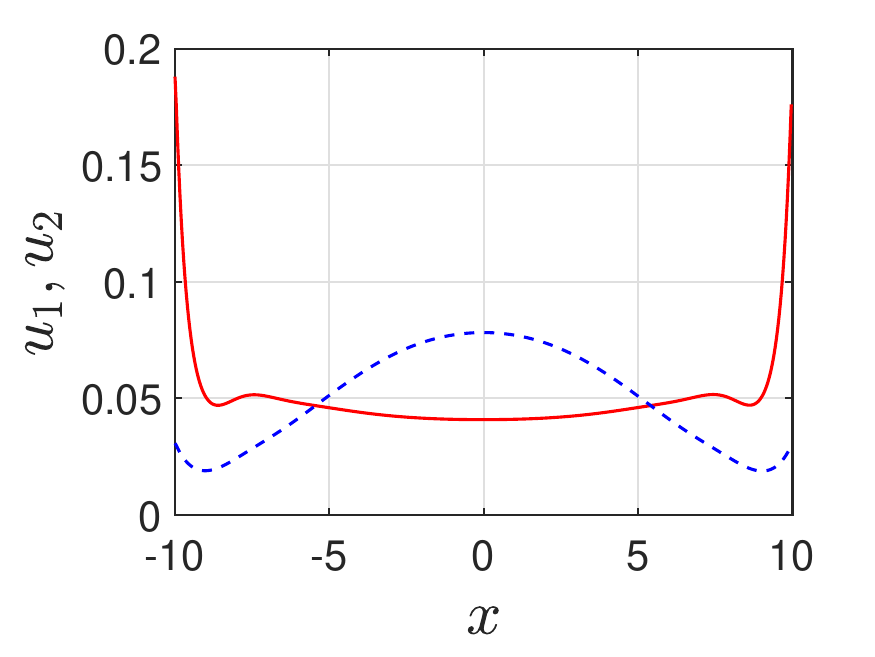}
\includegraphics[width=0.32\linewidth]{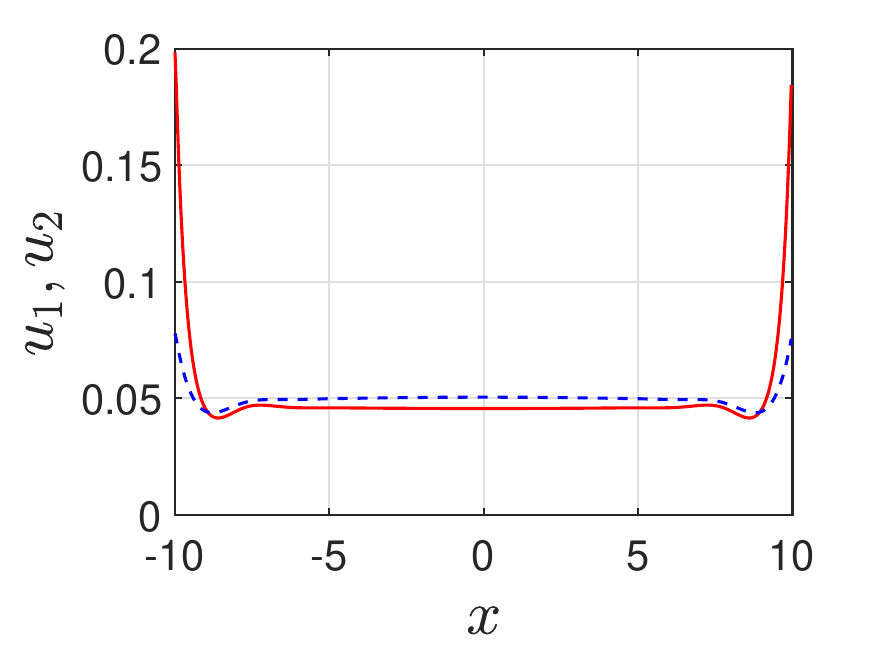}
\includegraphics[width=0.32\linewidth]{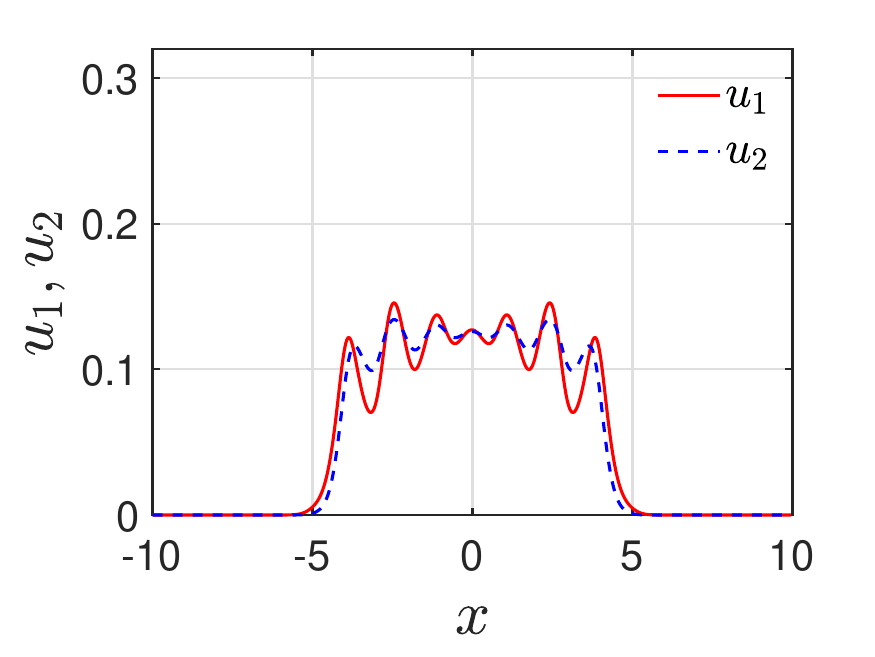}
\includegraphics[width=0.32\linewidth]{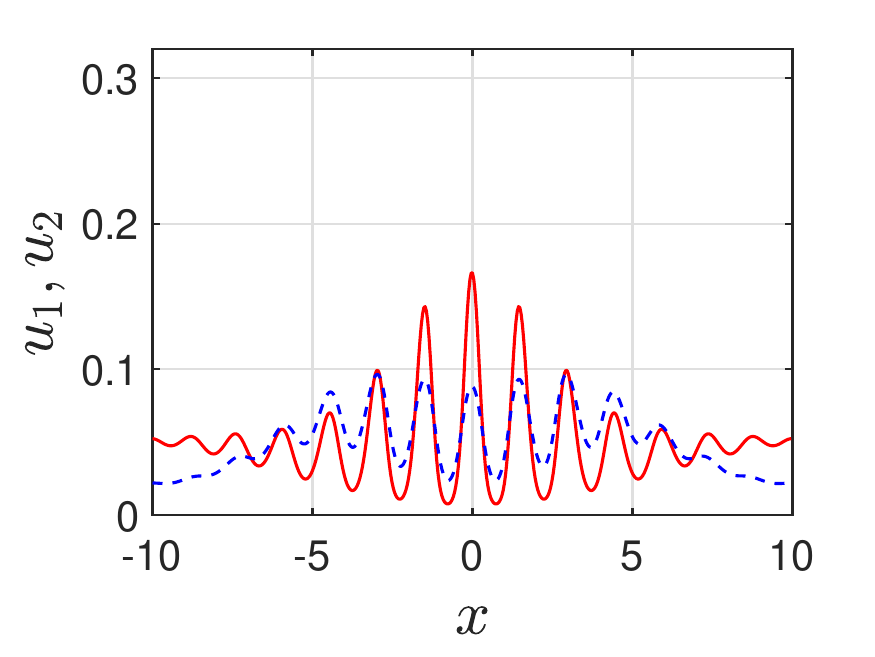}
\includegraphics[width=0.32\linewidth]{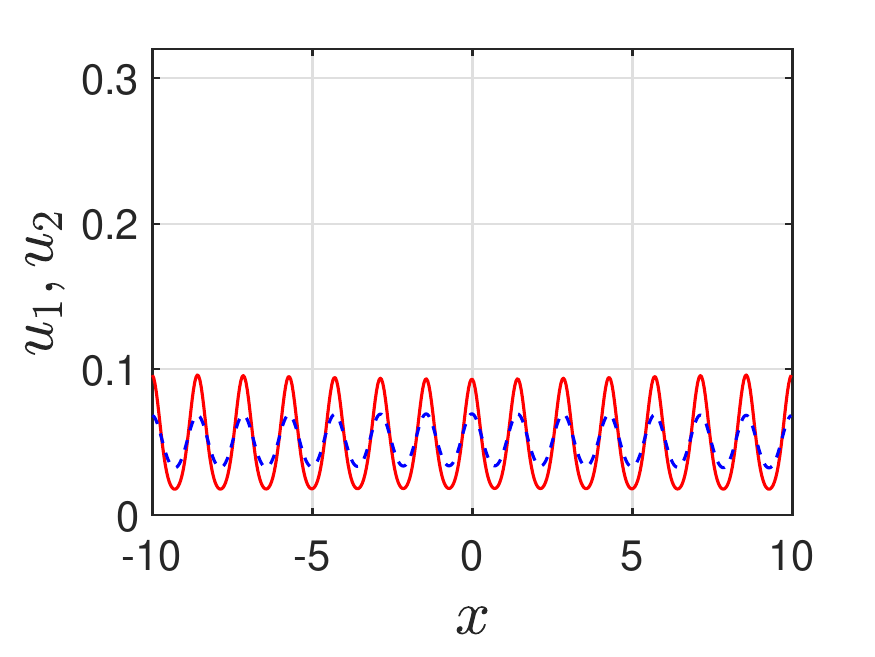}
\caption{Strong repulsive interactions: Solution profiles at times $t=0.2$ (left), $t=22.4$ (middle), and $t=300$ (right) using the Gaussian kernel (top row) and the top-hat kernel (bottom row).}
\label{fig.sdens}
\end{figure}

\begin{figure}[ht]
\includegraphics[width=0.32\linewidth]{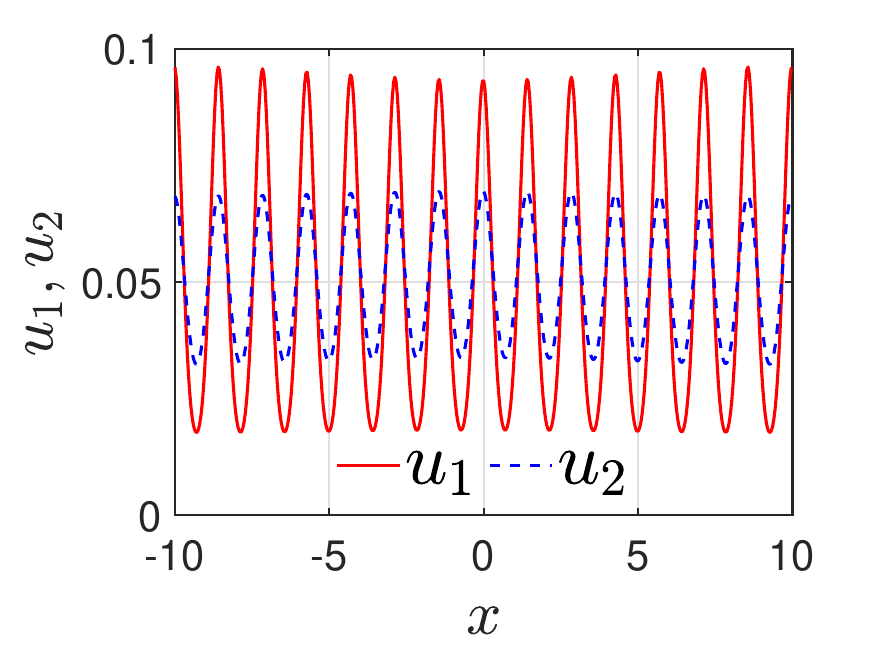}
\includegraphics[width=0.32\linewidth]{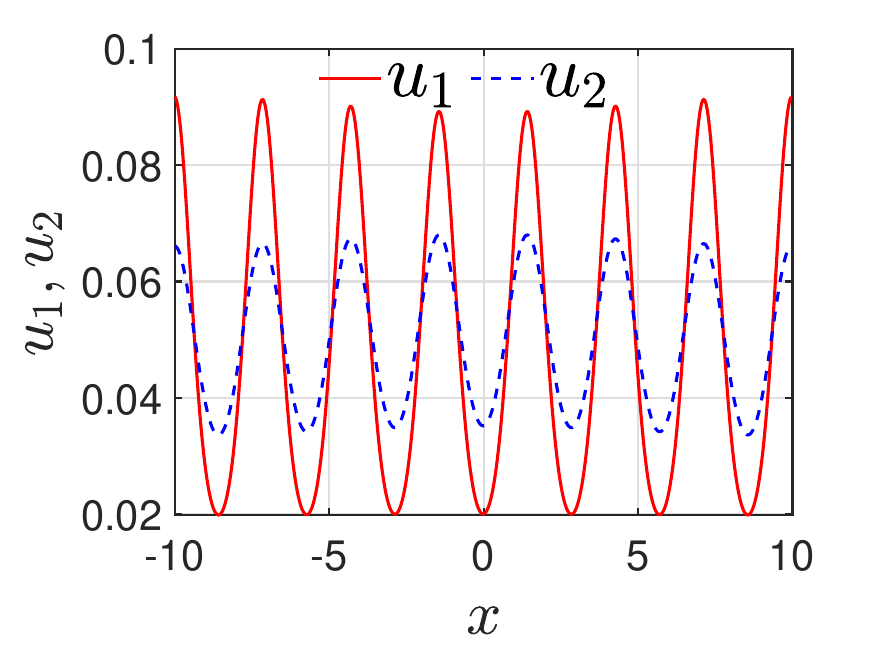}
\includegraphics[width=0.32\linewidth]{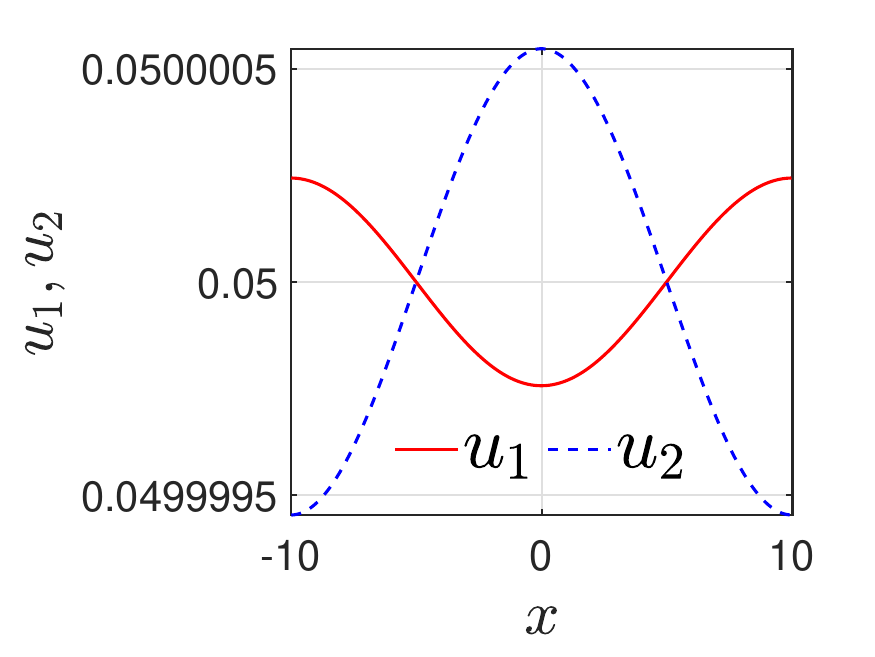}
\caption{Strong repulsive interactions: Solution profiles at $t=300$ using the top-hat kernel with $R=1$, $\kappa=0.25$ (left), $R=2$, $\kappa=0.25$ (middle), and $R=1$, $\kappa=0.5$ (right).}
\label{fig.sosc2}
\end{figure}

Finally, we present some simulations in the two-dimensional domain $\Omega = (0,1)^2$ using the top-hat kernel with $R=1$ and the coefficients \eqref{5.alpha}. We choose $\kappa=0.01$, the initial data
\begin{align*}
  u_1^0 = 0.1\cdot\mathrm{1}_{[1/2,3/4]\times[3/8,5/8]}, \quad
  u_2^0 = 0.1\cdot\mathrm{1}_{[1/4,1/2]\times[3/8,5/8]},
\end{align*}
and $\Delta t=\Delta x = 0.01$. The intersection of the supports of $u_1^0$ and $u_2^0$ equals the line $\{1/2\}\times[3/8,5/8]$, which means that the species are initially almost segregated. The solution profiles at various times are illustrated in Figure \ref{fig.sdens2d}. Again, we observe small-scale oscillations. The Boltzmann and Rao entropies, illustrated in Figure \ref{fig.sent2d}, are not monotonous in this example. The reason is that \eqref{1.posdef} is not satisfied, so we cannot expect entropy decay. 

\begin{figure}[ht]
\includegraphics[width=0.32\linewidth]{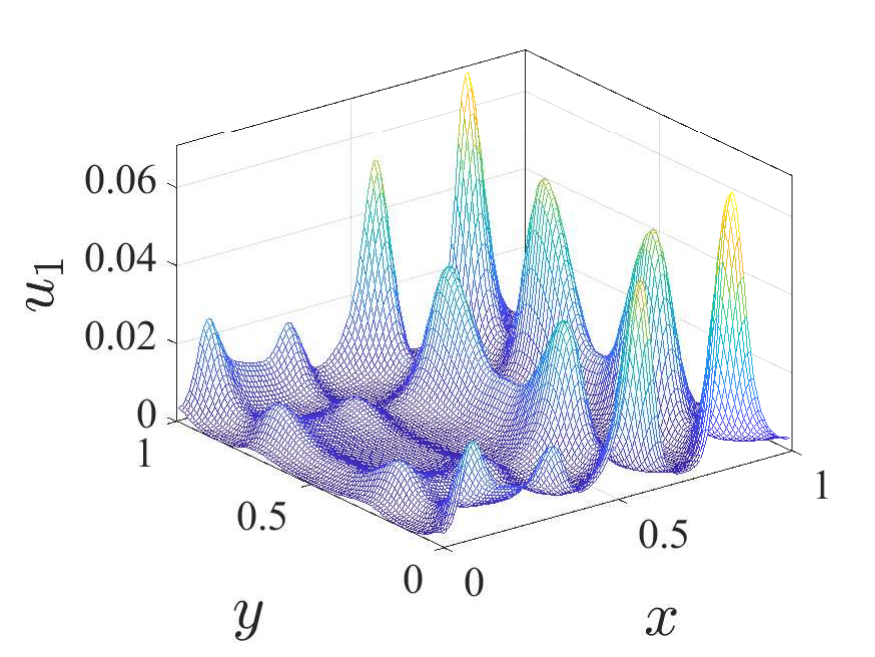}
\includegraphics[width=0.32\linewidth]{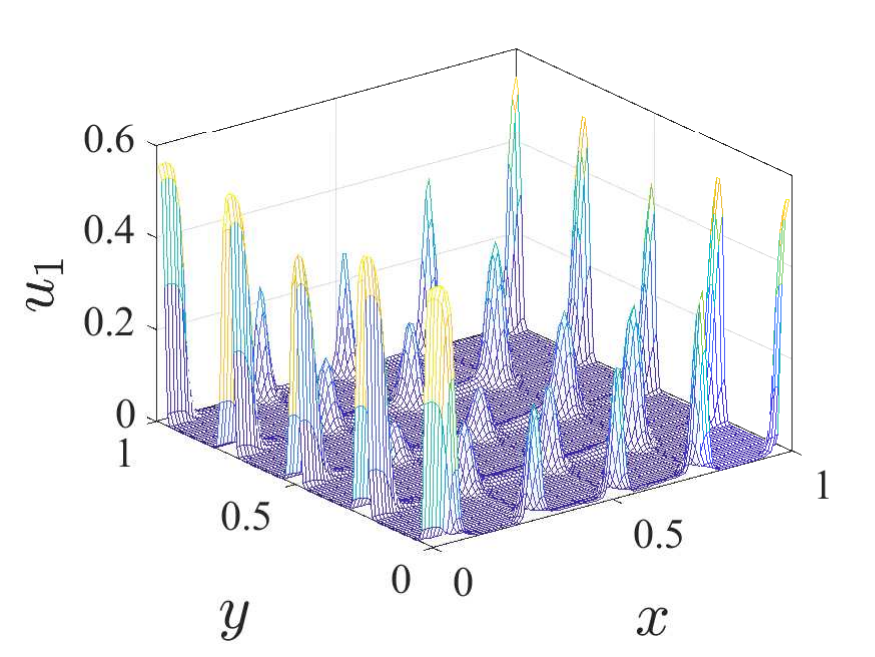}
\includegraphics[width=0.32\linewidth]{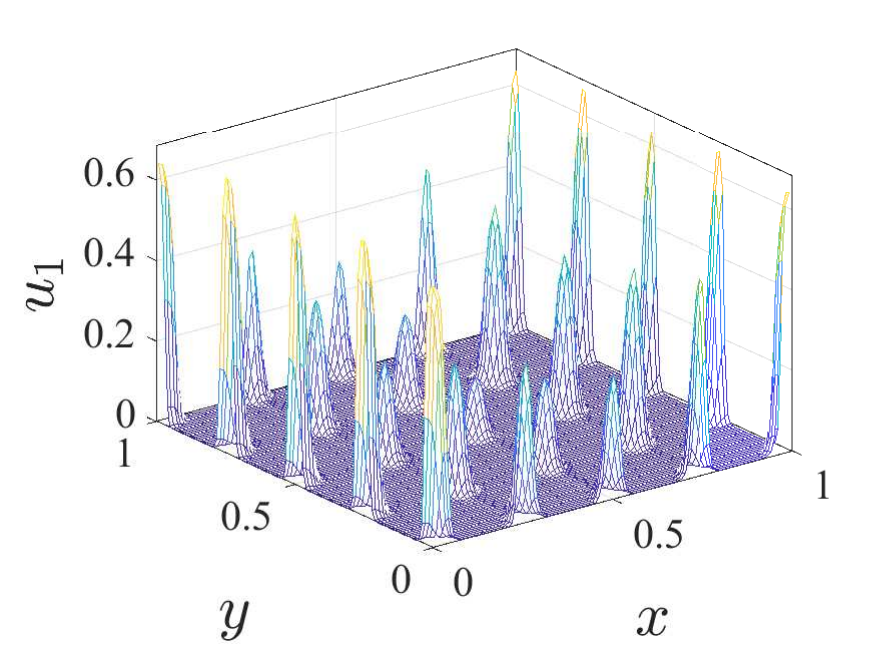}
\includegraphics[width=0.32\linewidth]{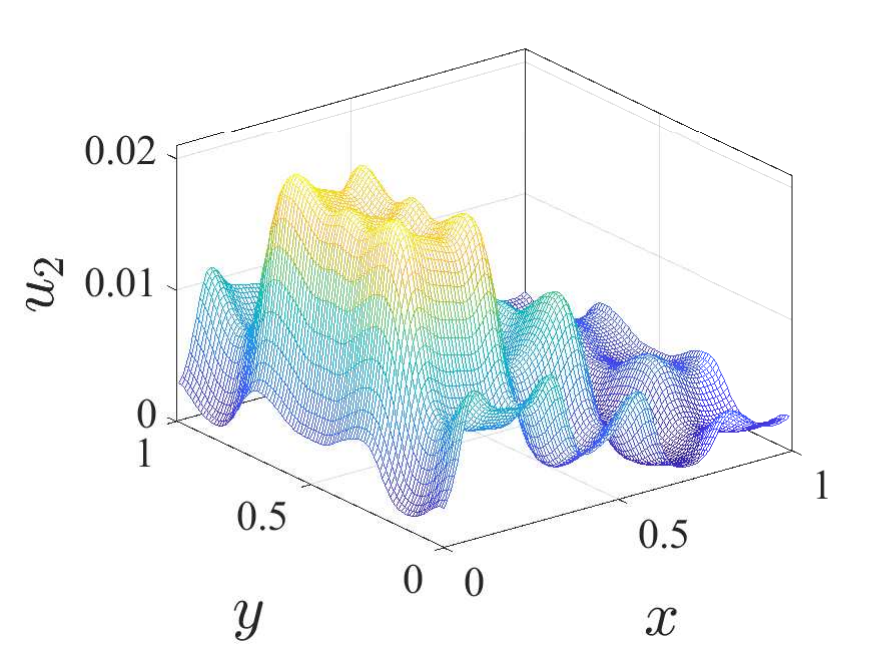}
\includegraphics[width=0.32\linewidth]{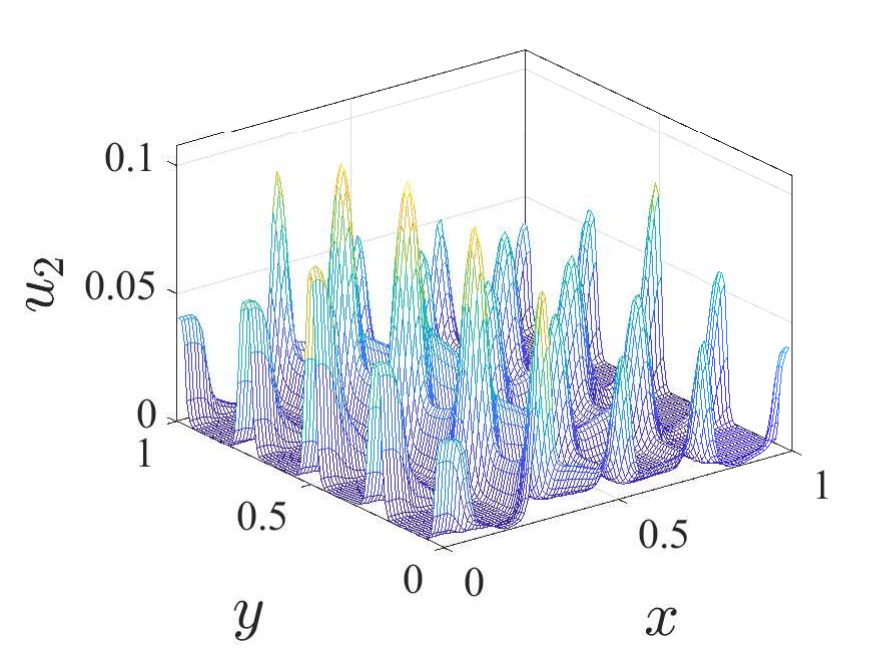}
\includegraphics[width=0.32\linewidth]{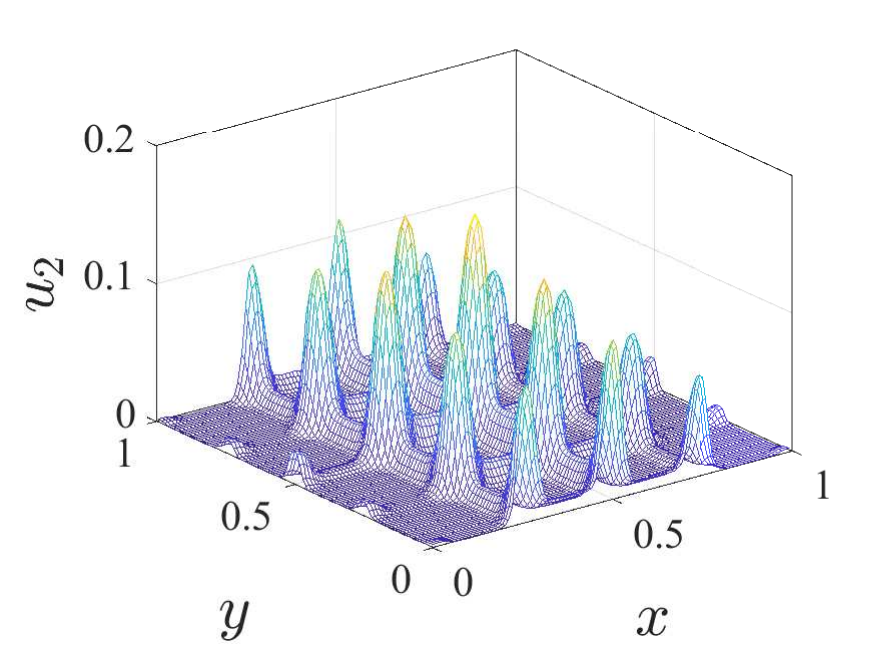}
\caption{Strong repulsive interactions: Solution profiles $u_1$ (top) and $u_2$ (bottom) at times $t = 4$ (left), $t=10$ (middle), and $t=34$ (right).}
\label{fig.sdens2d}
\end{figure} 

\begin{figure}[ht]
\includegraphics[width=0.4\linewidth]{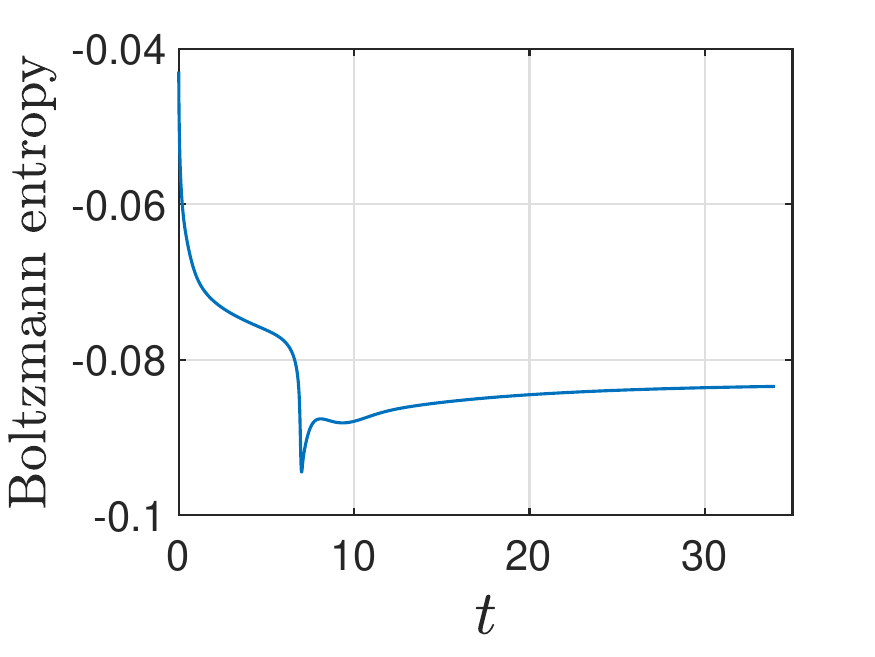}
\includegraphics[width=0.4\linewidth]{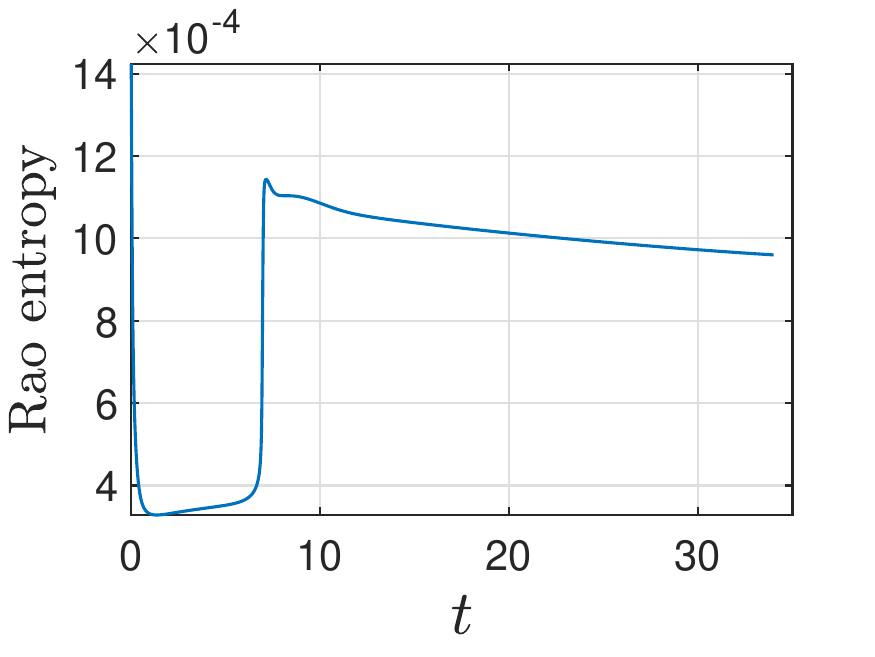}
\caption{Strong repulsive interactions: Evolution of the Boltzmann entropy (left) and Rao entropy (right).}
\label{fig.sent2d}
\end{figure}


\end{document}